\documentclass[reqno]{amsart}
\pdfoutput=1

\usepackage{amssymb,amsmath,amsthm}
\usepackage{color}
\usepackage{fullpage}
\usepackage{url}
\usepackage[Symbol]{upgreek}
\usepackage[unicode,pdfusetitle]{hyperref}
\usepackage[mathscr]{euscript}
\usepackage{tikz-cd}
\usepackage{mathtools}

\usepackage{etoolbox}
\makeatletter
\patchcmd{\@startsection}
  {\@afterindenttrue}
  {\@afterindentfalse}
  {}{}
\makeatother

\usepackage{enumitem}
\setlist[enumerate, 1]{label=(\arabic*),leftmargin=4em}
\AtBeginEnvironment{theorem}{\setlist[enumerate,1]{label=(\roman*),font=\upshape,leftmargin=4em}}
\AtBeginEnvironment{lemma}{\setlist[enumerate,1]{label=(\roman*),font=\upshape,leftmargin=4em}}
\AtBeginEnvironment{proposition}{\setlist[enumerate,1]{label=(\roman*),font=\upshape,leftmargin=4em}}
\AtBeginEnvironment{corollary}{\setlist[enumerate,1]{label=(\roman*),font=\upshape,leftmargin=4em}}

\newtheorem{theorem}{Theorem}
\newtheorem*{theorem*}{Theorem}

\newtheorem*{definition*}{Definition}

\newtheorem*{question*}{Question}

\newtheorem*{conjecture*}{Conjecture}

\newtheorem*{convention*}{Convention}
\newtheorem{assumption}[theorem]{Assumption}
\newtheorem*{assumption*}{Assumption}
\newtheorem*{induction*}{Induction Hypothesis}
\newtheorem{corollary}[theorem]{Corollary}
\newtheorem*{corollary*}{Corollary}
\newtheorem{proposition}[theorem]{Proposition}
\newtheorem*{proposition*}{Proposition}
\newtheorem{lemma}[theorem]{Lemma}
\newtheorem*{lemma*}{Lemma}
\newtheorem{fact}[theorem]{Fact}
\newtheorem*{fact*}{Fact}

\newtheorem*{claim*}{Claim}

\theoremstyle{definition}
\newtheorem{example}[theorem]{Example}
\newtheorem*{example*}{Example}
\newtheorem{remark}[theorem]{Remark}
\newtheorem*{remark*}{Remark}

\numberwithin{theorem}{section}
\numberwithin{claim}{section}
\numberwithin{equation}{section}

\DeclareMathOperator{\ac}{ac}

\DeclareMathOperator{\an}{an}

\DeclareMathOperator{\cl}{cl}
\DeclareMathOperator{\dcl}{dcl}

\DeclareMathOperator{\eq}{eq}

\DeclareMathOperator{\Graph}{Gr}
\DeclareMathOperator{\h}{h}

\DeclareMathOperator{\RCF}{RCF}

\DeclareMathOperator{\res}{res}

\DeclareMathOperator{\rk}{rk}
\DeclareMathOperator{\rv}{rv}
\DeclareMathOperator{\RV}{RV}

\DeclareMathOperator{\TCVF}{TCVF}
\DeclareMathOperator{\Th}{Th}

\newcommand{\N}{\mathbb{N}}
\newcommand{\PP}{\mathbb{P}}
\newcommand{\Q}{\mathbb{Q}}
\newcommand{\R}{\mathbb{R}}
\newcommand{\T}{\mathbb{T}}

\newcommand{\cC}{\mathcal C}

\newcommand{\cE}{\mathcal E}
\newcommand{\cG}{\mathcal G}

\newcommand{\cI}{\mathcal I}
\newcommand{\cK}{\mathcal K}
\newcommand{\cL}{\mathcal L}

\newcommand{\cO}{\mathcal O}

\newcommand{\fm}{\mathfrak{m}}

\newcommand{\0}{\emptyset}

\newcommand{\bg}{\boldsymbol{g}}
\newcommand{\cld}{\cl^\der}
\newcommand{\dclL}{\dcl_\cL}

\newcommand{\f}{\operatorname{f}}

\newcommand{\inv}{^{-1}}

\newcommand{\jet}{\mbox{\small$\mathscr{J}$}_\der}

\newcommand{\oRV}[1]{\RV_{\! #1}}
\newcommand{\llp}{(\!(}
\newcommand{\rrp}{)\!)}
\newcommand{\Ld}{\cL^\der}
\newcommand{\LdO}{\cL^{\cO,\der}}
\newcommand{\LO}{\cL^{\cO}}
\newcommand{\LRV}{\cL^{\RV}}
\newcommand{\LRVeq}{\cL^{\RV^{\eq}}}

\newcommand{\rkL}{\rk_\cL}

\newcommand{\Td}{T^\der}
\newcommand{\TdG}{T^\der_{\cG}}
\newcommand{\TdO}{T^{\cO,\der}}
\newcommand{\TdOmon}{T_{\operatorname{mon}}}
\newcommand{\TdOmonG}{T^{*}_{\operatorname{mon}}}
\newcommand{\TdOmonac}{T^{\ac}_{\operatorname{mon}}}
\newcommand{\TO}{T^{\cO}}
\newcommand{\TRV}{T^{\RV}}
\newcommand{\TRVeq}{T^{\RV^{\eq}}}
\newcommand{\val}{\operatorname{v}}

\newcommand{\llangle}{\langle\!\langle}
\newcommand{\rrangle}{\rangle\!\rangle}

\renewcommand{\preceq}{\preccurlyeq}
\renewcommand{\succeq}{\succcurlyeq}
\renewcommand{\geq}{\geqslant}
\renewcommand{\leq}{\leqslant}
\renewcommand{\epsilon}{\varepsilon}
\renewcommand{\k}{\boldsymbol{k}}

\renewcommand{\d}{\operatorname{d}}
\renewcommand{\r}{\operatorname{r}}
\renewcommand{\bf}{\boldsymbol{f}}

\DeclareFontFamily{OMS}{smallo}{}
\DeclareFontShape{OMS}{smallo}{m}{n}{<->s*[.65]cmsy10}{}
\DeclareSymbolFont{smallo@m}{OMS}{smallo}{m}{n}
\DeclareMathSymbol{\smallo}{\mathord}{smallo@m}{79}

\DeclareFontFamily{U}{fsy}{}
\DeclareFontShape{U}{fsy}{m}{n}{<->s*[.9]psyr}{}
\DeclareSymbolFont{der@m}{U}{fsy}{m}{n}
\DeclareMathSymbol{\der}{\mathord}{der@m}{182}
\DeclareSymbolFont{der@m}{U}{fsy}{m}{n}
\DeclareMathSymbol{\derdelta}{\mathord}{der@m}{100}
\author{Elliot Kaplan}
\author{Nigel Pynn-Coates}
\email{ekaplan@mpim-bonn.mpg.de}
\email{nigel.pynn-coates@univie.ac.at}
\title{Monotone \texorpdfstring{$T$}{T}-convex \texorpdfstring{$T$}{T}-differential fields}
\address{Max Planck Institute for Mathematics, Bonn, Germany}
\address{Kurt G\"{o}del Research Center for Mathematical Logic, Universit\"{a}t Wien, Austria}
\begin{document}

\begin{abstract}
Let $T$ be a complete, model complete o-minimal theory extending the theory of real closed ordered fields and assume that $T$ is power bounded. Let $K$ be a model of $T$ equipped with a $T$-convex valuation ring $\cO$ and a $T$-derivation $\der$ such that $\der$ is monotone, i.e., weakly contractive with respect to the valuation induced by $\cO$. We show that the theory of monotone $T$-convex $T$-differential fields, i.e., the common theory of such $K$, has a model completion, which is complete and distal. Among the axioms of this model completion, we isolate an analogue of henselianity that we call $\Td$-henselianity. We establish an Ax--Kochen/Ershov theorem and further results for monotone $T$-convex $T$-differential fields that are $\Td$-henselian.
\end{abstract}
\maketitle

%------------------------------------------------------------------------------%
\section{Introduction}
%------------------------------------------------------------------------------%
Let $\R_{\an}$ denote the expansion of the real field $\R$ by all globally subanalytic sets. Explicitly, $\R_{\an}$ is the structure obtained by adding a new function symbol for each $n$-ary function $F$ that is real analytic on a neighborhood of $[-1,1]^n$, and by interpreting this function symbol as the restriction of $F$ to $[-1,1]^n$. Let $T_{\an}$ be the elementary theory of $\R_{\an}$ in the language extending the language of ordered rings by the function symbols described above. This theory is model complete and o-minimal~\cite{Ga68,vdD86}. Let $K$ be a model of $T_{\an}$. A map $\der\colon K\to K$ is said to be a \emph{$T_{\an}$-derivation} if it is a field derivation on $K$ that satisfies the identity
\[
\der F(u)\ =\ \frac{\partial F}{\partial Y_1}(u)\der u_1+\cdots + \frac{\partial F}{\partial Y_n}(u)\der u_n
\]
for each restricted analytic function $F$ and each $u = (u_1,\ldots,u_n) \in K^n$ with $|u_i|< 1$ for each $i$. The only $T_{\an}$-derivation on $\R_{\an}$ itself is the trivial derivation, which takes constant value zero. However, there are a number of interesting examples of nonstandard models of $T_{\an}$ with nontrivial $T_{\an}$-derivations.
\begin{enumerate}
\item\label{introex1} Consider $\R(\!(t^{1/\infty})\!)\coloneqq \bigcup_n\R(\!(t^{1/n})\!)$, the field of Puiseux series over $\R$. We totally order $\R(\!(t^{1/\infty})\!)$ by taking $t$ to be a positive infinitesimal element. Then $\R(\!(t^{1/\infty})\!)$ admits an expansion to a model of $T_{\an}$, by extending each restricted analytic function on $\R$ to the corresponding box in $\R(\!(t^{1/\infty})\!)$ via Taylor expansion. We put $x \coloneqq 1/t$, and we let $\der\colon \R(\!(t^{1/\infty})\!)\to \R(\!(t^{1/\infty})\!)$ be the derivation with respect to $x$, so
\[
\der \sum_{q\in \Q}r_qx^q\ \coloneqq\ \sum_{q\in \Q}r_qqx^{q-1}.
\]
As $\der$ commutes with infinite sums, it is routine to verify that $\der$ is indeed a $T_{\an}$-derivation. 
\item\label{introex2} The field $\T$ of logarithmic-exponential transseries also admits a canonical expansion to a model of $T_{\an}$ using Taylor expansion~\cite[Corollary 2.8]{DMM97}. The usual derivation on $\T$ is a $T_{\an}$-derivation.
\item\label{introex3} Let $\k \models T_{\an}$ and let $\der_{\k}$ be a $T_{\an}$-derivation on $\k$. Let $\Gamma$ be a divisible ordered abelian group, and consider the Hahn field $\k\llp t^\Gamma \rrp$, ordered so that $0<t<\k^>$, where $\k^> = \{ a \in \k : a>0 \}$. We expand $\k\llp t^\Gamma \rrp$ to a model of $T_{\an}$ using Taylor expansion (see \cite[Proposition 2.13]{Ka21} for details). Let $c\colon \Gamma\to \k$ be an additive map. We use $c$ to define a derivation $\der$ on $\k \llp t^\Gamma \rrp$ as follows:
\[
\der\sum_\gamma f_\gamma t^\gamma\ \coloneqq\ \sum_\gamma\big(\der_{\k}f_\gamma + f_\gamma c(\gamma)\big)t^\gamma.
\]
The map $\der$ is the unique derivation that extends $\der_{\k}$, commutes with infinite sums, and satisfies the identity $\der t^\gamma = c(\gamma)t^\gamma$ for each $\gamma\in \Gamma$. By~\cite[Proposition~3.14]{Ka21}, this derivation is even a $T_{\an}$-derivation. We denote this expansion of $\k\llp t^\Gamma\rrp$ by $\k\llp t^\Gamma\rrp_{\an,c}$.
\end{enumerate}
In each of the above examples, there is a natural convex valuation ring $\cO$ (the convex hull of $\R$ in the first two examples, the convex hull of $\k$ in the third).  These convex valuation rings are even \emph{$T_{\an}$-convex}, as defined by van den Dries and Lewenberg~\cite{DL95}, and the derivation in each example is continuous. 

In this paper, we work not just with the theory $T_{\an}$, but with any complete, model complete, power bounded o-minimal $\cL$-theory $T$ (where \emph{power boundedness} is the assumption that every definable function is eventually dominated by a power function $x \mapsto x^\lambda$). Let $K$ be a model of $T$. A \emph{$T$-derivation} is a map $\der\colon K\to K$ that satisfies the identity 
\[
\der F(u)\ =\ \frac{\partial F}{\partial Y_1}(u)\der u_1+\cdots + \frac{\partial F}{\partial Y_n}(u)\der u_n
\]
for all $\cL(\emptyset)$-definable functions $F$ which are $\cC^1$ at $u$, and a \emph{$T$-convex valuation ring} is a nonempty convex subset $\cO\subseteq K$ that is closed under all $\cL(\emptyset)$-definable continuous functions. A \textbf{$T$-convex $T$-differential field} is the expansion of $K$ by a $T$-convex valuation ring and a continuous $T$-derivation.

In examples \ref{introex1} and \ref{introex3} above, the derivation $\der$ is \textbf{monotone}:\ the logarithmic derivative $\der a/a$ is in the valuation ring $\cO$ for each nonzero $a$. The derivation in example \ref{introex2} is not monotone. In this paper, our focus is primarily on monotone $T$-convex $T$-differential fields, and in this setting, our assumption that $T$ is power bounded comes almost for free; see Remark~\ref{rem:powerbounded}. We prove the following:
\begin{theorem*}
The theory of monotone $T$-convex $T$-differential fields has a model completion. This model completion is complete and distal (in particular, it has NIP).
\end{theorem*}
Our model completion is quite similar to Scanlon's model completion for the theory of monotone valued differential fields~\cite{Sc00}. In the case $T = T_{\an}$, a model of this model completion can be constructed as follows:\ consider the $T_{\an}$-convex $T_{\an}$-differential field $\k\llp t^\Gamma\rrp_{\an,c}$ in example \ref{introex3}, where the derivation $\der_{\k}$ on $\k$ is \emph{generic}, as defined in~\cite{FK21}, and where $c$ is taken to be the zero map. 

In axiomatizing this model completion, we introduce an analogue of differential-henselianity for $T$-convex $T$-differential fields, which we call \emph{$\Td$-henselianity}.
For a $T$-convex $T$-differential field $K$, this definition has three parts, which informally are:
\begin{enumerate}
    \item $K$ has small derivation, i.e., $\der$ maps the maximal ideal $\smallo$ of $\cO$ into itself;
    \item $K$ has linearly surjective differential residue field, i.e., for any $a_0,\ldots,a_r \in \cO$, not all in $\smallo$, there is $y \in K$ such that $a_0y + a_1\der y + \ldots+a_r\der^r y-1\in \smallo$;
    \item and that  for every $\cL(K)$-definable function $F\colon K^r\to K$, the function $y\mapsto y^{(r)} - F(y,\ldots,y^{(r-1)})$ has a zero in $\smallo$ whenever it is well-approximated by a linear differential operator on $\smallo$ and $F(0)$ is sufficiently small relative to this approximation.
     % \item and that  for every $\cL(K)$-definable function $F\colon K^r\to K$, if $a\in K$ is an approximate zero of the function $y\mapsto y^{(r)} - F(y,\ldots,y^{(r-1)})$, and if this function is well-approximated by a linear differential operator on a neighborhood of $a$, then it has an actual zero in this neighborhood.
\end{enumerate}
% This condition states that $\der$ maps the maximal ideal $\smallo$ of $\cO$ into itself (small derivation); that for any $a_0,\ldots,a_r \in \cO$, not all in $\smallo$, there is $y \in K$ such that $a_0y + a_1\der y + \ldots+a_r\der^r y-1\in \smallo$ (linearly surjective differential residue field); and that  for every $\cL(K)$-definable function $F\colon K^r\to K$, if $a\in K$ is an approximate zero of the function $y\mapsto y^{(r)} - F(y,\ldots,y^{(r-1)})$, and if this function is well-approximated by a linear differential operator on a neighborhood of $a$, then it has an actual zero in this neighborhood.
This last condition is inspired by Rideau-Kikuchi's definition of \emph{$\sigma$-henselianity} for  analytic difference valued fields~\cite{Ri17}. We prove an Ax--Kochen/Ershov theorem for monotone $\Td$-henselian fields, from which we derive the following:
\begin{theorem*}
Any monotone $\Td_{\an}$-henselian field $(K,\cO,\der)$ is elementarily equivalent to some $\k\llp t^\Gamma\rrp_{\an,c}$.
\end{theorem*}
Hakobyan~\cite{Ha18} previously established a similar statement for monotone differential-henselian fields.

Let $K$ be a $T$-convex $T$-differential field with small derivation and linearly surjective differential residue field. With an eye toward future applications, we develop the theory of $\Td$-henselian fields as much as possible without the assumption that $\der$ is monotone. We show that if $K$ is spherically complete, then $K$ is $\Td$-henselian, and that any $\Td$-henselian $K$ admits a \emph{lift of the differential residue field}:\ an elementary $\cL$-substructure $\k\preceq_{\cL} K$ that is closed under $\der$ and maps isomorphically onto the differential residue field $\res(K)$. An essential ingredient in this proof is the fact,  due to Garc\'{i}a~Ram\'{i}rez, that $\cL$-definable functions enjoy the \emph{Jacobian property}~\cite{Ga20}. We also show that if each immediate extension of $K$ has a property we call the \emph{$\Td$-hensel configuration property}, then $K$ has a unique spherically complete immediate $T$-convex $T$-differential field extension with small derivation. The existence of such a spherical completion was previously established by the first author~\cite[Corollary~6.4]{Ka22}, and the $\Td$-hensel configuration property is similar to the differential-henselian configuration property of van den Dries and the second author~\cite{DPC19}. Again making use of the Jacobian property, we show that monotone fields enjoy the $\Td$-hensel configuration property, thereby establishing the following:
\begin{theorem*}
Let $K$ be a monotone $T$-convex $T$-differential field with linearly surjective differential residue field. Then $K$ has a unique spherically complete immediate monotone $T$-convex $T$-differential field extension, up to isomorphism over~$K$.
\end{theorem*}
This uniqueness plays an essential role in our Ax--Kochen/Ershov theorem. 

One may be able to adapt our definition of $\Td$-henselianity to study compatible derivations on other tame expansions of equicharacteristic zero valued fields. One broad class where our methods may generalize is the class of \emph{Hensel minimal} expansions of valued fields, introduced by Cluckers,  Halupczok, and Rideau-Kikuchi~\cite{CHRK22}. The setting of Hensel minimal fields (1-h-minimal fields, to be precise) includes our present o-minimal setting, as well as the fields with analytic structure studied by Cluckers and Lipshitz~\cite{CL11}. As an indicator that generalizing to this class may be possible, we note that the Jacobian property, which proves so useful in this paper, holds in the setting of 1-h-minimal fields~\cite[Theorem 5.4.10]{CHRK22}.
%------------------------------------------------------------------------------%
\subsection{Organization of the paper}
%------------------------------------------------------------------------------%
Section~\ref{sec:background} contains background information on $T$-convex valuation rings and $T$-convex differential fields and results that we need, drawn from \cite{DL95,vdD97,Yi17,Ga20,FK21,ADH17,Ka22}.
In Section~\ref{sec:Tdh} we introduce $\Td$-henselianity and record basic consequences in Section~\ref{sec:Tdhbasic}.
In Section~\ref{sec:lift}, we show that $\Td$-henselianity yields a lift of the differential residue field (Theorem~\ref{thm:reslift}).
Section~\ref{sec:sphcompTdh} relates $\Td$-henselianity to immediate extensions and shows that every spherically complete $T$-convex $T$-differential field with small derivation and linearly surjective differential residue field is $\Td$-henselian (Corollary~\ref{cor:sphcompTdh}).
The next subsection develops further technical material for use in Section~\ref{sec:sphcompunique}, which establishes the uniqueness of spherically complete immediate extensions and related results, conditional on the $\Td$-hensel configuration property introduced there.
Section~\ref{sec:monotone} focuses on monotone fields.
In Section~\ref{sec:monotoneTdhc}, we show that monotone $T$-convex $T$-differential fields have the $\Td$-hensel configuration property and summarize the consequences.
The next subsection records variants of the results in Section~\ref{sec:sphcompunique}, of which only Corollary~\ref{cor:sphcompembed} is needed later, in the proof of the Ax--Kochen/Ershov theorem.
The short Section~\ref{sec:RCF} explains what $\Td$-henselianity means when $T$ is the theory of real closed ordered fields. In this case, we show that $\Td$-henselianity and differential-henselianity coincide for monotone fields.
Section~\ref{sec:AKE} contains the Ax--Kochen/Ershov theorem and related three-sorted results, while Section~\ref{sec:modelcompletion} contains the model completion result and related one-sorted results.

%------------------------------------------------------------------------------%
\section{Background}\label{sec:background}
%------------------------------------------------------------------------------%
%------------------------------------------------------------------------------%
\subsection{Notation and conventions}
%------------------------------------------------------------------------------%
Throughout, $m,n,q,r$ range over $\N \coloneqq \{0,1,2,\dots\}$. For a unital ring $R$, we let $R^\times$ denote the multiplicative group of units of $R$, and for an ordered set $S$ equipped with a distinguished element $0$, we set $S^> \coloneqq \{ s \in S : s>0 \}$.

In this paper, we fix a complete, model complete o-minimal theory $T$ extending the theory of real closed ordered fields in an appropriate language $\cL\supseteq\{0,1,+,\cdot,<\}$. Throughout, $K$ is a model of $T$. Given a subset $A \subseteq K$, we let $\dclL(A)$ denote the $\cL$-definable closure of $A$. Since $T$ has definable Skolem functions, $\dclL(A)$ is (the underlying set of) an elementary substructure of $K$. It is well-known that $\dclL$ is a pregeometry, and we denote the corresponding rank by $\rkL$. Let $M$ be a $T$-extension of $K$ (that is, a model of $T$ that extends $K$) and let $A\subseteq M$. We denote by $K\langle A\rangle$ the intermediate extension of $K$ with underlying set $\dclL(K\cup A)$. We say that $A$ is \textbf{$\cL(K)$-independent} if $A$ is independent over $K$ with respect to the pregeometry $\dclL$. Otherwise $A$ is said to be \textbf{$\cL(K)$-dependent}. If $A$ is $\cL(K)$-independent and $M = K\langle A\rangle$, then $A$ is said to be a \textbf{$\dclL$-basis for $M$ over $K$}. Given $a = (a_1,\ldots,a_n) \in M^n$, we write $K\langle a \rangle$ in place of $K\langle \{a_1,\ldots,a_n\}\rangle$, and we say that $a$ is \textbf{$\cL(K)$-independent} if the set $\{a_1,\ldots,a_n\}$ is $\cL(K)$-independent and no components are repeated. Equivalently, $a$ is  \textbf{$\cL(K)$-independent} if $\rkL(a|K) = n$.
For tuples $a,b \in K^n$, we let $a\cdot b\coloneqq a_1b_1+\cdots+a_nb_n$ denote the inner product of $a$ and~$b$.

A \textbf{power function} is an $\cL(K)$-definable endomorphism of the ordered multiplicative group $K^>$. Each power function $f$ is uniquely determined by its derivative at $1$, and if $f'(1) = \lambda$, then we suggestively write $a^\lambda$ instead of $f(a)$ for $a \in K^>$.  The collection $\Lambda \coloneqq \{f'(1): f\text{ is a power function}\}$ forms a subfield of $K$, called the \textbf{field of exponents of $K$}. If every  $\cL(K)$-definable function is eventually bounded by a power function, then $K$ is said to be \textbf{power bounded}. Suppose $K$ is power bounded. Then by Miller's dichotomy~\cite{Mi96}, every model of $T$ is power bounded and every power function is $\cL(\emptyset)$-definable, so $\Lambda$ does not depend on $K$. We just say that \text{$T$ is power bounded}, and we call $\Lambda$ the \textbf{field of exponents of $T$}. For the remainder of the paper, we assume that $T$ is power bounded with field of exponents~$\Lambda$.

%------------------------------------------------------------------------------%
\subsection{Background on \texorpdfstring{$T$}{T}-convex valuation rings}
%------------------------------------------------------------------------------%

In this subsection, let $\cO$ be a \textbf{$T$-convex valuation ring} of $K$; that is, $\cO \subseteq K$ is convex and nonempty and $F(\cO)\subseteq \cO$ for every $\cL(\0)$-definable continuous $F \colon K \to K$.
We call $(K, \cO)$ a \textbf{$T$-convex valued field}.
These structures were introduced and studied by van den Dries and Lewenberg \cite{DL95} and additionally by van den Dries \cite{vdD97}.
We briefly review notation, following \cite[Section 1]{Ka22}, and general facts we use; additional relevant facts on $T$-convex valuation rings can be found there, or in the original papers.
For valuation-theoretic notation we follow \cite[Chapter~3]{ADH17}.

Let $\LO \coloneqq \cL \cup \{\cO\}$ be the extension of $\cL$ by a unary predicate $\cO$ and $\TO$ be the theory extending $T$ by axioms stating that $\cO$ is a \emph{proper} $T$-convex valuation ring.
This $\TO$ is complete and model complete, since $T$ is \cite[Corollary~3.13]{DL95}; if $T$ has quantifier elimination and a universal axiomatization, then $\TO$ has quantifier elimination \cite[Theorem~3.10]{DL95}.
Let $\Gamma$ be the value group of the valuation $v \colon K^{\times} \to \Gamma$ induced by $\cO$, which moreover is an ordered $\Lambda$-vector space with scalar multiplication $\lambda va \coloneqq v(a^{\lambda})$ for $a \in K^>$ and $\lambda \in \Lambda$.
We extend $v$ to $v \colon K \to \Gamma \cup \{\infty\}$ by $v(0)\coloneqq \infty$, where $\infty$ is a symbol not in $\Gamma$, and extend the ordering and addition of $\Gamma$ to $\Gamma \cup \{\infty\}$ in the natural way.
We also extend $v$ to $K^n$ by $va \coloneqq \min\{va_1,\ldots,va_n\}$ for $a \in K^n$.
For $a,b \in K$, we set:
\[\begin{array}{lc}
a \preceq b\ \Leftrightarrow\ va\geq vb,\qquad a \prec b\ \Leftrightarrow\ va>vb,\\
a \asymp b\ \Leftrightarrow\ va=vb,\qquad  
a\sim b\ \Leftrightarrow\ a-b \prec b.
\end{array}\]
The relations $\asymp$ and $\sim$ are equivalence relations on $K$ and $K^{\times}$, respectively.
If $a \sim b$, then $a \asymp b$ and also $a>0$ if and only if $b>0$.
The unique maximal ideal of $\cO$ is $\smallo$ and $\res(K) \coloneqq \cO/\smallo$ is the \textbf{residue field} of $K$.
We usually denote $\res(K)$ by $\k$ (unlike in \cite{Ka22}) and let $\bar{a}$ or $\res(a)$ denote the image of $a \in \cO$ under the residue map to $\k$.
In fact, $\k$ can be expanded to a model of $T$ \cite[Remark~2.16]{DL95}, and we always construe $\k$ this way.
Related is the fact that the convex hull of an elementary $\cL$-substructure of $K$ is a $T$-convex valuation ring of $K$; cf.\ Section~\ref{sec:lift} on lifting the residue field.
If we need to indicate the dependence of these objects on $K$, we do so using a subscript, as in $\Gamma_K$ and $\k_K$. Given a $\TO$-extension $M$ of $K$, we identify $\k$ with a $T$-submodel of $\k_M$ and $\Gamma$ with an ordered $\Lambda$-subspace of $\Gamma_M$ in the natural way. If $\Gamma_M = \Gamma$ and $\k_M = \k$, then the $\TO$-extension $M$ is said to be \textbf{immediate}. As a consequence of our power boundedness assumption, we have the following analogue of the Abhyankar--Zariski inequality, established by van den Dries and referred to in the literature as the \textbf{Wilkie inequality}.

\begin{fact}[{\cite[Section~5]{vdD97}}]
\label{fact:wilkieineq}
Let $M$ be a $\TO$-extension of $K$ and suppose $\rkL(M|K)$ is finite. Then
\[
\rkL(M|K)\ \geq\ \rkL(\k_M|\k)+\dim_\Lambda(\Gamma_M/\Gamma).
\]
\end{fact}

An \textbf{open $v$-ball} is a set of the form $B(a, \gamma) \coloneqq \{ b \in K : v(b-a) > \gamma \}$ for $a \in K$ and $\gamma \in \Gamma$.
These open $v$-balls form a basis for the \emph{valuation topology} on $K$, which coincides with its order topology since $\cO$ is assumed to be a proper subring.
Similarly, a \textbf{closed $v$-ball} is a set of the form $\{ b \in K : v(b-a) \geq \gamma \}$; in this latter definition we allow $\gamma \in \Gamma \cup \{\infty\}$ so that singletons are closed $v$-balls.
Later we only use open $v$-balls, but closed $v$-balls are needed for the next definitions.
A collection of closed $v$-balls is \textbf{nested} if any two balls in the collection have nonempty intersection (recall that any two $v$-balls are either disjoint or one is contained in the other).
We call $K$ \textbf{spherically complete} if every nonempty nested collection of closed $v$-balls has nonempty intersection.
For the relationship between spherical completeness and immediate extensions, see the next subsection and the beginning of Section~\ref{sec:sphcompTdh}.

In places, we expand $K$ by two quotients.
We denote the quotient $K^\times/(1+\smallo)$ by $\oRV{K}^\times$, with corresponding quotient map $\rv\colon K^\times \to \oRV{K}^\times$. We set $\oRV{K}\coloneqq \oRV{K}^\times \cup \{0\}$, and we extend $\rv$ to all of $K$ by setting $\rv(0)\coloneqq 0$. The residue map $\cO^\times \to \k^\times$ induces a bijection $\rv(\cO^\times) \to \k^\times$, which we also call $\res$, by $\res \rv(a) = \bar{a}$ for $a \in \cO^\times$; conversely, in the same way we can recover $\res\colon \cO^{\times} \to \k^{\times}$ from $\res \colon \rv(\cO^\times) \to \k^\times$. We extend $\res$ to a map $\rv(\cO^\times)\cup \{0\}\to \k$ by setting $\res(0)\coloneqq 0$, and we view $\res$ as a partial map from $\oRV{K}$ to $\k$. Let $\LRV$ be the language extending $\LO$ by a sort for $\oRV{K}$ (in the language $(\cdot, \mbox{}^{-1}, 1, 0, <)$),
% (as a multiplicative group with an extra element 0 and a linear ordering induced by $K$),
a sort for $\k$ (in the language $\cL$), and the maps $\rv$ and $\res$.
Let $\TRV$ be the following $\LRV$-theory.
\begin{enumerate}
    \item $(K, \cO) \models \TO$;
    \item $\oRV{K}^{\times}$ is an abelian (multiplicative) group in the language $(\cdot,\mbox{}^{-1}, 1)$;
    \item $\rv \colon K^{\times} \to \oRV{K}^{\times}$ is a surjective group homomorphism with $\ker\rv = 1+\smallo$, extended to $K$ by $\rv(0)=0$;
    \item $<$ is interpreted in $\oRV{K}$ by $\rv(a)<\rv(b)$ if $a<b$ and $a \not\sim b$, for $a, b \in K$;
    % \item $\res\colon \cO \to \k$ is a surjective ring homomorphism extended to $K$ by $\res(K\setminus\cO)=\{0\}$;
    \item $\res\colon \rv(\cO^{\times}) \to \k^{\times}$ is a group isomorphism extended to $\oRV{K}$ by $\res(\oRV{K}\setminus\rv(\cO^{\times}))=\{0\}$;
    \item\label{TRV-k} $\k \models T$ and if $F \colon K^n \to K$ is an $\cL(\0)$-definable continuous function (with the corresponding function $\k^n \to \k$ also denoted by $F$), then $\res(F(a))=F(\res(a))$ for all $a \in \cO^n$.
\end{enumerate}
We denote such a structure simply by $(K, \oRV{K})$.
% Let $\TRV$ be the $\LRV$-theory that extends $\TO$ essentially by definitions.
Yin~\cite{Yi17} introduced the language $\LRV$ and the theory $\TRV$ (he denoted them by $\cL_{T\!\RV}$ and $\TCVF$, respectively), although for Yin, the residue field $\k$ is a predicate on the sort for $\oRV{K}$, not a sort in its own right, and $\cO$ is not in the language $\LRV$ (but it is definable).
% Note that the two theories are naturally bi-interpretable.

% \begin{remark} Here are some initial observations about a model $(K, \oRV{K})$ of $\TRV$.
%     \begin{itemize}
%         \item The relation $<$ in $\oRV{K}$ is a linear order induced by $K$ via $\rv$, which makes $\oRV{K}^> \coloneqq \{ d \in \oRV{K} : d>0 \}$ an ordered abelian group.
%         \item We have $\ker\res = \smallo$, so $\res$ induces an isomorphism $\res(K) \to \k$ of (ordered) fields.
%         \item Moreover, it follows from \ref{TRV-k} that the map $\res(K) \to \k$ is an isomorphism of $\cL$-structures by \cite[Lemma~1.13]{DL95} and \cite[Remark~2.3]{Yi17} (the latter is a syntactic manoeuvre replacing all primitives of $\cL$ except $<$ by function symbols interpreted as continuous functions).
%     \end{itemize}
% \end{remark}
Here are some initial observations about a model $(K, \oRV{K})$ of $\TRV$.
The relation $<$ in $\oRV{K}$ is a linear order induced by $K$ via $\rv$, which makes $\oRV{K}^> \coloneqq \{ d \in \oRV{K} : d>0 \}$ an ordered abelian group.
The map $\res \colon \cO \to \k$ has kernel $\smallo$, so it induces an isomorphism $\res(K) \to \k$ of (ordered) fields.
Moreover, it follows from \ref{TRV-k} that the map $\res(K) \to \k$ is an isomorphism of $\cL$-structures by \cite[Lemma~1.13]{DL95} and \cite[Remark~2.3]{Yi17} (the latter is a syntactic manoeuvre replacing all primitives of $\cL$ except $<$ by function symbols interpreted as continuous functions).

Observe that, again using \cite[Remark~2.3]{Yi17}, $\oRV{K}$ and $\k$ are interpretable in $(K, \cO)$, so $(K, \oRV{K})$ is a reduct of $(K, \cO)^{\eq}$, the expansion of $(K, \cO)$ by all imaginaries; more precisely, $(K, \oRV{K})$ is an expansion by definitions of such a reduct.
It follows on general model-theoretic grounds that $(K, \oRV{K})$ is the unique expansion of $(K, \cO)$ to a model of $\TRV$ and that if $(M,\cO_M)$ is an elementary $\TO$-extension of $(K,\cO)$, then $(M,\oRV{M})$ is an elementary $\TRV$-extension of $(K,\oRV{K})$; these facts were first observed by Yin in \cite[Proposition~2.13]{Yi17} and \cite[Corollary~2.17]{Yi17}, respectively.
% Every model of $\TO$ admits a unique expansion to a model of $\TRV$~\cite[Proposition 2.13]{Yi17}.
% \begin{fact}[{\cite[Corollary~2.17]{Yi17}}]\label{fact:RVexpiselem}
% If $(M,\cO_M)$ is an elementary $\TO$-extension of $(K,\cO)$, then $(M,\oRV{M})$ is an elementary $\TRV$-extension of $(K,\oRV{K})$.
% \end{fact}

Next we consider the further extension of $\LRV$ by all the imaginary sorts coming from $\oRV{K}$. The corresponding language is denoted by $\LRVeq$ and the corresponding theory is denoted by $\TRVeq$. We let $\oRV{K}^{\eq}$ be the structure consisting of the sort $\oRV{K}$, the sort $\k$, and all of these imaginary sorts.
As before, $(K,\cO)$ admits a unique expansion to a model $(K,\oRV{K}^{\eq})$ of $\TRVeq$, and if $(M,\cO_M)$ is an elementary $\TO$-extension of $(K,\cO)$, then $(M,\oRV{M}^{\eq})$ is an elementary $\TRVeq$-extension of $(K,\oRV{K}^{\eq})$.
% Any model $(K,\oRV{K})$ of $\TRV$ admits a unique expansion to a model $(K,\oRV{K}^{\eq})$ of $\TRVeq$. By Fact~\ref{fact:RVexpiselem}, if $(M,\cO_M)$ is an elementary $\TO$-extension of $(K,\cO)$, then $(M,\oRV{M}^{\eq})$ is an elementary $\TRVeq$-extension of $(K,\oRV{K}^{\eq})$. Also, if $(M,\cO_M)$ is an \emph{immediate} $\TO$-extension of $(K,\cO)$, then $\oRV{M}^{\eq} = \oRV{K}^{\eq}$.

Also note that if $(M,\cO_M)$ is an \emph{immediate} $\TO$-extension of $(K,\cO)$, then $\oRV{M}^{\eq} = \oRV{K}^{\eq}$.
In addition, any $\LRVeq(K \cup \oRV{K}^{\eq})$-definable subset of $K^n$ is already $\LO(K)$-definable.
We use these facts in combination with the following key result, the \emph{Jacobian property} for definable functions in models of~$\TO$:

\begin{fact}[{\cite[Theorem~3.18]{Ga20}}]\label{fact:rvapprox}
Let $A \subseteq K\cup \oRV{K}^{\eq}$ and let $F\colon K^n \to K$ be an $\LRVeq(A)$-definable function. Then there is an $\LRVeq(A)$-definable map $\chi\colon K^n \to \oRV{K}^{\eq}$ such that for each $s \in \chi(K^n)$, if $\chi\inv(s)$ contains an open $v$-ball, then either $F$ is constant on $\chi\inv(s)$ or there is $d \in K^n$ such that
\[
v\big(F(x)- F(y) - d\cdot (x-y)\big) \ >\ vd+ v(x-y)
\]
for all $x, y \in \chi\inv(s)$ with $x \neq y$.
\end{fact}

%------------------------------------------------------------------------------%
\subsection{Background on \texorpdfstring{$T$}{T}-derivations}\label{sec:Td}
%------------------------------------------------------------------------------%
Let $\der\colon K \to K$ be a map. For $a \in K$, we use $a'$ or $\der a$ in place of $\der(a)$ and if $a\neq 0$, then we set $a^\dagger\coloneqq a'/a$. Given $r \in \N$, we write $a^{(r)}$ in place of $\der^r(a)$, and we let $\jet^r(a)$ denote the tuple $(a,a',\ldots,a^{(r)})$. We use $\jet^\infty(a)$ for the infinite tuple $(a,a',a'',\ldots)$. 

Let $F\colon U \to K$ be an $\cL(\emptyset)$-definable $\cC^1$-function with $U \subseteq K^n$ open.  We say that $\der$ is \textbf{compatible with $F$} if
\[
F(u)'\ =\ \frac{\partial F}{\partial Y_1}(u)u_1'+\cdots+\frac{\partial F}{\partial Y_n}(u)u_n'
\]
for each $u = (u_1,\ldots,u_n)\in U$. We say that $\der$ is a \textbf{$T$-derivation on $K$} if $\der$ is compatible with every $\cL(\emptyset)$-definable $\cC^1$-function with open domain. Let $\Td$ be the $\Ld \coloneqq \cL\cup\{\der\}$-theory that extends $T$ by axioms stating that $\der$ is a $T$-derivation.

The study of $T$-derivations was initiated in~\cite{FK21}. Any $T$-derivation on $K$ is a \emph{derivation} on $K$, that is, a map satisfying the identities $(a+b)' = a' +b'$ and $(ab)' = a'b+ab'$ for $a,b \in K$ (this follows from compatibility with the functions $(x,y) \mapsto x+y$ and $(x,y) \mapsto xy$). For the remainder of this subsection, $\der$ is a $T$-derivation on $K$. We also call $(K, \der)$ a \textbf{$T$-differential field}. We let $C_K\coloneqq \ker(\der)$ denote the \textbf{constant field} of $K$; if the dependence on $K$ is clear, we just write $C$. It is straightforward to verify the following fact.

\begin{fact}\label{fact:powerderivative}
If $x \mapsto x^\lambda\colon K^>\to K$ is an $\cL(\emptyset)$-definable power function on $K$, then $(y^\lambda)' = \lambda y ^{\lambda-1}y'$ for all $y \in K^>$.
\end{fact}

Given an element $a$ in a $\Td$-extension of $K$, we write $K\llangle a \rrangle$ in place of $K\langle \jet^\infty(a)\rangle$. Then $K\llangle a \rrangle$ is itself a $\Td$-extension of $K$. We say that $a$ is \textbf{$\Td$-algebraic over $K$} if $\jet^r(a)$ is $\cL(K)$-dependent for some $r$; otherwise, $a$ is said to be \textbf{$\Td$-transcendental}. Equivalently, $a$ is $\Td$-algebraic over $K$ if $a \in \cld(K)$, where $\cld$ is the $\der$-closure pregeometry considered in~\cite[Section 3]{FK21}. More generally, a $\Td$-extension $L$ of $K$ is said to be \textbf{$\Td$-algebraic over $K$} if each $a \in L$ is $\Td$-algebraic over $K$ or, equivalently, if $L \subseteq \cld(K)$. As a consequence of~\cite[Lemma 3.2 (4)]{FK21}, we have the following:

\begin{fact}\label{fact:dependence}
Let $a$ be an element in a $\Td$-extension of $K$. If the tuple $\jet^r(a)$ is $\cL(K)$-dependent, then $K\llangle a \rrangle = K \langle \jet^{r-1} a\rangle$.
\end{fact}

Any $T$-extension of $K$ can be expanded to a $\Td$-extension of $K$, and this expansion depends entirely on how the derivation is extended to a $\dclL$-basis:

\begin{fact}[{\cite[Lemma~2.13]{FK21}}]
\label{fact:transext}
Let $M$ be a $T$-extension of $K$, let $A$ be a $\dclL$-basis for $M$ over $K$, and let $s\colon A\to M$ be a map. Then there is a unique extension of $\der$ to a $T$-derivation on $M$ such that $a'= s(a)$ for all $a \in A$.
\end{fact}

The $T$-derivation $\der$ is said to be \textbf{generic} if for each $\cL(K)$-definable function $F\colon U\to K$ with $U \subseteq K^r$ open, there is $a \in K$ with $\jet^{r-1}(a) \in U$ such that $a^{(r)} = F(\jet^{r-1}(a))$. Let $\TdG$ be the $\Ld$-theory that extends $\Td$ by axioms stating that $\der$ is generic. In~\cite{FK21}, it was shown that the theory $\TdG$ is the model completion of $\Td$. We record here the main extension and embedding results that go into this proof, for use in Section~\ref{sec:modelcompletion}.

\begin{fact}[{\cite[Proposition~4.3 and Lemma~4.7]{FK21}}]
\label{fact:tdextend}
There is a model $M \models \TdG$ that extends $K$ with $|M| = |K|$. Given any such extension $M$ and any $|K|^+$-saturated model $M^* \models \TdG$ extending $K$, there is an $\Ld(K)$-embedding $M\to M^*$.
\end{fact}

Next we recall some terminology and notation about linear differential operators from \cite[Chapter~5]{ADH17}.
Let $K\{Y\} \coloneqq K[Y, Y', \dots]$ be the differential ring of differential polynomials over $K$ and let $K[\der]$ be the ring of linear differential operators over $K$.
An element in $K[\der]$ corresponds to a $C$-linear operator $y \mapsto A(y)$ on $K$, where $A(Y) = a_0Y+a_1Y' + \dots + a_rY^{(r)} \in K\{Y\}$ is a homogeneous linear differential polynomial.
Viewed as an operator in $K[\der]$, we write $A = a_0+a_1\der+\dots+a_r\der^r$, and we freely pass between viewing $A$ as an element of $K[\der]$ and of $K\{Y\}$.
Henceforth, $A = a_0+a_1\der+\dots+a_r\der^r \in K[\der]$.
If $a_r \neq 0$, we say that $A$ has \textbf{order} $r$.
Multiplication in the ring $K[\der]$ is given by composition.
In particular, for $g \in K^{\times}$ the linear differential operator $Ag \in K[\der]$ corresponds to the differential polynomial $A_{\times g}(Y) = A(gY) \in K\{Y\}$.
We call $K$ \textbf{linearly surjective} if every $A \in K[\der]^{\neq} \coloneqq K[\der]\setminus\{0\}$ is surjective.

%------------------------------------------------------------------------------%
\subsection{Background on \texorpdfstring{$T$}{T}-convex \texorpdfstring{$T$}{T}-differential fields}
%------------------------------------------------------------------------------%
Let $\LdO\coloneqq \LO\cup\Ld = \cL\cup\{\cO,\der\}$, and let $\TdO$ be the $\LdO$-theory that extends $\Td$ and $\TO$ by the additional axiom ``$\der \smallo \subseteq \smallo$''.
\begin{assumption}
For the remainder of this paper, let $K = (K,\cO,\der) \models \TdO$.
\end{assumption}
This additional axiom, called \textbf{small derivation}, ensures that $\der$ is continuous with respect to the valuation topology (equivalently, order topology) on $K$ \cite[Lemma~4.4.6]{ADH17}, so every model of $\TdO$ is a $T$-convex $T$-differential field, as defined in the introduction. 
Also, small derivation gives $\der\cO\subseteq\cO$ too \cite[Lemma~4.4.2]{ADH17}, so $\der$ induces a derivation on $\k$.
Moreover, $\k$ is a $T$-differential field (see \cite[p.\ 280]{Ka22}).
In this paper we are interested in the case that the derivation induced on $\k$ is nontrivial; indeed, often it will be linearly surjective or even generic.
A consequence of the main results of \cite{Ka22} provides spherically complete extensions in this case:
\begin{fact}[{\cite[Corollary~6.4]{Ka22}}]\label{fact:sphcomp}
If the derivation induced on $\k$ is nontrivial, then $K$ has a spherically complete immediate $\TdO$-extension.
\end{fact}
It follows that for $K$ such that the derivation induced on $\k$ is nontrivial, $K$ is spherically complete if and only if $K$ has no proper immediate $\TdO$-extension.
The notion of an $\cL(K)$-definable function being in \emph{implicit form}, defined in the next section, plays an important role in the result above and our work here, as does the notion of \emph{vanishing}, which we define in Section~\ref{sec:sphcompTdh} when it is needed.

Now we define an important function on $\Gamma$ defined by a linear differential operator $A \in K[\der]^{\neq}$; for details see \cite[Sections~4.5, 5.6, 6.1]{ADH17}.
The valuation $v$ of $K$ induces a valuation $v$ on the additive group $K[\der]$ given by $vA \coloneqq \min\{ va_i : 0 \leq i \leq r\}$; in the same way we extend $v$ to linear differential polynomials $A(Y) \in K\{Y\}$.
Combining this with multiplicative conjugation yields a strictly increasing function $v_A \colon \Gamma \to \Gamma$ defined by $v_A(\gamma) \coloneqq v(A_{\times g})$ for any $g \in K^{\times}$ with $vg=\gamma$.
As a consequence of the Equalizer Theorem \cite[Theorem~6.0.1]{ADH17}, $v_A$ is bijective.
We extend $v_A$ to $\Gamma \cup \{\infty\}$ by $v_A(\infty)=v(A_{\times 0})\coloneqq \infty$.

%------------------------------------------------------------------------------%
\section{\texorpdfstring{$T^\der$}{T∂}-henselianity}\label{sec:Tdh}
%------------------------------------------------------------------------------%
Given an $\cL(K)$-definable function $F\colon K^{1+r}\to K$, a linear differential operator $A \in K[\der]^{\neq}$, and an open $v$-ball $B\subseteq K$, we say that \textbf{$A$ linearly approximates $F$ on $B$} if 
\[
v\big(F(\jet^r b) - F(\jet^r a) - A(b-a)\big) \ > \ vA_{\times (b-a)}
\]
for all $a,b \in B$ with $a \neq b$. Note that if $A$ linearly approximates $F$ on $B$, then $A$ linearly approximates $F$ on any open $v$-ball contained in~$B$.

In this section, $F\colon K^{1+r}\to K$ will always be an $\cL(K)$-definable function in \textbf{implicit form}, that is, \[F \ = \ \mathfrak{m}_F\big(Y_{r}-I_F(Y_0,\dots,Y_{r-1})\big)\] for some $\mathfrak{m}_F \in K^{\times}$ and $\cL(K)$-definable function $I_F \colon K^{r} \to K$; implicit form of definable functions was introduced and exploited in~\cite{Ka22}.
Additionally, let $A$ range over $K[\der]$, $a$ range over $K$, and $\gamma$ range over $\Gamma$.
% We say that \textbf{$(F,A,a,\gamma)$ is in $\Td$-hensel configuration} if $A$ linearly approximates $F$ on $B(a,\gamma)$ and $vF(\jet^r a) > v_A(\gamma)$.
We say that $K$ is \textbf{$\Td$-henselian} if:
\begin{enumerate}[label=($\Td$H\arabic*)]
    \item its differential residue field $\k$ is linearly surjective;
    \item\label{TdH2} whenever $A$ linearly approximates $F$ on $\smallo$ and $vF(0)>vA$, there is $\epsilon \in \smallo$ with $F(\jet^r \epsilon) = 0$ and $vA_{\times \epsilon}\geq vF(0)$.
\end{enumerate}
In the above definition and later, $F(0)$ means $F(0,\dots,0)$ as appropriate.
We allow valuations to be infinite, so if $F(0)=0$, then we may take $\epsilon=0$.

We now provide a more flexible equivalent formulation of \ref{TdH2}, making use of the operation of \emph{affine conjugation} from \cite[Section~2.1]{Ka22} to pass between the two.
This is the form we use throughout the paper without comment, and it
% The definition of $\Td$-henselianity
is inspired by the definition of \emph{$\sigma$-henselianity} for  analytic difference valued fields from \cite{Ri17}.
For this, we say that \textbf{$(F,A,a,\gamma)$ is in $\Td$-hensel configuration} if $A$ linearly approximates $F$ on $B(a,\gamma)$ and $vF(\jet^r a) > v_A(\gamma)$.

\begin{lemma}
The axiom \textnormal{\ref{TdH2}} is equivalent to 
\begin{enumerate}[label=($\Td$H2$'$)]
    \item\label{TdH2'} for every $(F,A,a,\gamma)$ in $\Td$-hensel configuration, there exists $b \in B(a,\gamma)$ with $F(\jet^r b) = 0$ and $vA_{\times(b-a)} \geq vF(\jet^r a)$.
\end{enumerate}
As before, if $F(\jet^r a) = 0$, we may take $b = a$.
\end{lemma}
\begin{proof}
Suppose that \ref{TdH2'} holds. Given $F$ and $A$ as in \ref{TdH2}, we see that $(F,A,0,0)$ is in $\Td$-hensel configuration, which gives $b \in B(0,0) = \smallo$ with $F(\jet^r b) = 0$ and $vA_{\times b} \geq vF(0)$. Now suppose that \ref{TdH2} holds and that $(F,A,a,\gamma)$ is in $\Td$-hensel configuration. Take $d \in  K^\times$ with $vd = \gamma$. Recall from~\cite[Section~2.1]{Ka22} the operation of \emph{affine conjugation}---this associates to $F$ the $\cL(K)$-definable function $F_{+a,\times d}\colon K^{1+r}\to K$ with the property that $F_{+a,\times d}(\jet^r y) = F(\jet^r(dy+a))$ for all $y \in K$. Since $A$ linearly approximates $F$ on $B(a,\gamma)$, it is routine to check that $A_{\times d}$ linearly approximates $F_{+a,\times d}$ on $\smallo$. Additionally, $vF_{+a,\times d}(0) = vF(\jet^ra) > v_A(\gamma) = vA_{\times d}$, so applying \ref{TdH2} to $F_{+a,\times d}$ and $A_{\times d}$ yields $\epsilon \in \smallo$ with $F_{+a,\times d}(\jet^r\epsilon) = 0$ and $v(A_{\times d})_{\times \epsilon}\geq vF_{+a,\times d}(0)$. Putting $b \coloneqq d\epsilon + a \in B(a,\gamma)$, we see that $F(\jet^rb)= 0$ and that $vA_{\times(b-a)} = vA_{\times d\epsilon} \geq vF(\jet^ra)$.
\end{proof}

\subsection{Basic consequences}\label{sec:Tdhbasic}
Our notion of $\Td$-henselianity is an analogue of differential-henselianity, as considered in~\cite{Sc00,ADH17} (abbreviated by ``$\d$-henselianity,'' and likewise with ``$\d$-henselian''). We discuss the precise connection between these notions in Section~\ref{sec:RCF}. 
%For now, we note that many of the basic consequences of $\Td$-henselianity are routine adaptations of results about $\d$-henselian fields from \cite[Sections 7.1 and 7.5]{ADH17}. 
For now, we list some  consequences of $\Td$-henselianity needed later.
The proofs are mostly routine adaptations of results from \cite[Sections 7.1 and 7.5]{ADH17}, using $\Td$-henselianity instead of $\d$-henselianity, so we give a couple as an illustration and omit most of the others.

For the next lemma, we call $A \in K[\der]^{\neq}$ \textbf{neatly surjective} if for every $b \in K^{\times}$ there is $a \in K^{\times}$ with $A(a)=b$ and $v_A(va)=vb$.
\begin{lemma}\label{lem:Tdhneatsurj}
Suppose that $K$ is $\Td$-henselian. Then every $A \in K[\der]^{\neq}$ is neatly surjective.
\end{lemma}
\begin{proof}
	%To see that $A$ is neatly surjective,
    Let $b \in K^{\times}$ and take $\alpha \in \Gamma$ with $v_A(\alpha) = \beta \coloneqq vb$. We need to find $a \in K^{\times}$ with $va=\alpha$ and $A(a)=b$.
	Take $\phi \in K^{\times}$ with $v\phi = \alpha$ and let $D \coloneqq b^{-1}A\phi \in K[\der]^{\neq}$, so $v(D)=0$.
	Let $A \in K[\der]^{\neq}$ have order $r$ and take an $\cL(K)$-definable function $F \colon K^{1+r} \to K$ in implicit form such that $F(\jet^{r} a)=D(a)-1$ for all $a \in K$.
	Since $\k$ is linearly surjective, we have $u \in \cO^{\times}$ with $D(u)-1 \prec 1$.
	As $D$ linearly approximates $F$ on $K$, $(F, D, u, 0)$ is in $\Td$-hensel configuration, and thus we have $y \sim u$ with $F(\jet^{r}y)=0$.
	Then $a \coloneqq \phi y$ works.
\end{proof}

The next corollary follows from Lemma~\ref{lem:Tdhneatsurj} as \cite[Corollary~7.1.9]{ADH17} follows from \cite[Lemma~7.1.8]{ADH17}.
\begin{corollary}\label{cor:7.1.9}
If $K$ is $\Td$-henselian, then $\smallo = (1+\smallo)^{\dagger}$.
\end{corollary}

%Likewise we obtain analogues of \cite[Lemma~7.1.10]{ADH17} and \cite[Corollary~7.1.11]{ADH17}.
%\begin{lemma}\label{lem:7.1.10}
%Suppose that $K$ is $\Td$-henselian and monotone.
%If $b \in K$ with $b' \prec b$, then $b \asymp c$ for some $c \in C$.
%\end{lemma}

Under the assumption of monotonicity this yields additional information as in \cite[Corollary~7.1.11]{ADH17}.
We say that $K$ has \textbf{many constants} if $v(C^{\times})=\Gamma$.
Note that if $K$ has many constants (and small derivation), then $K$ is monotone.
\begin{corollary}\label{cor:7.1.11}
Suppose that $K$ is $\Td$-henselian and monotone, and $(\k^{\times})^{\dagger} = \k$.
Then $K$ has many constants and $(K^{\times})^{\dagger} = (\cO^{\times})^{\dagger} = \cO$.
\end{corollary}

Rather opposite to monotonicity, we say that $K$ has \textbf{few constants} if $C \subseteq \cO$.
Now we record several lemmas about $K$ with few constants adapted from \cite[Section~7.5]{ADH17}.
By \cite[Lemma~9.1.1]{ADH17} and Lemma~\ref{lem:Tdhneatsurj}, any $\Td$-henselian $K$ with few constants is \textbf{asymptotic} in the sense of \cite[Chapter~9]{ADH17}; that is, for all nonzero $a, b \in \smallo$, we have $a \prec b \iff a' \prec b'$.
Conversely, any asymptotic $K$ obviously satisfies $C \subseteq \cO$.

The next lemma is analogous to \cite[Lemma~7.5.1]{ADH17} but has a simpler proof.
\begin{lemma}\label{lem:7.5.1}
Suppose that $K$ is $\Td$-henselian and let $A \in K[\der]^{\neq}$.
If $A(1) \prec A$, then $A(y)=0$ for some $y \in K$ with $y \sim 1$ and $vA_{\times (y-1)} \geq vA(1)$.
\end{lemma}
\begin{proof}
Suppose $A$ has order $r$ and take an $\cL(K)$-definable function $F \colon K^{1+r} \to K$ in implicit form such that $F(\jet^r a) = A(a)$ for all $a \in K$.
If $A(1)\prec A$, then $(F, A, 1, 0)$ is in $\Td$-hensel configuration, so $\Td$-henselianity yields the desired~$y$.
\end{proof}

The next lemma follows from Lemma~\ref{lem:7.5.1} as \cite[Lemma~7.5.2]{ADH17} follows from \cite[Lemma~7.5.1]{ADH17}.
\begin{lemma}\label{lem:7.5.2}
Suppose that $K$ is $\Td$-henselian and $C \subseteq \cO$.
Let $A \in K[\der]^{\neq}$ have order $r$.
There do not exist $b_0, \dots, b_r \in K^{\times}$ such that $b_0 \succ b_1 \succ \dots \succ b_r$ and $A(b_i) \prec Ab_i$ for $i=0, \dots, r$.
\end{lemma}
% \begin{proof}
% Suppose towards a contradiction that there exist $b_0, \dots, b_r \in K^{\times}$ such that $b_0 \succ b_1 \succ \dots \succ b_r$ and $A(b_i) \prec Ab_i$ for $i=0, \dots, r$.
% Then by Lemma~\ref{lem:7.5.1} applied to each $Ab_i$, we get $y_i \sim 1$ such that $A(b_iy_i)=0$.
% From $b_iy_i \sim b_i$, we get that $b_0y_0, \dots, b_ry_r$ are $C$-linearly independent zeros of $A$, a contradiction.
% \end{proof}
The proof of Lemma~\ref{lem:7.5.5} is adapted from that of \cite[Lemma~7.5.5]{ADH17}.
\begin{lemma}\label{lem:7.5.5}
Suppose that $K$ is $\Td$-henselian and $C \subseteq \cO$.
Let $F \colon K^{1+r} \to K$ be an $\cL(K)$-definable function in implicit form and $A \in K[\der]^{\neq}$ have order $q$.
There do not exist $y_0, \dots, y_{q+1} \in K$ such that
\begin{enumerate}
	\item\label{lem:7.5.5i} $y_{i-1}-y_i \succ y_i-y_{i+1}$ for $i = 1, \dots, q$ and $y_q \neq y_{q+1}$;
	\item\label{lem:7.5.5ii} $F(\jet^ry_i) = 0$ for $i=0, \dots, q+1$;
	\item\label{lem:7.5.5iii} $(F, A, y_{q+1}, \gamma)$ is in $\Td$-hensel configuration and $v(y_0-y_{q+1})>\gamma$ for some $\gamma \in \Gamma$.
\end{enumerate}
\end{lemma}
\begin{proof}
Suppose towards a contradiction that we have $y_0, \dots, y_{q+1} \in K$ and $\gamma \in \Gamma$ satisfying \ref{lem:7.5.5i}--\ref{lem:7.5.5iii}.
Below, $i$ ranges over $\{0, \dots, q\}$.
Set $G \coloneqq F_{+y_{q+1}}$ and $b_i \coloneqq y_i-y_{q+1}$.
Then $b_i \sim y_i-y_{i+1}$, so $b_0 \succ b_1 \succ \dots \succ b_q \neq 0$.
Also, $G(\jet^r b_i) = F(\jet^r y_i)  = 0$ and $G(0) = F(\jet^r y_{q+1}) = 0$.
Since $vb_0>\gamma$, by \ref{lem:7.5.5iii} we have $\varepsilon_i$ such that $v\varepsilon_i > vA_{\times b_i}$ and $0=G(\jet^r b_i) = A(b_i)+\varepsilon_i$.
In particular, $vA(b_i)=v(\varepsilon_i)> vA_{\times b_i}$, contradicting Lemma~\ref{lem:7.5.2} with $q$ in the role of $r$.
\end{proof}

% \begin{proposition}\label{prop:7.5.6}
% Suppose that $K$ is $\Td$-henselian and $C \subseteq \cO$.
% Let $F \colon K^{1+r} \to K$ be an $\cL(K)$-definable function in implicit form and $A \in K[\der]^{\neq}$ have order at most $r$.
% If $L$ is an immediate $\TdO$-extension of $K$ and $(F, A, 0, 0)$ is in $\Td$-hensel configuration in $L$, then $F(\jet^{r})$ has the same zeros in $\smallo_L$ as in~$\smallo$.
% \end{proposition}
% \begin{proof}
% Suppose towards a contradiction that we have $\ell \in \smallo_L \setminus \smallo$ such that $F(\jet^{r}\ell)=0$.
% We claim that for any $\gamma \in v(\ell-K)^{\geq}$, there is $y \in \smallo$ with $F(\jet^{r}y)=0$ and $v(\ell-y)>\gamma$.
% To see this, fix $\gamma \in v(\ell-K)^{\geq}$ and take $a \in K$ with $v(\ell-a)>\gamma$.
% Since $F(\jet^{r}\ell)=0$ and $(F, A, 0, 0)$ is in $\Td$-hensel configuration in $L$, so is $(F, A, \ell, \gamma)$.
% Then $v\big( F(\jet^{r}a) - A(\ell-a) \big) > v_A\big(v(\ell-a)\big) > v_A(\gamma)$.
% In particular, $vF(\jet^{r}a) > v_A(\gamma)$, and hence $(F, A, a, \gamma)$ is in $\Td$-hensel configuration.
% Thus we have $b \in B(a, \gamma)$ such that $F(\jet^{r}b)=0$. %and $vA_{\times (b-a)} \geq vF(\jet^{r}a)$.
% It remains to note that $v(\ell-b)>\gamma$.
% From the claim we get $y_0, \dots, y_{r+1} \in \smallo$ with $F(\jet^{r}y_i)=0$ for all $i \in \{0, \dots, r+1\}$ and
% \[
% 0\ <\ v(\ell-y_0)\ <\ v(\ell-y_1)\ <\ \cdots\ <\ v(\ell-y_{r+1}),
% \]
% contradicting Lemma~\ref{lem:7.5.5} with $\gamma=0$.
% \end{proof}

\begin{remark}\label{rem:powerbounded}
What of our assumption that $T$ is power bounded? The definitions of $T$-convex valuation ring, $T$-derivation, and $\Td$-henselianity do not use it, nor do the above consequences of $\Td$-henselianity.
Suppose temporarily that $T$ is not power bounded. Then by Miller's dichotomy~\cite{Mi96}, we have an $\cL(\0)$-definable exponential function $E \colon K \to K^>$ (i.e., $E$ is an isomorphism from the ordered additive group of $K$ to the ordered multiplicative group $K^>$ that is equal to its own derivative). 
As we now explain, this $E$ is not compatible with monotonicity, at least if we suppose in addition that the derivation of $\k$ is nontrivial and that $K$ is nontrivially valued. Let $a \in K$ with $a \succ 1$. If $a' \succ 1$, then $E(a)' = a'E(a) \succ E(a)$, so $K$ is not monotone. If $a'\preceq 1$, taking  $u \in K$ with $u \asymp u' \asymp 1$ yields $(au)'=a'u+au' \sim au' \asymp a$, so $b\coloneqq au$ satisfies $b' \succ 1$ and $E(b)' \succ E(b)$.
\end{remark}

%------------------------------------------------------------------------------%
\subsection{Lifting the residue field}\label{sec:lift}
%------------------------------------------------------------------------------%
In this subsection we show that every $\Td$-henselian $\TdO$-model admits a lift of its differential residue field as an $\Ld$-structure in the sense defined before Corollary~\ref{cor:reslift} (cf.\ \cite[Proposition~7.1.3]{ADH17}).
A \textbf{partial $T$-lift of the residue field $\k$} is a $T$-submodel $E\subseteq K$ that is contained in $\cO$. If $E$ is a partial $T$-lift of $\k$, then the residue map induces an $\cL$-embedding $E \to \k$. If this embedding is surjective onto $\k$, then $E$ is called a \textbf{$T$-lift of~$\k$}. By \cite[Theorem~2.12]{DL95}, $\k$ always admits a $T$-lift.

\begin{lemma}\label{lem:closesamechi}
Let $E\subseteq K$ be a partial $T$-lift of $\k$ and let $a,b$ be tuples in $\cO^n$ with $a-b \in \smallo^n$. Suppose that $a$ is $\cL(E)$-independent and that $E\langle a \rangle$ is a partial $T$-lift of $\k$. Then $a$ and $b$ have the same $\LRV$-type over $E \cup \rv\!\big(E\langle a \rangle\big)$. In particular, $E\langle b\rangle$ is also a partial $T$-lift of $\k$.
\end{lemma}
\begin{proof}
We proceed by induction on $n$, with the case $n = 0$ holding trivially. Assume that $(a_1,\ldots,a_{n-1})$ and $(b_1,\ldots,b_{n-1})$ have the same $\LRV$-type over $E \cup \rv\!\big(E\langle a_1,\ldots,a_{n-1} \rangle\big)$. This assumption yields a partial elementary $\LRV$-embedding $\imath\colon E\langle a_1,\ldots,a_{n-1} \rangle \to E\langle b_1,\ldots,b_{n-1}\rangle$ that fixes $E$ and $\rv\!\big(E\langle a_1,\ldots,a_{n-1} \rangle\big)$ and sends $a_i$ to $b_i$ for each $i<n$. Let $g \in E\langle a_1,\ldots,a_{n-1} \rangle$. We claim that $g < a_n \Longleftrightarrow \imath(g)<b_n$. Take some $\cL(E)$-definable function $G$ with $g = G(a_1,\ldots,a_{n-1})$. Since $a$ is $\cL(E)$-independent, the function $G$ is continuous on some $\cL(E)$-definable neighborhood $U$ of $(a_1,\ldots,a_{n-1})$. Note that $(b_1,\ldots,b_{n-1})\in U$ as well, so $G(a_1,\ldots,a_{n-1}) - G(b_1,\ldots,b_{n-1}) \prec 1$ by~\cite[Lemma 1.13]{DL95}. Since $E\langle a\rangle$ is a partial $T$-lift of $\k$, we have $a_n - G(a_1,\ldots,a_{n-1}) \asymp 1$. By assumption, $a_n - b_n \prec 1$, so
\[
a_n - g\ =\ a_n - G(a_1,\ldots,a_{n-1})\ \sim\ b_n - G(b_1,\ldots,b_{n-1})\ =\ b_n - \imath(g).
\]
In particular, $a_n - g$ and $b_n - \imath(g)$ have the same sign. This allows us to extend $\imath$ to an $\cL(E)$-embedding $\jmath\colon E \langle a \rangle \to E\langle b\rangle$ by mapping $a_n$ to $b_n$. To see that $\jmath$ is even an elementary $\LRV$-embedding over $E \cup \rv\!\big(E\langle a \rangle\big)$, it suffices to show that $\rv h = \rv \jmath(h)$ for each $h \in E \langle a \rangle$. We may assume that $h \neq 0$, so $h \asymp 1$. Take some $\cL(E)$-definable function $H$ with $h = H(a)$. Again, $H$ is continuous on an open set containing $a$ and $b$, so we may use~\cite[Lemma 1.13]{DL95} to get that $H(a) -H(b) \prec 1$. Thus, $H(a) \sim H(b)$, so 
\[
\rv h\ =\ \rv H(a) \ =\ \rv H(b)\ =\ \rv\jmath(h).\qedhere
\]
\end{proof}

\begin{lemma}\label{lem:dependentlift}
Let $n$ be given, let $E$ be a partial $T$-lift of $\k$, let $a \in \cO^\times$ with $\bar{a} \not\in \res(E)$, and suppose that $\jet^{n-1}(a)$ is $\cL(E)$-independent and that $E\langle\jet^{n-1}(a) \rangle$ is a partial $T$-lift of $\k$. Let $G\colon K^{1+n}\to K$ be an $\cL(E)$-definable function in implicit form with $\fm_G = 1$. Then there is $A \in K[\der]$ with $vA = 0$ that linearly approximates $G$ on $a+\smallo$.
\end{lemma}
\begin{proof}
By applying Fact~\ref{fact:rvapprox} to the function $I_G$, we find an $\LRVeq(E)$-definable map $\chi\colon K^n \to \oRV{K}^{\eq}$ such that for each $s \in \chi(K^n)$, if $\chi\inv(s)$ contains an open $v$-ball, then either $I_G$ is constant on $\chi\inv(s)$ or there is $d \in K^n$ such that
\begin{equation}\label{eq:dependentlift}
v\big(I_G(x)- I_G(y) - d\cdot (x-y)\big) \ >\ vd+ v(x-y)
\end{equation}
for all $x, y \in \chi\inv(s)$ with $x \neq y$. Let $s_0\coloneqq \chi\big(\jet^{n-1}(a)\big)$ and let $U\coloneqq \chi\inv(s_0)$. Note that if $x\in \jet^{n-1}(a)+\smallo^n$, then $x$ and $\jet^{n-1}(a)$ have the same $\LRV$-type over $E \cup \rv E\langle \jet^{n-1}(a)\rangle$ by Lemma~\ref{lem:closesamechi}, so $x \in U$. In particular, $\jet^{n-1}(u) \in U$ for all $u \in a+\smallo$, since $K$ has small derivation. 

We choose $A$ as follows:\ if $I_G$ is constant on $U$, then we let $A$ be the linear differential operator $\der^n$. If $I_G$ is not constant on $U$, then we let $A$ be
\[
\der^n - d_n \der^{n-1} - \cdots-d_1\ \in \ K[\der],
\]
where $d = (d_1,\ldots,d_n) \in K^n$ is chosen such that (\ref{eq:dependentlift}) holds for $x,y \in U$ with $x \neq y$. We claim that $vA = 0$. This is clear if $A = \der^n$, so we may assume that $I_G$ is not constant on $U$, and we need to show that $d_i \preceq 1$ for each $i =1,\ldots,n$. Take $i \leq n$ with $vd_i=vd$ and suppose towards contradiction that $d_i \succ 1$. Take tuples $x,y \in \jet^{n-1}(a)+\smallo^n$ such that $x_i - y_i \succ d_i\inv$ and $x_j = y_j$ for $j \leq n$ with $j\neq i$. Then $d \cdot (x-y) = d_i(x_i-y_i) \succ 1$. Since $E\langle x\rangle$ is a partial $T$-lift of $K$ by Lemma~\ref{lem:closesamechi}, we have $I_G(x) \preceq 1$. Likewise, $I_G(y) \preceq 1$, so 
\[
v\big(I_G(x)- I_G(y) - d\cdot (x-y)\big) \ = \ vd_i+v(x_i-y_i)\ =\ vd+ v(x-y).
\]
This is a contradiction, since both $x$ and $y$ are in $U$. 

Now, we will show that $A$ linearly approximates $G$ on $a+\smallo$. Again, this is clear if $I_G$ is constant on $U$, so we may assume that $I_G$ is not constant on $U$. Let $t,u \in a+\smallo$ with $t\neq u$. We have
\[
v\big(G(\jet^n u)- G(\jet^n t) - A(u-t)\big) \ = \ v\big(I_G(\jet^{n-1}u)- I_G(\jet^{n-1}t)- d\cdot \jet^{n-1}(u-t)\big) \ >\ vd + v\big(\jet^{n-1}(u-t)\big),
\]
so it suffices to show that $vd+ v\big(\jet^{n-1}(u-t)\big)\geq vA_{\times(u-t)}$. We consider two cases. First, if $(u-t)' \preceq u-t$, then $(u-t)^{(m)} \preceq(u-t)$ for all $m$ and $vA_{\times (u-t)} = vA + v(u-t)$ by~\cite[Lemma 4.5.3]{ADH17}. Using also that $vA \leq vd$, this yields 
\[
vd + v\big(\jet^{n-1}(u-t)\big)\ =\ vd+ v(u-t)\ \geq \ vA + v(u-t)\ =\ vA_{\times (u-t)},
\]
in which case we are done. Next, suppose that $(u-t)' \succ (u-t)$, so $(u-t)^{(m)} \asymp (u-t)\big((u-t)^\dagger\big)^m$ for each $m$ by~\cite[Lemma 6.4.1]{ADH17}. In particular, $(u-t)^{(i)}\prec(u-t)^{(n)}$ whenever $i<n$, since $(u-t)^\dagger \succ 1$. Since $vA=0$, this gives
\[
A_{\times (u-t)}\ \succeq\ A(u-t) \ =\ (u-t)^{(n)} - d_n (u-t)^{(n-1)} - \cdots-d_1(u-t)\ \asymp\ (u-t)^{(n)}.
\]
Having already established $v(\jet^{n-1}(u-t))=v((u-t)^{(n)})$ and $vd\geq 0$, we conclude that
\[
vA_{\times(u-t)}\ \leq \ v\big((u-t)^{(n)}\big) \ = \ v\big(\jet^{n-1}(u-t)\big)\ \leq\ vd+v\big(\jet^{n-1}(u-t)\big),
\]
as desired.
\end{proof}

\begin{proposition}\label{prop:resext}
Suppose that $K$ is $\Td$-henselian, let $E$ be a %$\TdO$-submodel
$T$-convex $T$-differential subfield of $K$, and let $a \in \k \setminus \res(E)$. Then there is $b \in \cO$ with $\res(b) = a$ and the following properties:
\begin{enumerate}
\item\label{prop:resexti} $\rk_{\cL}\!\big(E\llangle b\rrangle \big| E\big)=\rk_{\cL}\!\big(\res(E)\llangle a\rrangle \big| \res(E)\big)$;
\item\label{prop:resextii} $v\big(E\llangle b\rrangle^\times\big) = v(E^\times)$;
\item\label{prop:resextiii} If $K^*$ is any $\Td$-henselian $\TdO$-extension of $E$, then any $\Ld(\res(E))$-embedding $\imath\colon \res(E)\llangle a\rrangle \to \res(K^*)$ lifts to an $\LdO(E)$-embedding $\jmath\colon E\llangle b\rrangle \to K^*$;
\end{enumerate}
\end{proposition}
\begin{proof}
First, consider the case that $a$ is $\Td$-transcendental over $\res(E)$.
Let $b\in K$ be a lift of $a$. We first show that \ref{prop:resexti} holds. We claim that $b$ is $\Td$-transcendental over $E$. Suppose not, and take $n$ and an $\cL(E)$-definable function $G\colon K^n\to K$ with $b^{(n)} = G(\jet^{n-1}b)$. Then $\jet^n(b)$ belongs to $\Graph(G)$, the graph of $G$, so $\jet^n(a)$ belongs to $\res(\Graph(G))$. By~\cite[Corollaries 1.13 and 1.14]{vdD97}, the set $\res(\Graph(G))$ is $\res(E)$-definable of $\cL$-dimension at most $n$, contradicting that $a$ is $\Td$-transcendental over $\res(E)$.
For each $n$, let $E_n\coloneqq E\langle \jet^{n-1} b\rangle$, so $E\llangle b\rrangle = \bigcup_n E_n$ and
\[
\rk_{\cL}(E_n | E)\ =\ \rk_{\cL}\!\big(\res(E_n)| \res(E)\big)\ =\ n.
\]
The Wilkie inequality gives that $v(E_n^\times) = v(E)$, so $v\big(E\llangle b\rrangle^\times\big) = v(E^\times)$ as well, proving \ref{prop:resextii}.
Now, let $K^*$ and $\imath$ be as in \ref{prop:resextiii} and let $b^* \in K^*$ be a lift of $\imath(a)$. Let $\jmath\colon E\llangle b\rrangle \to K^*$ be the map that fixes each element of $E$ and sends $\jet^\infty(b)$ to $\jet^\infty(b^*)$. By \ref{prop:resexti}, \ref{prop:resextii}, and a straightforward induction on $n$, we see that the restriction of $\jmath$ to $E_n$ is an $\LO(E)$-embedding, so $\jmath$ itself is an $\LdO(E)$-embedding lifting $\imath$.

Now assume that $a$ is $\Td$-algebraic over $\res(E)$, and let $n$ be minimal such that $a^{(n)} \in \res(E)\langle\jet^{n-1} a\rangle$. Let $b \in K$ be a lift of $a$. By the argument in the previous case, we get that $\jet^{n-1}(b)$ is $\cL(E)$-independent. Let $E_0\subseteq E$ be a $T$-lift of $\res(E)$, so $E_0\langle \jet^{n-1}b\rangle$ is a $T$-lift of $\res(E)\langle\jet^{n-1} a\rangle$. Since $a^{(n)} \in \res(E)\langle\jet^{n-1} a\rangle$, there is some $\cL(E_0)$-definable function $G\colon K^n\to K$ with $b^{(n)} - G(\jet^{n-1}b)\prec 1$.
Let $H\colon K^{n+1}\to K$ be the function 
\[
H(Y_0,\ldots,Y_n)\ \coloneqq\ Y_n - G(Y_0,\ldots,Y_{n-1}),
\]
so $H$ is in implicit form with $\fm_H = 1$ and $H(\jet^nb) \prec 1$. Applying Lemma~\ref{lem:dependentlift}, we get a linear differential operator $A$ with $vA = 0$ that linearly approximates $H$ on $b+\smallo$. Then $(H,A,b,0)$ is in $\Td$-hensel configuration, and since $K$ is assumed to be $\Td$-henselian, we may replace $b$ with some element in $b+\smallo$ and arrange that $H(\jet^nb) = 0$. Note that then $E\langle \jet^{n-1} b \rangle = E\llangle b \rrangle$, so \ref{prop:resexti} holds.
As above, \ref{prop:resextii} follows from \ref{prop:resexti} and the Wilkie inequality. For \ref{prop:resextiii}, let $K^*$ and $\imath$ be given. Arguing as we did with $K$, we can take a lift $b^*$ of $\imath(a)$ in $K^*$ such that $H(\jet^nb^*) = 0$.
Let $\jmath\colon E\llangle b\rrangle \to K^*$ be the map that fixes each element of $E$ and sends $\jet^{n-1}(b)$ to $\jet^{n-1}(b^*)$.
As above, this map is easily seen to be an $\LdO(E)$-embedding lifting~$\imath$.
\end{proof}

A \textbf{partial $\Td$-lift of the differential residue field $\k$} is a $\Td$-submodel $E\subseteq K$ that is contained in $\cO$. Equivalently, a partial $\Td$-lift is a partial $T$-lift that is closed under $\der$. If $E$ is a partial $\Td$-lift of $\k$, then the residue map induces an $\Ld$-embedding $E \to \k$. If $\res(E) = \k$, then $E$ is called a \textbf{$\Td$-lift of $\k$}.

\begin{corollary}\label{cor:reslift}
Suppose that $K$ is $\Td$-henselian. Then any partial $\Td$-lift of $\k$ can be extended to a $\Td$-lift of $\k$.
\end{corollary}
\begin{proof}
Let $E$ be a partial $\Td$-lift of $\k$ and let $a \in \k \setminus \res(E)$. Proposition~\ref{prop:resext} gives us $b \in \cO$ with $\res(b) = a$ such that $E\llangle b\rrangle$ is a partial $\Td$-lift of $\k$. The corollary follows by Zorn's lemma.
\end{proof}

Note that the prime model of $T$ (identified with $\dclL(\emptyset) \subseteq K$ and equipped with the trivial derivation) is a partial $\Td$-lift of $\k$, so Corollary~\ref{cor:reslift} has the following consequence:

\begin{theorem}\label{thm:reslift}
Suppose that $K$ is $\Td$-henselian. Then $\k$ admits a $\Td$-lift.
\end{theorem}

%------------------------------------------------------------------------------%
\subsection{Spherically complete implies \texorpdfstring{$T^\der$}{T∂}-henselian}\label{sec:sphcompTdh}
%------------------------------------------------------------------------------%

In this subsection we assume that the derivation on $\k$ is nontrivial.
To establish the result claimed in the subsection heading, under the assumption that $\k$ is linearly surjective, we work as usual with pseudocauchy sequences, which we abbreviate as \emph{pc-sequences}; see \cite[Sections~2.2 and 3.2]{ADH17} for definitions and basic facts about them.
In particular, recall that $K$ is spherically complete if and only if every pc-sequence in $K$ has a pseudolimit in $K$; if $a$ is a pseudolimit of $(a_\rho)$, we write $a_\rho \leadsto a$.
Also, if $L$ is an immediate $\TdO$-extension of $K$, then every element of $L \setminus K$ is the pseudolimit of a divergent pc-sequence in $K$ (i.e., a pc-sequence in $K$ that has no pseudolimit in $K$).
We are interested in the behavior of $\cL(K)$-definable functions along pc-sequences, which will allow us to study immediate $\TdO$-extensions of~$K$.

\begin{lemma}\label{lem:successorapprox}
Suppose that $\k$ is linearly surjective, $(F,A,a,\gamma)$ is in $\Td$-hensel configuration, and $F(\jet^r a) \neq 0$.
Then there is $b \in B(a,\gamma)$ such that $(F,A,b,\gamma)$ is in $\Td$-hensel configuration, $vA_{\times(b-a)} = vF(\jet^r a)$, and $F(\jet^r b) \prec F(\jet^r a)$. Moreover, if $b^*$ is any element of $K$ with $b^* - a \sim b- a$, then $(F,A,b^*,\gamma)$ is in $\Td$-hensel configuration, $vA_{\times(b^*-a)} = vF(\jet^r a)$, and $F(\jet^r b^*) \prec F(\jet^r a)$.
\end{lemma}
\begin{proof}
Let $\alpha \coloneqq vF(\jet^r a)> v_A(\gamma)$ and take $g\in K$ with $v_A(vg)= \alpha$, so $vg>\gamma$. Then for $y \in \cO^{\times}$, we have $a+gy \in B(a,\gamma)$ and
\[
F\big(\jet^r(a+gy)\big)\ =\ F(\jet^r a) + A(gy) + \epsilon,
\]
where $v\epsilon>vA_{\times gy} = v_A(vg) = \alpha$. Let $D$ be the linear differential operator $F(\jet^r a)\inv A_{\times g}$, so $D \asymp 1$. Since $\k$ is linearly surjective, we can find $u \in \cO^\times$ with $1+D(u) \prec 1$. Let $b\coloneqq a+gu$, so 
\[
F(\jet^r b) \ =\ F(\jet^r a) + A(gu)+\epsilon \ =\ F(\jet^r a)\big(1+D(u)\big)+\epsilon\ \prec\ F(\jet^r a).
\]
Since $b-a \asymp g$, we have $vA_{\times(b-a)} = v_A(vg)=\alpha$.
Since $B(b,\gamma) = B(a,\gamma)$ and $vF(\jet^r b)>vF(\jet^r a) > v_A(\gamma)$, the tuple $(F,A,b,\gamma)$ is in $\Td$-hensel configuration.

Now, let $b^* \in K$ with $b^* - a \sim b- a = gu$.
Then we have $u^* \in K$ with $b^* =a+ gu^*$ and $u^* \sim u$.
In particular, $1+D(u^*) \prec 1$, so the same argument as above gives that $(F,A,b^*,\gamma)$ is in $\Td$-hensel configuration, $vA_{\times(b^*-a)} = vF(\jet^r a)$, and $F(\jet^r b^*) \prec F(\jet^r a)$.
\end{proof}

\begin{lemma}\label{lem:Tdhcdivpc}
Suppose that $\k$ is linearly surjective, $(F,A,a,\gamma)$ is in $\Td$-hensel configuration, and $F(\jet^r b) \neq 0$ for all $b \in B(a,\gamma)$ with $vA_{\times(b-a)} \geq vF(\jet^r a)$. Then there is a divergent pc-sequence $(a_\rho)$ in $K$ such that $F(\jet^r a_\rho) \leadsto 0$.
\end{lemma}
\begin{proof}
Suppose we have a nonzero ordinal $\lambda$ and a sequence $(a_\rho)_{\rho<\lambda}$ in $B(a,\gamma)$ satisfying
\begin{enumerate}
\item\label{divpcseq1} $a_0 = a$ and $(F,A,a_\rho,\gamma)$ is in $\Td$-hensel configuration for each $\rho<\lambda$;
\item\label{divpcseq2} $v_A(\gamma_\rho)= vF(\jet^r a_\rho)$ whenever $\rho+1<\lambda$, where $\gamma_\rho\coloneqq v(a_{\rho+1}-a_\rho)$;
\item\label{divpcseq3} $vF(\jet^r a_\rho)$ is strictly increasing as a function of $\rho$, for $\rho<\lambda$.
\end{enumerate}
Such a sequence exists when $\lambda=1$, so it suffices to extend $(a_\rho)_{\rho<\lambda}$ to a sequence $(a_\rho)_{\rho<\lambda+1}$ in $B(a, \gamma)$ satisfying \ref{divpcseq1}--\ref{divpcseq3} with $\lambda+1$ in place of $\lambda$.

If $\lambda = \mu+1$ is a successor ordinal, then we use Lemma~\ref{lem:successorapprox} to find $a_\lambda \in B(a,\gamma)$ such that $(F,A,a_\lambda,\gamma)$ is in $\Td$-hensel configuration, $vA_{\times(a_\lambda-a_\mu)} = vF(\jet^r a_\mu)$, and $F(\jet^r a_\lambda) \prec F(\jet^r a_\mu)$. This extended sequence $(a_\rho)_{\rho<\lambda+1}$ satisfies conditions \ref{divpcseq1}--\ref{divpcseq3}.

Suppose that $\lambda$ is a limit ordinal. Then $v_A(\gamma_\rho)$ is strictly increasing as a function of $\rho$, so $\gamma_\rho$ is also strictly increasing by~\cite[Lemma 4.5.1(iii)]{ADH17}. Hence, $(a_\rho)_{\rho<\lambda}$ is a pc-sequence with $F(\jet^r a_\rho) \leadsto 0$. If $(a_\rho)_{\rho<\lambda}$ is divergent, then we are done. Otherwise, let $a_\lambda$ be any pseudolimit of $(a_\rho)_{\rho<\lambda}$ in $K$. Since $a_\lambda- a_\rho\sim a_{\rho+1} - a_\rho$ for each $\rho< \lambda$, Lemma~\ref{lem:successorapprox} gives that $(F,A,a_\lambda,\gamma)$ is in $\Td$-hensel configuration and that $F(\jet^r a_\lambda) \prec F(\jet^r a_\rho)$ for each $\rho<\lambda$. Thus, the extended sequence $(a_\rho)_{\rho<\lambda+1}$ satisfies \ref{divpcseq1}--\ref{divpcseq3}.
\end{proof}

\begin{corollary}\label{cor:sphcompTdh}
If $\k$ is linearly surjective and $K$ is spherically complete, then $K$ is $\Td$-henselian.
\end{corollary}

In fact, the divergent pc-sequence we build in Lemma~\ref{lem:Tdhcdivpc} is of a specific form analogous to ``differential-algebraic type'' in valued differential fields with small derivation (see \cite[Sections 4.4 and 6.9]{ADH17}). To introduce this notion and refine Corollary~\ref{cor:sphcompTdh}, let $(a_\rho)$ be a divergent pc-sequence in $K$, and let $\ell$ be a pseudolimit of $(a_\rho)$ in some $\TdO$-extension $L$ of $K$. Then $v(\ell-K) \subseteq \Gamma$ has no greatest element, and we say that a property holds \textbf{for all $y \in K$ sufficiently close to $\ell$} if there exists $\gamma \in v(\ell-K)$ such that the property holds for all $y \in K$ with $v(\ell-y)>\gamma$. Under the assumptions that in this paper $K$ has small derivation and in this section $\k$ has nontrivial derivation, the definition of vanishing from \cite{Ka22} simplifies as follows. We say that $F$ \textbf{vanishes at $(K, \ell)$} if whenever $a \in K$ and $d \in K^{\times}$ satisfy $\ell-a \prec d$, we have $I_{F_{+a, \times d}}(\jet^{r-1}y) \prec 1$ for all $y \in K$ sufficiently close to $d^{-1}(\ell-a)$. Let $Z_r(K, \ell)$ be the set of all $F$ of arity $1+r$ that vanish at $(K, \ell)$ and $Z(K, \ell) \coloneqq \bigcup_{r} Z_r(K, \ell)$; we have $Z_0(K, \ell)=\0$ by \cite[Lemma~5.2]{Ka22}.

We say that $(a_\rho)$ is of \textbf{$\Td$-algebraic type over $K$} if $Z(K,\ell)\neq \0$. Otherwise, $(a_\rho)$ is said to be of \textbf{$\Td$-transcendental type over $K$}. Note that $Z(K,\ell)$ does not depend on $\ell$ or even $(a_\rho)$, only on $v(\ell-K) \subseteq
\Gamma$. Thus, if $(a_\rho)$ is of $\Td$-algebraic type over $K$, then so is $(b_\sigma)$ for any pc-sequence $(b_\sigma)$ in $K$ equivalent to $(a_\rho)$ (recall the various characterizations of equivalence of pc-sequences in \cite[Lemma~2.2.17]{ADH17}). If $(a_\rho)$ is of $\Td$-algebraic type over $K$ and $r$ is minimal with $Z_r(K,\ell) \neq \0$, then any member of $Z_r(K,\ell)$ is called a \textbf{minimal $\Ld$-function} of $(a_\rho)$ over $K$.
We use the following two propositions to construct immediate extensions.

\begin{fact}{{\cite[Proposition~6.1]{Ka22}}}\label{prop:Ka226.1}
Suppose that $(a_\rho)$ is of $\Td$-transcendental type over $K$. Then $\ell$ is $\Td$-transcendental over $K$ and $K\llangle\ell\rrangle$ is an immediate $\TdO$-extension of $K$.
If $b$ is a pseudolimit of $(a_\rho)$ in a $\TdO$-extension $M$ of $K$, then there is a unique $\LdO(K)$-embedding $K\llangle\ell\rrangle \to M$ sending $\ell$ to~$b$.
\end{fact}

\begin{fact}{{\cite[Proposition~6.2]{Ka22}}}\label{prop:Ka226.2}
Suppose that $(a_\rho)$ is of $\Td$-algebraic type over $K$, and let $F$ be a minimal $\Ld$-function of $(a_\rho)$ over $K$.
Then $K$ has an immediate $\TdO$-extension $K\llangle a\rrangle$ with $F(\jet^{r}a)=0$ and $a_\rho\leadsto a$.
If $b$ is a pseudolimit of $(a_\rho)$ in a $\TdO$-extension $M$ of $K$ with $F(\jet^{r}b)=0$, then there is a unique $\LdO(K)$-embedding $K\llangle a\rrangle \to M$ sending $a$ to~$b$.
\end{fact}

Let us connect pc-sequences of $\Td$-algebraic type to $\Td$-algebraic $\TdO$-extensions.
We say that $K$ is \textbf{$\Td$-algebraically maximal} if $K$ has no proper immediate $\TdO$-extension that is $\Td$-algebraic over~$K$.

\begin{lemma}\label{lem:TdalgmaxTdalgtype}
%Suppose that $K$ has linear approximations along pc-sequences.
The following are equivalent:
\begin{enumerate}
	\item $K$ is $\Td$-algebraically maximal;
	\item every pc-sequence of $\Td$-algebraic type over $K$ has a pseudolimit in $K$.
\end{enumerate}
\end{lemma}
\begin{proof}
If $(a_\rho)$ is of $\Td$-algebraic type over $K$, then Fact~\ref{prop:Ka226.2} provides a proper immediate $\TdO$-extension of $K$ that is $\Td$-algebraic over~$K$. Conversely, if $a$ is an element in a proper immediate $\TdO$-extension that is $\Td$-algebraic over $K$, then any divergent pc-sequence in $K$ with pseudolimit $a$ is necessarily of $\Td$-algebraic type over $K$ by Fact~\ref{prop:Ka226.1}.
%Conversely, let $L$ be an immediate $\TdO$-extension of $K$ that is $\Td$-algebraic over $K$ and let $\ell \in L \setminus K$.
%Then we have a divergent pc-sequence $(a_\rho)$ in $K$ with $a_\rho \leadsto \ell$ and an $\cL(K)$-definable $G \colon K^{r} \to K$ such that $\ell^{(r)} = G(\jet^{r-1} \ell)$, so $r \geq 1$.
%Let $F \colon K^{1+r} \to K$ be the $\cL(K)$-definable function $Y_{r}-G(Y_0, \dots, Y_{r-1})$, so $F(\jet^{r}\ell)=0$. Then by~\cite[Proposition~5.4]{Ka22}, either $F \in Z_r(K,\ell)$ or there is $q< r$ such that $Z_q(K,\ell)\neq\0$, so $(a_\rho)$ is of $\Td$-algebraic type over~$K$.
\end{proof}

Now we show that the divergent pc-sequence constructed in Lemma~\ref{lem:Tdhcdivpc} is of $\Td$-algebraic type over~$K$.

\begin{lemma}\label{lem:pconvvanish}
If $Z_q(K, \ell) = \0$ for all $q<r$ and $F(\jet^r a_\rho) \leadsto 0$, then $F \in Z_r(K, \ell)$.
In particular, if $F(\jet^r a_\rho) \leadsto 0$, then $Z(K, \ell)\neq\0$.
\end{lemma}
\begin{proof}
Suppose that $Z_q(K, \ell) = \0$ for all $q<r$ and $F \not\in Z_r(K, \ell)$.
By \cite[Proposition~5.4]{Ka22}, $F(\jet^r \ell) \neq 0$ and $F(\jet^r a_\rho) \sim F(\jet^r \ell)$ for all sufficiently large $\rho$, and thus $F(\jet^r a_\rho) \not\leadsto 0$.
\end{proof}

In light of Lemmas~\ref{lem:TdalgmaxTdalgtype} and~\ref{lem:pconvvanish}, Lemma~\ref{lem:Tdhcdivpc} yields the following refinement of Corollary~\ref{cor:sphcompTdh}.
\begin{corollary}\label{cor:TdalgmaxTdh}
If $\k$ is linearly surjective and $K$ is $\Td$-algebraically maximal, then $K$ is $\Td$-henselian.
\end{corollary}

%\begin{corollary}\label{cor:Zmindiffequiv}
%Suppose that $K$ has linear approximations along pc-sequences.
%The following are equivalent:
%\begin{enumerate}
%	\item $Z_q(K, \ell) = \0$ for all $q<r$ and $F \in Z_r(K, \ell)$;
%	\item $F$ is a minimal $\Ld$-function of $(a_\rho)$ over~$K$.
%	\item $(a_\rho)$ has a minimal $\Ld$-function over~$K$.
%	\item there is an $\cL(K)$-definable $G \colon K^{1+r}\to K$ in implicit form such that $G(\jet^r b_\sigma) \leadsto 0$ for some pc-sequence $(b_\sigma)$ in $K$ equivalent to $(a_\rho)$.
%\end{enumerate}
%\end{corollary}

%------------------------------------------------------------------------------%
\subsection{More on behavior along pc-sequences}\label{sec:Tdhc}
%------------------------------------------------------------------------------%

In this subsection, we prove some technical lemmas regarding behavior along pc-sequences. For the remainder of this subsection, let $(a_\rho)$ be a divergent pc-sequence in $K$, let $L$ be a $\TdO$-extension of $K$, and let $\ell \in L$ be a pseudolimit of $(a_\rho)$. Set $\gamma_{\rho} \coloneqq v(a_{\rho+1}-a_{\rho})$, and set $B_\rho \coloneqq B(a_{\rho+1},\gamma_\rho)$. In this context, \emph{eventually} means for all sufficiently large $\rho$.
In our definition of linear approximation, we quantify over (pairs of) elements of $K$, so linear approximation in $K$ need not imply linear approximation in $L$. Thus, we will often need to work in $L$ instead of in $K$; when we do this, we implicitly identify $F$ with its natural extension $F^L \colon L^{1+r} \to L$ and $A$ with the obvious element of $L[\der]$.
The next lemma is an analogue of \cite[Lemma~6.8.1]{ADH17}. 
%, in which we can relax $F$ to being $\cL(L)$-definable.
\begin{lemma}\label{lem:6.8.1}
Suppose that the derivation of $\k$ is nontrivial and $A$ linearly approximates $F$ on $B_{\rho}^L$, eventually.
Then there is a pc-sequence $(b_\rho)$ in $K$ equivalent to $(a_\rho)$ such that $F(\jet^r b_\rho) \leadsto F(\jet^r \ell)$.
\end{lemma}
\begin{proof}
By removing some initial terms of the sequence, we can assume that for all $\rho$, $A$ linearly approximates $F$ on $B_{\rho}^L$ and $\gamma_\rho = v(\ell-a_\rho)$, and also that $\gamma_\rho$ is strictly increasing as a function of $\rho$.
Take $g_\rho \in K$ and $y_\rho \in \cO_L^{\times}$ as in the proof of \cite[Lemma~6.8.1]{ADH17} (with $\ell$, $F$, and $A$ replacing $a$, $G$, and $P$, respectively) so that $vg_\rho = \gamma_\rho$, $b_\rho \coloneqq \ell + g_{\rho}y_{\rho} \in K$, and $vA(g_\rho y_\rho) = v_A(\gamma_\rho)$.
Then $(b_\rho)$ is a pc-sequence equivalent to $(a_\rho)$ and
\[
F(\jet^r b_\rho) - F(\jet^r \ell)\ =\ A(g_\rho y_\rho) + \varepsilon_\rho,
\]
where $\varepsilon_\rho \in L$ with $v\varepsilon_\rho > vA_{\times g_\rho y_\rho} = v_A(\gamma_\rho)$.
Since $\gamma_\rho$ is strictly increasing and $vA(g_\rho y_\rho) = v_A(\gamma_\rho)$, we have $F(\jet^r b_\rho) \leadsto F(\jet^r \ell)$, as desired.
\end{proof}

We say that $(F, A, (a_\rho))$ is in \textbf{$\Td$-hensel configuration} if there is an index $\rho_0$ such that $(F, A, a_{\rho'}, \gamma_\rho)$ is in $\Td$-hensel configuration for all $\rho'>\rho\geq\rho_0$.
Tacitly, this is relative to $K$; if this also holds relative to $L$, then we say that  $(F, A, (a_\rho))$ is in \textbf{$\Td$-hensel configuration in $L$}. Note that $(F, A, (a_\rho))$ is in $\Td$-hensel configuration  in $L$ if and only if $(F, A, (a_\rho))$ is in $\Td$-hensel configuration and $A$ linearly approximates $F$ on $B^L_\rho$, eventually. 
%We most often use this when $L$ is an immediate $\TdO$-extension of $K$ containing a pseudolimit of $(a_\rho)$.
Modulo passing to cofinal subsequences, being in $\Td$-hensel configuration is preserved by equivalence of pc-sequences:
\begin{lemma}\label{lem:Tdhcequivpcseq}
Suppose that $(F, A, (a_\rho))$ is in $\Td$-hensel configuration in $L$ and let $(b_{\sigma})$ be a pc-sequence in $K$ equivalent to $(a_\rho)$.
There exists a cofinal subsequence $(b_{\lambda})$ of $(b_{\sigma})$ such that $(F, A, (b_{\lambda}))$ is in $\Td$-hensel configuration in $L$.
\end{lemma}
\begin{proof}
% This proof makes repeated use of \cite[Lemma~2.2.17]{ADH17} characterizing equivalent pc-sequences.
By discarding some initial terms, we can assume that $\gamma_\rho$ and $\delta_\sigma \coloneqq v(b_{\sigma+1}-b_{\sigma})$ are strictly increasing.

% Suppose we have a nonzero ordinal $\lambda$ and a sequence $(a_\rho)_{\rho<\lambda}$ in $B(a,\gamma)$ satisfying
% \begin{enumerate}
% \item\label{divpcseq1} $a_0 = a$ and $(F,A,a_\rho,\gamma)$ is in $\Td$-hensel configuration for each $\rho<\lambda$;
% \item\label{divpcseq2} $v_A(\gamma_\rho)= vF(\jet^r a_\rho)$ whenever $\rho+1<\lambda$, where $\gamma_\rho\coloneqq v(a_{\rho+1}-a_\rho)$;
% \item\label{divpcseq3} $vF(\jet^r a_\rho)$ is strictly increasing as a function of $\rho$, for $\rho<\lambda$.
% \end{enumerate}
% Such a sequence exists when $\lambda=1$, so it suffices to extend $(a_\rho)_{\rho<\lambda}$ to a sequence $(a_\rho)_{\rho<\lambda+1}$ in $B(a, \gamma)$ satisfying \ref{divpcseq1}--\ref{divpcseq3} with $\lambda+1$ in place of $\lambda$.

First, take $\rho_0$ and $\sigma_0$ large enough that $(F, A, a_{\rho'}, \gamma_{\rho})$ is in $\Td$-hensel configuration in $L$ for all $\rho'>\rho \geq \rho_0$ and $v(a_\rho-b_\sigma)>\gamma_{\rho_0}$ for all $\rho>\rho_0$ and $\sigma>\sigma_0$.
By increasing $\sigma_0$, we can arrange that $\delta_{\sigma_0} \geq \gamma_{\rho_0}$.
Then for any $\rho>\rho_0$ and $\sigma>\sigma_0$, $A$ linearly approximates $F$ on 
$B^L_{\rho_0}\supseteq B^L_{\sigma_0}$.
%$B(a_{\rho}, \gamma_{\rho_0}) = B(b_{\sigma}, \gamma_{\rho_0}) \supseteq B(b_{\sigma}, \delta_{\sigma_0})$.

Second, take $\sigma_1 > \sigma_0$ and $\rho_1 > \rho_0$ sufficiently large that $v(b_{\sigma}-a_{\rho})>\delta_{\sigma_1}$ for all $\sigma>\sigma_1$ and $\rho>\rho_1$.
Let $\sigma>\sigma_1$.
To show that $(F, A, b_{\sigma}, \delta_{\sigma_1})$ is in $\Td$-hensel configuration in $L$, suppose towards a contradiction that $vF(\jet^{r}b_{\sigma}) \leq v_A(\delta_{\sigma_1})$.
By increasing $\rho_1$, we can assume that $\gamma_{\rho_1}>\delta_{\sigma_1}$.
Then for $\rho>\rho_1$ we have $vF(\jet^{r}a_{\rho})>v_A(\gamma_{\rho_1})>v_A(\delta_{\sigma_1})$ and $vA(b_\sigma-a_\rho) \geq vA_{\times(b_\sigma-a_\rho)}>v_A(\delta_{\sigma_1})$, and thus
\[
v\big( F(\jet^{r}b_{\sigma})-F(\jet^{r}a_{\rho}) - A(b_{\sigma}-a_{\rho}) \big)\ =\ vF(\jet^{r}b_{\sigma})\ \leq\ v_A(\delta_{\sigma_1})\ <\ vA_{\times (b_{\sigma}-a_{\rho})},
\]
contradicting that $A$ linearly approximates $F$ on $B^L_{\rho_0}$.
Iterating this second step yields a cofinal subsequence of $(b_{\sigma})$ with the desired property.
\end{proof}

If $(F, A, (a_\rho))$ is in $\Td$-hensel configuration in $L$, then by definition, $A$ linearly approximates $F$ on $B_{\rho}^L$, eventually. The converse holds, under the assumption that $F(\jet^ra_\rho)\leadsto 0$.

\begin{lemma}\label{lem:Tdhc}
Suppose that $F(\jet^ra_\rho)\leadsto 0$ and that $A$ linearly approximates $F$ on $B_{\rho}^{L}$, eventually. Then $(F, A, (a_\rho))$ is in $\Td$-hensel configuration in $L$. 
\end{lemma}
\begin{proof}
Take $\rho_0$ such that $A$ linearly approximates $F$ on $B_{\rho_0}^L$. By increasing $\rho_0$, we may assume that $vF(\jet^ra_\rho)$ is strictly increasing for $\rho>\rho_0$. Let $\rho' > \rho> \rho_0$ be given. We need to show that $(F,A,a_{\rho'},\gamma_\rho)$ is in $\Td$-hensel configuration in $L$. Since $A$ linearly approximates $F$ on $B_\rho^L$, it suffices to show that $vF(\jet^ra_{\rho'}) > v_A(\gamma_\rho)$. Since $a_{\rho'}$ and $a_\rho$ are in $B_{\rho_0}$, we have
\[
v\big(F(\jet^ra_{\rho'}) - F(\jet^ra_\rho) - A(a_{\rho'}-a_\rho)\big) \ > \ vA_{\times (a_{\rho'}-a_\rho)}\ =\ v_A(\gamma_\rho). 
\]
Since $vA(a_{\rho'}- a_\rho) \geq v_A(\gamma_\rho)$, we must have $v\big(F(\jet^ra_{\rho'}) - F(\jet^ra_\rho)\big) \geq v_A(\gamma_\rho)$ as well. This gives $vF(\jet^ra_\rho)\geq v_A(\gamma_\rho)$, since $F(\jet^ra_{\rho'}) \prec F(\jet^ra_\rho)$, so $vF(\jet^ra_{\rho'}) > v_A(\gamma_\rho)$, as desired.
\end{proof}

The next lemma shows that $\Td$-hensel configuration in the presence of $\Td$-henselianity allows us to find a pseudolimit that is also a zero of the definable function, a key step towards our results.
\begin{lemma}\label{lem:mindiffzero}
Suppose that $L$ is $\Td$-henselian and that $(F, A, (a_\rho))$ is in $\Td$-hensel configuration in $L$. Then there exists $b \in L$ such that $a_\rho \leadsto b$ and $F(\jet^rb)=0$.
\end{lemma}
\begin{proof}
Suppose that $(F, A, a_{\rho'}, \gamma_\rho)$ is in $\Td$-hensel configuration in $L$ for all $\rho'>\rho\geq\rho_0$.
By increasing $\rho_0$, we can assume that $v(\ell-a_\rho) = \gamma_\rho$ for all $\rho>\rho_0$ and $\gamma_\rho$ is strictly increasing as a function of $\rho>\rho_0$.
For $\rho>\rho_0$, an argument similar to one in the previous proof shows that $(F, A, \ell, \gamma_\rho)$ is in $\Td$-hensel configuration in $L$.
% Let $\rho>\rho_0$.
% To show that $(F, A, \ell, \gamma_\rho)$ is in $\Td$-hensel configuration in $L$, suppose towards a contradiction that $vF(\jet^r \ell)\leq v_A(\gamma_\rho)$.
% For $\sigma>\rho$, we have $vF(\jet^r a_\sigma)>v_A(\gamma_{\rho})$ and $vA(\ell-a_\sigma)\geq v_A(\gamma_\sigma)>v_A(\gamma_{\rho})$, and thus
% \[
% v\big(F(\jet^r \ell) - F(\jet^r a_\sigma) - A(\ell-a_\sigma)\big)\ =\ vF(\jet^r \ell)\ \leq\ v_A(\gamma_\rho)\ <\ v_A(\gamma_\sigma),
% \]
% contradicting that $A$ linearly approximates $F$ on $B_{\rho}^L$.
% Hence $(F, A, \ell, \gamma_\rho)$ is in $\Td$-hensel configuration in $L$, so we have $b \in L$ with $F(\jet^r b)=0$ and $vA_{\times (b-\ell)} \geq vF(\jet^r \ell)$.
Hence $\Td$-henselianity yields $b \in L$ with $F(\jet^r b)=0$ and $vA_{\times (b-\ell)} \geq vF(\jet^r \ell)$.
From $vF(\jet^r \ell) > v_A(\gamma_\rho)$ for all $\rho>\rho_0$, we get $v(b-\ell)>\gamma_\rho$ for all $\rho>\rho_0$, and thus $a_\rho \leadsto b$.
\end{proof}

%------------------------------------------------------------------------------%
\subsection{Uniqueness and the \texorpdfstring{$T^\der$}{T∂}-hensel configuration property}\label{sec:sphcompunique}
%------------------------------------------------------------------------------%

In this subsection, we assume that the derivation of $\k$ is nontrivial.
We now introduce the key property needed for our theorem on the uniqueness of spherically complete immediate $\TdO$-extensions of monotone $\TdO$-models. We say that $K$ has the \textbf{$\Td$-hensel configuration property} if whenever we have
\begin{enumerate}
\item a divergent pc-sequence $(a_\rho)$ in $K$ of $\Td$-algebraic type over $K$,
\item a minimal $\Ld$-function $F$ of $(a_\rho)$ over $K$, and
\item an immediate $\TdO$-extension $L$ of $K$ containing a pseudolimit of $(a_\rho)$,
\end{enumerate}
there is an $A \in K[\der]^{\neq}$ that linearly approximates $F$ on $B(a_{\rho+1},\gamma_\rho)^{L}$ for all sufficiently large $\rho$, where $\gamma_\rho\coloneqq v(a_{\rho+1}-a_\rho)$. The $\Td$-hensel configuration property is an analogue of the differential-henselian configuration property for valued differential fields with small derivation that was implicitly used in \cite[Chapter~7]{ADH17} and explicitly introduced in \cite{DPC19}. In this subsection, we will show how uniqueness follows from the $\Td$-hensel configuration property, without assuming monotonicity.
Then we show in Section~\ref{sec:monotoneTdhc} that monotone $\TdO$-models with nontrivial induced derivation on their differential residue fields have the $\Td$-hensel configuration property.
First, we use the $\Td$-hensel configuration property to give an alternative description of minimal $\Ld$-functions.

\begin{lemma}\label{lem:vanishpconv}
Suppose that $K$ has the $\Td$-hensel configuration property, and let $(a_\rho)$ be a divergent pc-sequence in $K$. Then the following are equivalent:
\begin{enumerate}
\item $(a_\rho)$ is of $\Td$-algebraic type over $K$ and $F$ is a minimal $\Ld$-function for $(a_\rho)$ over $K$. 
\item $F(\jet^r b_\sigma) \leadsto 0$ for some pc-sequence $(b_\sigma)$ in $K$ equivalent to $(a_\rho)$, and $G(\jet^q b_\sigma) \not\leadsto 0$ for $q<r$, every $\cL(K)$-definable function $G \colon K^{1+q} \to K$ in implicit form, and every pc-sequence $(b_\sigma)$ in $K$ equivalent to $(a_\rho)$.
\end{enumerate}
\end{lemma}
\begin{proof}
Suppose that $F$ is a minimal $\Ld$-function for $(a_\rho)$ over $K$. Fact~\ref{prop:Ka226.2} yields an immediate $\TdO$-extension $K\llangle a \rrangle$ of $K$ such that $a_\rho \leadsto a$ and $F(\jet^r a)=0$. As $K$ has the $\Td$-hensel configuration property, Lemma~\ref{lem:6.8.1} provides a pc-sequence $(b_\rho)$ in $K$ equivalent to $(a_\rho)$ such that $F(\jet^r b_\rho) \leadsto 0$.
For each $q<r$, we have $Z_q(K,a) = \0$, so $G(\jet^q b_\sigma) \not\leadsto 0$ for any $\cL(K)$-definable function $G \colon K^{1+q} \to K$ in implicit form and any pc-sequence $(b_\sigma)$ in $K$ equivalent to $(a_\rho)$ by Lemma~\ref{lem:pconvvanish}. 

Now suppose that $F$ is not a minimal $\Ld$-function for $(a_\rho)$ over $K$, and fix a pseudolimit $\ell$ of $(a_\rho)$ in some $\TdO$-extension of $K$. If $G \in Z_q(K,\ell)$ is a minimal $\Ld$-function for $(a_\rho)$ over $K$ for some $q<r$, then we have $G(\jet^q b_\sigma) \leadsto 0$ for some pc-sequence $(b_\sigma)$ in $K$ equivalent to $(a_\rho)$ by the first part of this proof. If $Z_q(K,\ell) =\0$ for all $q<r$, then $F \not\in Z_r(K,\ell)$, so $F(\jet^r b_\sigma) \not\leadsto 0$ for any pc-sequence $(b_\sigma)$ in $K$ equivalent to $(a_\rho)$ by Lemma~\ref{lem:pconvvanish}.
\end{proof}

\begin{assumption}
For the rest of this section, suppose that every immediate $\TdO$-extension of $K$ has the $\Td$-hensel configuration property.
\end{assumption}

\begin{theorem}\label{thm:sphcompunique}
Suppose that $\k$ is linearly surjective.
Any two spherically complete immediate $\TdO$-extensions of $K$ are $\LdO$-isomorphic over $K$.
Any two $\Td$-algebraically maximal immediate $\TdO$-extensions of $K$ that are $\Td$-algebraic over $K$ are $\LdO$-isomorphic over~$K$.
\end{theorem}
\begin{proof}
Let $L_0$ and $L_1$ be spherically complete immediate $\TdO$-extensions of $K$.
Let $\mu \colon K_0 \to K_1$ be a maximal $\LdO$-isomorphism between $\TdO$-extensions $K_0 \subseteq L_0$ and $K_1 \subseteq L_1$ of $K$. %; note that $K$ retains the $\Td$-hensel configuration property by our assumption above.
For convenience, we identify $K_0$ and $K_1$ via $\mu$, so $\mu$ becomes the identity, and assume moreover that $K_0=K_1=K$.
Suppose that $K \neq L_0$ (equivalently, $K \neq L_1$), so we have $\ell \in L_0 \setminus K$ and a divergent pc-sequence $(a_\rho)$ in $K$ with $a_\rho \leadsto \ell$.
%We distinguish the cases $Z(K, \ell) = \0$ and $Z(K, \ell) \neq \0$.

If $(a_\rho)$ is $\Td$-transcendental, then we can take $b \in L_1$ with $a_\rho \leadsto b$ and extend $\mu$ to an $\LdO$-isomorphism $K\llangle \ell \rrangle \to K\llangle b \rrangle$ sending $\ell$ to $b$ by Fact~\ref{prop:Ka226.1}, contradicting the maximality of~$\mu$.

Now suppose that $(a_\rho)$ is $\Td$-algebraic, and let $F$ be a minimal $\Ld$-function of $(a_\rho)$ over $K$. Using Lemma~\ref{lem:vanishpconv}, we replace $(a_\rho)$ by an equivalent pc-sequence in $K$ to arrange that $F(\jet^r a_\rho) \leadsto 0$.
By assumption $K$ has the $\Td$-hensel configuration property, so by Lemma~\ref{lem:Tdhc}, we have $A_0, A_1 \in K[\der]$ such that $(F, A_0, (a_\rho))$ is in $\Td$-hensel configuration in $L_0$ and $(F, A_1, (a_\rho))$ is in $\Td$-hensel configuration in $L_1$.
By Corollary~\ref{cor:sphcompTdh} and Lemma~\ref{lem:mindiffzero}, we have $b_0 \in L_0 \setminus K$ and $b_1 \in L_1 \setminus K$ such that $a_\rho \leadsto b_0$, $a_\rho \leadsto b_1$, $F(\jet^r b_0)=0$, and $F(\jet^r b_1)=0$.
Then Fact~\ref{prop:Ka226.2} yields an extension of $\mu$ to an $\LdO$-isomorphism $K\llangle b_0\rrangle \to K\llangle b_1\rrangle$, contradicting the maximality of~$\mu$.

The proof of the second statement is similar but it uses Lemma~\ref{lem:TdalgmaxTdalgtype} %(and its proof)
 and also Corollary~\ref{cor:TdalgmaxTdh} replaces Corollary~\ref{cor:sphcompTdh}.
\end{proof}

In the case of few constants, we have two additional results and an easy corollary.
The first is a converse to Corollary~\ref{cor:TdalgmaxTdh}, and it does not need the assumption that proper immediate $\TdO$-extensions of $K$ have the $\Td$-hensel configuration property, only that $K$ itself does.

\begin{theorem}\label{thm:TdhTdalgmax}
Suppose that $K$ is $\Td$-henselian and $C \subseteq \cO$.
Then $K$ is $\Td$-algebraically maximal.
\end{theorem}
\begin{proof}
Let $(a_\rho)$ be a pc-sequence in $K$ of $\Td$-algebraic type over $K$ with minimal $\Ld$-function $F$ over $K$. By Lemma~\ref{lem:TdalgmaxTdalgtype}, it suffices to show that $(a_\rho)$ has a pseudolimit in $K$. Assume towards a contradiction that $(a_\rho)$ is divergent. Then we may replace $(a_\rho)$ by an equivalent pc-sequence in $K$ to arrange that $F(\jet^r a_\rho) \leadsto 0$ by Lemma~\ref{lem:vanishpconv}.

By assumption, $K$ has the $\Td$-hensel configuration property, so by Lemma~\ref{lem:Tdhc} we have $A \in K[\der]^{\neq}$ of order $q$ such that $(F,A,(a_\rho))$ is in $\Td$-hensel configuration in $K$.
By removing some initial terms of the sequence, we arrange that $\gamma_\rho \coloneqq v(a_{\rho+1}-a_\rho)$ is strictly increasing as a function of $\rho$, $\gamma_\rho = v(a-a_\rho)$, and $(F,A,a_{\rho'},\gamma_\rho)$ is in $\Td$-hensel configuration for all $\rho'>\rho$.
By $\Td$-henselianity take for each $\rho$ a $z_\rho \in B(a_{\rho+1},\gamma_\rho)$ with $F(\jet^r z_\rho)=0$. %and $vA_{\times (z_{\rho}-a_{\rho+1})} \geq vF(\jet^{r}a_{\rho+1})$.
Then $v(a-z_{\rho})>\gamma_\rho$ and $v(z_{\rho}-a_{\rho})=\gamma_\rho$.
Since $(\gamma_\rho)$ is cofinal in $v(a-K)$, we can take indices $\rho_0<\dots<\rho_{q+2}$ such that $a-z_{\rho_j} \prec a-z_{\rho_i}$ whenever $0 \leq i<j \leq q+2$.
Set $y_i \coloneqq z_{\rho_{i+1}}$ for $i=0, \dots, q+1$.
If $1 \leq i \leq q$, then
\[
y_i-y_{i-1}\ \sim\ a-y_{i-1}\ \succ\ a-y_{i}\ \sim\ y_{i+1}-y_{i},
\]
so conditions \ref{lem:7.5.5i} and \ref{lem:7.5.5ii} from Lemma~\ref{lem:7.5.5} are satisfied.
Letting $\gamma \coloneqq \gamma_{\rho_0}$, we will reach a contradiction with that lemma by showing that $(F, A, y_{q+1}, \gamma)$ is in $\Td$-hensel configuration and $v(y_0-y_{q+1})>\gamma$.
The latter holds because \[v(a-y_{q+1})\ >\ v(a-y_0)\ >\ \gamma_{\rho_1}\ >\ \gamma.\]
We also have $B(y_{q+1}, \gamma) = B(y_0, \gamma) = B(a_{\rho_1}, \gamma)$, so since $(F, A, a_{\rho_1}, \gamma)$ is in $\Td$-hensel configuration, $A$ linearly approximates $F$ on $B(y_{q+1}, \gamma)$.
It remains to note that $vF(\jet^r y_{q+1})=\infty>v_A(\gamma)$.
\end{proof}

The previous result fails without the assumption $C \subseteq \cO$, as the next example shows.
\begin{example}
We build an increasing sequence $(\Gamma_n)$ of divisible ordered abelian groups as follows:\ set $\Gamma_0 \coloneqq \{0\}$ and given $\Gamma_n$, set $\Gamma_{n+1} \coloneqq \Gamma_n \oplus \Q\gamma_{n}$, where $\gamma_n$ is a new element greater than $\Gamma_n$.
Now take a linearly surjective $\k \models \Td_{\an}$.
From $\k$ and $(\Gamma_n)$ we obtain the ordered Hahn field $\k\llp t^{\Gamma_n}\rrp$ with $0<t<\k^>$.
As in the introduction, we expand $\k\llp t^{\Gamma_n}\rrp$ to a model $E_n \coloneqq \k\llp t^{\Gamma_n}\rrp_{\an,c} \models \TdO_{\an}$ with $c$ the zero map.
That is, we expand $\k\llp t^{\Gamma_n}\rrp$ to a model of $T_{\an}$ by Taylor expansion; equip it with the convex hull of $\k$, which is a $T_{\an}$-convex valuation ring; and equip it with the derivation given by $\der(\sum_{\gamma} f_{\gamma}t^{\gamma}) = \sum_{\gamma} \der(f_{\gamma})t^{\gamma}$, which is a $\Td_{\an}$-derivation by \cite[Proposition~3.14]{Ka21}.
Identifying $E_n$ with an $\LdO$-substructure of $E_{n+1}$ in the obvious way, we get an increasing sequence $(E_n)$ of $\TdO_{\an}$-models.
Then $E \coloneqq \bigcup_n E_n \models \TdO_{\an}$ by \cite[Corollary~3.16]{Ka21}.
Note that $E$ is $\Td$-henselian since each $E_n$ is by Corollary~\ref{cor:sphcompTdh}.
Now set $\Gamma \coloneqq \bigcup_n \Gamma_n$ and $L \coloneqq \k\llp t^{\Gamma} \rrp_{\an,c}$ as before, so $E\subseteq L$.
We have $\sum_n t^{\gamma_n} \in L \setminus E$ with $\der(\sum_n t^{\gamma_n})=0$, so $E$ is not $\Td_{\an}$-algebraically maximal.
\end{example}

Recall that any $\Td$-henselian $K$ with $C \subseteq \cO$ is asymptotic (see Section~\ref{sec:Tdhbasic}), so Theorem~\ref{thm:TdhTdalgmax} is really about asymptotic $\TdO$-models, as are the next results about minimal $\Td$-henselian extensions.
An immediate $\TdO$-extension $L$ of $K$ is a \textbf{$\Td$-henselization of $K$} if $L$ is $\Td$-henselian and for every immediate $\Td$-henselian $\TdO$-extension $M$ of $K$, there is an $\LdO$-embedding $L \to M$ that is the identity on~$K$.

\begin{theorem}\label{thm:Tdhenselization}
Suppose that $\k$ is linearly surjective and $K$ is asymptotic.
Then $K$ has a $\Td$-henselization that is $\Td$-algebraic over $K$ and has no proper $\Td$-henselian $\LdO$-substructure containing $K$.
In particular, any two $\Td$-henselizations of $K$ are isomorphic over~$K$.
\end{theorem}
\begin{proof}
Let $L$ be a $\Td$-algebraically maximal immediate $\TdO$-extension of $K$ that is $\Td$-algebraic over $K$.
Then $L$ is $\Td$-henselian by Corollary~\ref{cor:TdalgmaxTdh} and asymptotic by \cite[Lemma~1.12]{ADH18} (or \cite[Lemmas~9.4.2 and 9.4.5]{ADH17}).
In particular, $C_L \subseteq \cO_L$, and thus by Theorem~\ref{thm:TdhTdalgmax} no proper $\LdO$-substructure of $L$ containing $K$ is $\Td$-henselian.
Now let $M$ be an immediate $\Td$-henselian $\TdO$-extension of $K$, so $M$ is asymptotic and hence $\Td$-algebraically maximal by Theorem~\ref{thm:TdhTdalgmax}, making Lemma~\ref{lem:mindiffzero} available.
To see that there is an $\LdO$-embedding of $L$ into $M$ over $K$, argue as in the proof of Theorem~\ref{thm:sphcompunique}.
\end{proof}

\begin{corollary}
Suppose that $\k$ is linearly surjective and $K$ is asymptotic.
Any immediate $\Td$-henselian $\TdO$-extension of $K$ that is $\Td$-algebraic over $K$ is a $\Td$-henselization of~$K$.
\end{corollary}
\begin{proof}
Let $M$ be an immediate $\Td$-henselian $\TdO$-extension of $K$ that is $\Td$-algebraic over $K$ and let $L$ be a $\Td$-henselization of $K$.
Then the $\LdO$-embedding $L \to M$ over $K$ is surjective by Theorem~\ref{thm:TdhTdalgmax}.
\end{proof}

%------------------------------------------------------------------------------%
\section{Monotone fields}\label{sec:monotone}
%------------------------------------------------------------------------------%

\subsection{\texorpdfstring{$T^\der$}{T∂}-hensel configuration property for monotone fields}\label{sec:monotoneTdhc}
In this subsection, we assume that $K$ is monotone and the derivation of $\k$ is nontrivial. Let $(a_\rho)$ be a divergent pc-sequence in $K$, and let $\ell$ be a pseudolimit of $(a_\rho)$ in a monotone $\TdO$-extension $M$ of $K$. Let $F\colon K^{1+r} \to K$ be an $\cL(K)$-definable function in implicit form, and assume that $Z_q(K,\ell) = \emptyset$ for all $q < r$. For each $\rho$, set $\gamma_\rho\coloneqq v(a_{\rho+1}-a_\rho)$ and set $B_\rho\coloneqq B(a_{\rho+1},\gamma_\rho)$. We assume that $\gamma_\rho$ is strictly increasing as a function of $\rho$, so $B_{\rho'} \subsetneq B_\rho$ for $\rho'> \rho$. 

\begin{proposition}\label{prop:monotoneapprox}
There is an $A \in K[\der]^{\neq}$ and an index $\rho_0$ such that $A$ linearly approximates $F$ on $B_{\rho_0}^M$.
\end{proposition}
\begin{proof}
Note that if $A$ linearly approximates $\fm_F\inv F$ on an open $v$-ball $B$, then $\fm_FA$ linearly approximates $F$ on $B$, so we may assume that $\fm_F = 1$.
By applying Fact~\ref{fact:rvapprox} to the function $I_F$, we find an $\LRVeq(K)$-definable map $\chi\colon K^r \to \oRV{K}^{\eq}$ such that for each $s \in \chi(K^r)$, if $\chi\inv(s)$ contains an open $v$-ball, then either $I_F$ is constant on $\chi\inv(s)$ or there is $d \in K^r$ such that
\begin{equation}\label{eq:monotoneapprox}
v\big(I_F(x)- I_F(y) - d\cdot (x-y)\big) \ >\ vd+ v(x-y)
\end{equation}
for all $x, y \in \chi\inv(s)$ with $x \neq y$.
By \cite[Lemma~5.5]{Ka22}, $L \coloneqq K\langle\jet^{r-1}\ell\rangle \subseteq M$ is an immediate $\TO$-extension of $K$.
Let $\chi^L$ and $\chi^M$ denote the natural extensions of $\chi$ to $L^r$ and $M^r$, respectively, and let 
\[
s_0\ \coloneqq\ \chi^L(\jet^{r-1} \ell)\ \in\ \oRV{L}^{\eq}\ =\ \oRV{K}^{\eq}.
\]
Let $U\coloneqq \chi\inv(s_0) \subseteq K^r$. Then $U$ is $\LO(K)$-definable by~\cite[Corollary 2.18]{Yi17}.
% Let $\chi^M$ denote the natural extension of $\chi$ to $M^r$, and let 
% \[
% s_0\ \coloneqq\ \chi^M(\jet^{r-1} \ell)\ \in\ \oRV{M}^{\eq}\ =\ \oRV{K}^{\eq}.
% \]
% Let $U\coloneqq \chi\inv(s_0) \subseteq K^r$. Then $U$ is $\LO(K)$-definable by~\cite[Corollary 2.18]{Yi17}.

Since $\jet^{r-1}(\ell) \in U^M$, we can apply~\cite[Lemma 5.6]{Ka22} to get that $U$ has nonempty interior and that $\jet^{r-1}(y) \in U$ for all $y \in M$ sufficiently close to $\ell$. Take $\rho_0$ such that $\jet^{r-1}(y) \in U^M$ whenever $v(\ell-y) > \gamma_{\rho_0}$, so $B_{\rho_0}^M \subseteq U^M$. We choose the linear differential operator $A$ as follows:\ If $I_F$ is constant on $U$, then we let $A$ be $\der^r \in K[\der]$. If $I_F$ is not constant on $U$, then we let $A$ be 
\[
\der^r - d_r \der^{r-1} - \cdots-d_1\ \in \ K[\der],
\]
where $d = (d_1,\ldots,d_r) \in K^r$ is chosen such that (\ref{eq:monotoneapprox}) holds for $x,y \in U$ with $x \neq y$. We claim that $A$ linearly approximates $F$ on $B_{\rho_0}^M$. This is clear if $I_F$ is constant on $U$, for then $I_F$ is  constant on $U^M$ as well. Suppose that $I_F$ is not constant on $U$, and let $a,b \in B_{\rho_0}^M$ with $a \neq b$. Then
\[
v\big(F(\jet^rb)- F(\jet^r a) - A(b-a)\big) \ = \ v\big(I_F(\jet^{r-1}b)- I_F(\jet^{r-1}a)- d\cdot \jet^{r-1}(b-a)\big) \ >\ vd + v\big(\jet^{r-1}(b-a)\big),
\]
where the inequality holds since $M$ is an elementary $\TO$-extension of $K$. %Since $K$ is monotone, $M$ is as well by~\cite[Corollary 6.3.6]{ADH17}, so
We have $v\big(\jet^{r-1}(b-a)\big) = v(b-a)$ since $M$ is monotone. By~\cite[Corollary 4.5.4]{ADH17}, $vA_{\times (b-a)}= vA +v(b-a)$. Using that $vA \leq vd$, we get
\[
v\big(F(\jet^rb)- F(\jet^r a) - A(b-a)\big) \ >\ vd + v\big(\jet^{r-1}(b-a)\big)\ \geq\ vA + v(b-a)\ =\ vA_{\times (b-a)}. \qedhere
\]
\end{proof}

Note that every immediate extension of $K$ is monotone by \cite[Corollary~6.3.6]{ADH17}.
Requiring $M$ to be an immediate extension of $K$ in the previous result thus shows that every monotone $\TdO$-model with nontrivial induced derivation on its differential residue field has the $\Td$-hensel configuration property.
Combining this with the previous section yields the following results (for monotone $K$).
\begin{theorem}\label{thm:sphcompuniquemono}
Suppose that $\k$ is linearly surjective.
Any two spherically complete immediate $\TdO$-extensions of $K$ are $\LdO$-isomorphic over $K$.
Any two $\Td$-algebraically maximal immediate $\TdO$-extensions of $K$ that are $\Td$-algebraic over $K$ are $\LdO$-isomorphic over~$K$.
\end{theorem}
\begin{theorem}\label{thm:TdhTdalgmaxmono}
Suppose $\k$ is linearly surjective.
If $K$ is $\Td$-algebraically maximal, then $K$ is $\Td$-henselian.
If $K$ is $\Td$-henselian and $C\subseteq\cO$, then $K$ is $\Td$-algebraically maximal.
\end{theorem}
\begin{theorem}\label{thm:Tdhenselizationmono}
Suppose that $K$ is asymptotic, and $\k$ is linearly surjective.
Then $K$ has a $\Td$-henselization that is $\Td$-algebraic over $K$ and has no proper $\Td$-henselian $\LdO$-substructure containing $K$.
In particular, any two $\Td$-henselizations of $K$ are $\LdO$-isomorphic over~$K$.
\end{theorem}

\subsection{The \texorpdfstring{$T^\der$}{T∂}-hensel configuration property in a class of models}\label{sec:Tdhcclass}

In Section~\ref{sec:sphcompunique}, we introduced the $\Td$-hensel configuration property, and used it to establish results about spherically complete immediate extensions. This property referred to immediate $\TdO$-extensions of $K$, and now we generalize it by considering other kinds of extensions. 
Let $\cC$ be a class of $\TdO$-models.
We say that $(K, \cC)$ has the \textbf{$\Td$-hensel configuration property} if whenever we have:
\begin{enumerate}
\item a divergent pc-sequence $(a_\rho)$ in $K$ of $\Td$-algebraic type over $K$,
\item a minimal $\Ld$-function $F$ of $(a_\rho)$ over $K$, and
\item a $\TdO$-extension $L\in \cC$ of $K$ containing a pseudolimit of $(a_\rho)$,
\end{enumerate}
there is an $A \in K[\der]^{\neq}$ that linearly approximates $F$
on $B(a_{\rho+1},\gamma_\rho)^{L}$ for all sufficiently large $\rho$, where, again, $\gamma_\rho\coloneqq v(a_{\rho+1}-a_\rho)$. We say that $\cC$ has the \textbf{$\Td$-hensel configuration property} if for every $K \in \cC$ with nontrivial derivation on $\k$, $(K, \cC)$ has the $\Td$-hensel configuration property.
In these terms, the previous subsection established:
\begin{corollary}\label{cor:monotoneTdhc2}
The class of monotone $\TdO$-models has the $\Td$-hensel configuration property.
\end{corollary}

Theorem~\ref{thm:sphcompunique} only required the assumption that the class $\cI(K)$ of immediate extensions of $K$ has the $\Td$-hensel configuration property, but its proof yields a variant involving non-immediate extensions whose Corollary~\ref{cor:sphcompembed} is needed later in the proof of the Ax--Kochen/Ershov theorem.
\begin{theorem}\label{thm:Tdhcclassembed}
Suppose that $\k$ is linearly surjective, $\cC$ has the $\Td$-hensel configuration property, and $\cI(K) \subseteq \cC$.
Let $L \in \cC$ be a $\Td$-henselian $\TdO$-extension of $K$.
If $L$ is $|\Gamma|^+$-saturated, then any immediate $\TdO$-extension of $K$ can be embedded in $L$ over~$K$.
\end{theorem}
\begin{proof}
Argue as in the proof of Theorem~\ref{thm:sphcompunique}, but use saturation instead of spherical completeness to obtain pseudolimits in~$L$.
\end{proof}
\begin{corollary}\label{cor:sphcompembed}
Suppose that $K$ is monotone and $\k$ is linearly surjective, and let $L$ be a monotone $\Td$-henselian $\TdO$-extension of $K$.
If $L$ is $|\Gamma|^+$-saturated, then any immediate $\TdO$-extension of $K$ can be embedded in $L$ over~$K$.
\end{corollary}
\begin{proof}
Note that every immediate extension of $K$ is monotone by \cite[Corollary~6.3.6]{ADH17}.
\end{proof}

These definitions also lead to improvements of Theorem~\ref{thm:Tdhenselization} and Theorem~\ref{thm:Tdhenselizationmono}.
We call $L$ a \textbf{$\cC$-$\Td$-henselization of $K$} if $L \in \cC$ is an immediate $\TdO$-extension of $K$ that is $\Td$-henselian and for any $\Td$-henselian $\TdO$-extension $M \in \cC$ of $K$, there is an $\LdO$-embedding $L \to M$ that is the identity on $K$.
The proof of Theorem~\ref{thm:Tdhenselization} shows:
\begin{theorem}
Suppose that $\k$ is linearly surjective, $\cC$ has the $\Td$-hensel configuration property, $\cI(K) \subseteq \cC$, and every $M \in \cC$ is asymptotic.
Then $K$ has a $\cC$-$\Td$-henselization that is $\Td$-algebraic over $K$ and has no proper $\Td$-henselian $\LdO$-substructure containing $K$.
In particular, any two $\cC$-$\Td$-henselizations of $K$ are isomorphic over~$K$. 
\end{theorem}
\begin{corollary}
Suppose that $K$ is monotone and asymptotic, and $\k$ is linearly surjective.
Then the $\Td$-henselization of $K$ embeds into every monotone, asymptotic, $\Td$-henselian $\TdO$-extension of~$K$.
\end{corollary}
% \begin{proof}
% Let $\cC$ be the class of $\TdO$-extensions of $K$ that are monotone and asymptotic.
% \end{proof}

\subsection{The \texorpdfstring{$c$}{c}-map}\label{sec:cmap}
In this subsection, assume that $K$ is monotone.
A \textbf{section for $K$} is a $\Lambda$-vector space embedding $s\colon \Gamma\to K^>$ such that $v(s\gamma) = \gamma$ for all $\gamma \in \Gamma$. An \textbf{angular component for $K$} is a $\Lambda$-linear map $\ac\colon K^>\to \k^>$ such that $\ac(a) = \res(a)$ whenever $a \asymp 1$. We extend any angular component map $\ac$ to a map $K\to \k$ (also denoted by $\ac$) by setting $\ac(0)\coloneqq 0$ and $\ac(-a)\coloneqq -\ac(a)$ for $a \in K^>$. Given any section $s\colon\Gamma\to K^>$, the map $a \mapsto \res\!\big(a/s(va)\big)\colon K^>\to \k^>$ is an angular component for $K$, which we say is \textbf{induced by $s$}. 

\begin{lemma}\label{lem:sectionexists}
Let $A \subseteq K^>$ be a $\Lambda$-subspace such that $\{va:a \in A\} = \Gamma$. Then there is a section $s\colon \Gamma\to K^>$ with image contained in $A$.
\end{lemma}
\begin{proof}
The inclusion $\cO^{\times}\cap A\to A$ and the restricted valuation map $A\to \Gamma$ yield an exact sequence 
\[
1\to \cO^{\times}\cap A\to A\to \Gamma\to 0
\]
of $\Lambda$-vector spaces. This exact sequence splits, yielding a section $s\colon \Gamma\to A$ as claimed.
\end{proof}

\begin{corollary}\label{cor:acinduced}
Let $\ac$ be an angular component for $K$. Then $\ac$ is induced by a section for $K$.
\end{corollary}
\begin{proof}
Apply the previous lemma to the set $A \coloneqq \{a \in K^>:\ac(a) = 1\}$.
\end{proof}

Let $s\colon \Gamma \to K^>$ be a section for $K$. We define a map $c\colon \Gamma\to \k$ by setting $c(\gamma) \coloneqq \res\!\big(s(\gamma)^\dagger\big)$ for $\gamma \in \Gamma$. Since $(a^\lambda)^\dagger = \lambda a^\dagger$ for $a \in K^>$ by Fact~\ref{fact:powerderivative}, the map $c$ is $\Lambda$-linear. If $K$ has many constants, then by Lemma~\ref{lem:sectionexists} with $C^>$ in place of $A$, we can choose $s$ so that its image is contained in the constant field. In this case, $c$ is the zero map.

%If $K$ is monotone and $\Td$-henselian and $\k$ is a lift of its differential residue field, then we can additionally arrange $c(\Gamma) \cap \k^{\dagger} = \{0\}$. 

The argument above tells us that we can associate to any monotone $\TdO$-model $K$ a structure $(\k,\Gamma,c)$ where $\k\models \Td$, $\Gamma$ is an ordered $\Lambda$-vector space, and $c\colon \Gamma\to \k$ is $\Lambda$-linear. As a converse, we show below that any such structure $(\k,\Gamma,c)$ comes from a monotone $\TdO$-model. We need the following fact.

\begin{fact}[{\cite[Proposition 2.6]{Ka23}}]\label{fact:smallderivsim}
% Let $M$ be a $\TO$-extension of $K$, let $a \in M$ with $a\not\sim f$ for all $f \in K$, and let $F\colon K\to K$ be an $\cL(K)$-definable function. Then $\frac{\partial F}{\partial Y}(a) \preceq a\inv F(a)$.
Let $E$ be a $T$-convex valued field, let $M$ be a $\TO$-extension of $E$, let $a \in M$ with $a\not\sim f$ for all $f \in E$, and let $F\colon E\to E$ be an $\cL(E)$-definable function. Then $F'(a) \preceq a\inv F(a)$.
\end{fact}

\begin{proposition}\label{prop:fieldbuilding}
Let $\k\models \Td$, let $\Gamma$ be an ordered $\Lambda$-vector space, and let $c\colon \Gamma \to \k$ be a $\Lambda$-linear map. Then there is a %spherically complete 
monotone $\TdO$-model $K$ with differential residue field $\k$ and value group $\Gamma$ and a section $s\colon \Gamma \to K^>$ such that $\res\!\big(s(\gamma)^\dagger \big)=c(\gamma)$ for all $\gamma \in \Gamma$. 
\end{proposition}
\begin{proof}
%It suffices to find a monotone $\TdO$-model $K$ with the above properties, for then we can pass to a strict immediate $\TdO$-extension of $K$ which is spherically complete. 
Let $(\gamma_\alpha)_{\alpha<\beta}$ be a $\Lambda$-basis for $\Gamma$, so we may write each $\gamma \in \Gamma$ uniquely as a sum $\gamma = \sum_{\alpha<\beta}\lambda_\alpha\gamma_\alpha$ where each $\lambda_\alpha$ is in $\Lambda$ and only finitely many $\lambda_\alpha$ are nonzero. Let $K\coloneqq \k\langle(t_\alpha)_{\alpha<\beta}\rangle$ be a $\TO$-extension of $\k$, equipped with the convex hull of $\k$ as its $T$-convex valuation ring and ordered so that $t_\alpha >0$ and $v(t_\alpha) = \gamma_\alpha$ for each $\alpha$ (such an extension can be built by a transfinite construction, using repeated applications of~\cite[Lemma 2.3]{Ka23}). Then $K$ has residue field $\k$ and value group $\Gamma$, and the sequence $(t_\alpha)$ is necessarily $\cL(\k)$-independent. Let $s\colon\Gamma\to K^>$ be the section mapping $\gamma = \sum_{\alpha<\beta}\lambda_\alpha\gamma_\alpha\in \Gamma$ to $\prod_{\alpha<\beta}t_\alpha^{\lambda_\alpha}\in K^>$.

Using Fact~\ref{fact:transext}, we extend the $T$-derivation on $\k$ to a $T$-derivation on $K$ by setting $t_\alpha' = c(\gamma_{\alpha})t_\alpha$. Since $(t_\alpha^\lambda)^\dagger = \lambda t_\alpha^\dagger$ by Fact~\ref{fact:powerderivative}, we have $s(\gamma)^\dagger =c(\gamma) \in \k$ for all $\gamma \in \Gamma$. Thus, we need only check that $K$ is monotone. Each element of $K$ is of the form $G(a,t)$ where $G\colon K^{m+n}\to K$ is $\cL(\emptyset)$-definable, $a = (a_1,\ldots,a_m)$ is an $\cL(\emptyset)$-independent tuple from $\k$, and $t = (t_{\alpha_1},\ldots,t_{\alpha_n})$ for some distinct $\alpha_1,\ldots,\alpha_n<\beta$. We fix such an element $G(a,t)$, and we need to show that $G(a,t)' \preceq G(a,t)$. Viewing $G$ as a function of the variables $X_1,\ldots,X_m,Y_1,\ldots,Y_n$, we have
\[
G(a,t)'\ =\ \frac{\partial G}{\partial X_1}(a,t)a_1'+\cdots+\frac{\partial G}{\partial X_m}(a,t)a_m'+\frac{\partial G}{\partial Y_1}(a,t)t_{\alpha_1}'+\cdots+\frac{\partial G}{\partial Y_n}(a,t)t_{\alpha_n}'.
\]
We will show that $\frac{\partial G}{\partial X_i}(a,t)a_i',\frac{\partial G}{\partial Y_j}(a,t)t_{\alpha_j}' \preceq G(a,t)$ for all $i \in \{1,\ldots,m\}$ and $j \in \{1,\ldots,n\}$. By symmetry, it suffices to handle the case $i = j = 1$. We start with $\frac{\partial G}{\partial X_1}(a,t)a_1'$. Since $a_1' \preceq 1$, it suffices to show that $\frac{\partial G}{\partial X_1}(a,t)\preceq G(a,t)$. Set $E \coloneqq \dclL(a_2,\ldots,a_m,t)$, so $E$ is an $\LO$-substructure of $K$, and set $E_0\coloneqq \dclL(a_2,\ldots,a_m) \subseteq E$. Since $E_0$ is trivially valued and $a_1 \not\in E_0$, we have $a_1 \not\sim f$ for any $f \in E_0$. The Wilkie inequality gives $\res(E)= \res(E_0)$, so $a_1 \not\sim f$ for any $f \in E$. Applying Fact~\ref{fact:smallderivsim} with
% $E$ in place of $K$,
$a_1$ in place of $a$ and the function $x\mapsto G(x,a_2,\ldots,a_m,t)$ in place of $F$ gives 
\[
\frac{\partial G}{\partial X_1}(a,t)\ \preceq\ a_1\inv G(a,t)\ \asymp\ G(a,t),
\]
as desired. Next, we show that $\frac{\partial G}{\partial Y_1}(a,t)t_{\alpha_1}' \preceq G(a,t)$. This time, we set $E\coloneqq \dclL(a,t_{\alpha_2},\ldots,t_{\alpha_n})$, so $t_{\alpha_1}\not\asymp f$ for any $f \in E$. In particular, $t_{\alpha_1}\not\sim f$ for any $f \in E$, so Fact~\ref{fact:smallderivsim} (this time with $x\mapsto G(a,x,t_{\alpha_2},\ldots,t_{\alpha_n})$ in place of $F$) gives 
\[
\frac{\partial G}{\partial Y_1}(a,t)\ \preceq\ t_{\alpha_1}\inv G(a,t).
\]
Thus, $\frac{\partial G}{\partial Y_1}(a,t)t_{\alpha_1}'\preceq\ t_{\alpha_1}^\dagger G(a,t) \preceq G(a,t)$, since $t_{\alpha_1}^\dagger = c(\gamma_{\alpha_1})\preceq 1$. 
\end{proof}

\begin{lemma}\label{lem:acmap}
Let $K$ be a monotone $\TdO$-model, let $c\colon \Gamma\to \k$ be a $\Lambda$-linear map, and suppose that for each $\gamma \in \Gamma$, there is $a \in K^>$ with $va= \gamma$ and $\res(a^\dagger) = c(\gamma)$. Then there is a section $s\colon \Gamma\to K^>$ such that $\res\!\big(s(\gamma)^\dagger\big) = c(\gamma)$ for all $\gamma \in \Gamma$. The corresponding angular component map $\ac\colon K\to \k$ induced by $s$ satisfies the equality $\ac(a)^\dagger =\res(a^\dagger)- c(va)$ for all $a \in K^\times$.
\end{lemma}
\begin{proof}
Let $A \subseteq K^>$ be the set of all $a \in K^>$ with $\res(a^\dagger) = c(va)$. Then $A$ is a multiplicative $\Lambda$-subspace of $K^>$ and $\{va:a \in A\} = \Gamma$, so there is a section $s\colon \Gamma\to K^>$ with image contained in $A$ by Lemma~\ref{lem:sectionexists}. Let $\ac\colon K\to \k$ be the angular component map induced by $s$. Then for $a \in K^\times$, we have
\[
\ac(a)^\dagger\ =\ \res\!\big(a/s(va)\big)^\dagger\ =\ \res(a^\dagger - s(va)^\dagger)\ =\ \res(a^\dagger) - c(va).\qedhere
\]
\end{proof}

% \begin{lemma}
% Let $K$ be a monotone $\Td$-henselian $\TdO$-model. Then there is a section $s\colon \Gamma \to K^>$ such that $\res\!\big(s(\Gamma)^\dagger\big) \cap \k^{\dagger} = \{0\}$.
% \end{lemma}
% \begin{proof}
% Let
% \[
% A\ \coloneqq\ \big\{a \in K^>:\res(a^\dagger) = 0\text{ or }\res(b^\dagger)\neq 0\text{ for all }b \in K^>\text{ with } b \asymp a\big\}.
% \]
% Clearly, $\{va:a \in A\} = \Gamma$, and it is routine to check that $A$ is a multiplicative $\Lambda$-subspace of $K^>$, so Lemma~\ref{lem:sectionexists} gives us a section $s\colon \Gamma\to K^>$ with image contained in $A$. To see that this section $s$ works, let $\gamma \in \Gamma$, put $a\coloneqq s(\gamma)$, and suppose that $\res(a^\dagger) \in \k^{\dagger}$. Take $u \in K^>$ with $u \asymp 1$ and $\res(u^\dagger) = \res(a^\dagger)$. Then $\res(a/u)^\dagger  = 0$, so $\res(a^\dagger) = 0$ since $a \in A$.
% \end{proof}

In the case of $T_{\an}$, we can now show that every spherically complete monotone $\TdO_{\an}$-model with linearly surjective differential residue field is isomorphic to a Hahn field model in the sense of the introduction.
\begin{theorem}\label{thm:sphcompHahniso}
Suppose that $K \models \TdO_{\an}$ is spherically complete and $\k$ is linearly surjective.
Then there is an $\LdO_{\an}$-isomorphism $K \to \k\llp t^{\Gamma}\rrp_{\an, c}$ for some additive $c \colon \Gamma \to \k$.
\end{theorem}
\begin{proof}
Since $K$ is $\Td_{\an}$-henselian by Corollary~\ref{cor:sphcompTdh}, we can equip it with a $\Td_{\an}$-lift $K_0$ of $\k$ by Theorem~\ref{thm:reslift}.
Now let $A \coloneqq \{ a \in K^> : a^{\dagger} \in K_0 \}$, a multiplicative $\Q$-subspace of $K^>$.
We claim that $\{ va : a \in A \}=\Gamma$.
To see this, let $b \in K^>$. Since $K$ is monotone, we have $\epsilon \in \smallo$ with $b^{\dagger} + \epsilon \in K_0$.
Then Corollary~\ref{cor:7.1.9} yields $u \in 1+\smallo$ with $u^{\dagger}=\epsilon$, so $bu \in K^>$ satisfies $v(bu)=vb$ and $(bu)^{\dagger} \in K_0$, proving the claim.
By Lemma~\ref{lem:sectionexists}, we get a section $s\colon \Gamma \to K^>$ such that $s(\gamma)^{\dagger} \in K_0$ for all $\gamma \in \Gamma$.
Let $\imath \colon K_0 \to \k$ be the $\Ld_{\an}$-isomorphism induced by the residue map.
Also identifying $s(\Gamma)$ with the monomial group $t^{\Gamma}$ in the obvious way, we obtain an expansion of the Hahn field $\k\llp t^{\Gamma}\rrp$ to $\k\llp t^{\Gamma}\rrp_{\an, c}$ as in the introduction, where $c \colon \Gamma \to \k$ is defined by $c(\gamma)\coloneqq\imath(s(\gamma)^{\dagger})$ for $\gamma \in \Gamma$.
Iterating~\cite[Lemma 2.3]{Ka23}, we extend $\imath$ to an $\LO_{\an}$-isomorphism $K_0\langle s(\Gamma)\rangle \to \k\langle t^{\Gamma}\rangle$. This map is even an $\LdO_{\an}$-isomorphism by Fact~\ref{fact:transext}, where we construe $K_0\langle s(\Gamma)\rangle$ and $\k\langle t^{\Gamma}\rangle$ as $\LdO_{\an}$-substructures of $K$ and $\k\llp t^{\Gamma}\rrp_{\an, c}$, respectively.
It remains to appeal to Theorem~\ref{thm:sphcompuniquemono}.
\end{proof}

%------------------------------------------------------------------------------%
\section{When \texorpdfstring{$T=\RCF$}{T = RCF}}\label{sec:RCF}
%------------------------------------------------------------------------------%

In the setting of real closed fields we investigate whether the notions and results in the previous sections specialize to the analogous notions and results of \cite[Chapter~7]{ADH17}, providing a full positive answer in the monotone case.
First, it is worth clarifying the relationship between the version of $\d$-henselianity considered in~\cite[Section~5]{Sc00} and the version considered in~\cite[Chapter~7]{ADH17}, as these two notions are slightly different. Indeed, the notion in~\cite{Sc00} is more restrictive:\ if a valued differential field $E$ with small derivation is $\d$-henselian in the sense of~\cite{Sc00}, then it is clearly $\d$-henselian in the sense of~\cite{ADH17}. In the case that $E$ is \emph{monotone} (the only case considered in~\cite{Sc00}), these two notions coincide by~\cite[Corollary 4.5.4 and Lemma 7.2.2]{ADH17}. However, the less restrictive definition given in~\cite{ADH17} seems to be more useful in the non-monotone case; for example, this is the version used in~\cite{Pyn20} to study immediate extensions of asymptotic fields.

If $T=\RCF$ (in the language of ordered rings) and $\k$ has nontrivial derivation,  then $K$ has no proper immediate $\TdO$-extension if and only if it has no proper immediate valued differential field extension with small derivation by Fact~\ref{fact:sphcomp} and \cite[Corollary~6.9.4]{ADH17}.
Additionally:
\begin{lemma}\label{lem:Tdalgmaxiffdalgmax}
Let $T=\RCF$.
The $\TdO$-model $K$ is $\Td$-algebraically maximal if and only if the valued differential field $K$ is $\d$-algebraically maximal in the sense of \cite[Chapter~7]{ADH17}.
\end{lemma}
\begin{proof}
Since $\cL(K)$-definable functions are semialgebraic, the notions $\Td$-algebraic and $\d$-algebraic coincide.
Thus the right-to-left direction is trivial.
Conversely, let $L$ be a $\d$-algebraic immediate valued differential field extension of $K$ with small derivation.
Since $\Gamma$ is divisible and $\k$ is real closed, the henselization $L^{\h}$ of $L$ is real closed, and it has small derivation by \cite[Proposition~6.2.1]{ADH17}.
Additionally, its valuation ring is convex, and hence $T$-convex by (see \cite[Proposition~4.2]{DL95}).
Thus $L^{\h}$ is a $\Td$-algebraic immediate $\TdO$-extension of~$K$.
\end{proof}

Combining this lemma with the observation preceding it shows that when $T=\RCF$, Theorem~\ref{thm:sphcompuniquemono} follows immediately from \cite[Theorem~7.4.3]{ADH17}.

%%older argument in monotone and asymptotic case
% Lemma~\ref{lem:Tdalgmaxiffdalgmax} also yields:
% \begin{corollary}
% Suppose that $K$ is monotone and asymptotic.
% Then $K$ is $\Td$-henselian if and only if $K$ is $\d$-henselian in the sense of \cite[Chapter~7]{ADH17}.
% \end{corollary}
% \begin{proof}
% Suppose that $\k$ is linearly surjective.
% Then by \cite[Theorems~7.0.1, 7.0.3]{ADH17}, $K$ is $\d$-algebraically maximal if and only if $K$ is $\d$-henselian.
% Likewise, by Theorem~\ref{thm:TdhTdalgmaxmono}, $K$ is $\Td$-algebraically maximal if and only if $K$ is $\Td$-henselian.
% The result follows.
% \end{proof}
% This raises the question of whether $\Td$-henselianity always implies $\d$-henselianity.
% Moreover, at present the only proof we have of this implication in the case that $K$ is monotone and asymptotic is the roundabout proof given above.

Now we show that in the monotone case, $\Td$-henselianity implies $\d$-henselianity. 
Recall the terminology and notation for a class of $\TdO$-models from Section~\ref{sec:Tdhcclass}.
\begin{lemma}\label{lem:Tdhtodh}
Let $\cC$ be a class of $\TdO$-models containing $\cI(K)$ and every elementary extension of $K$. Suppose that $K$ is $\Td$-henselian and $\cC$ has the  $\Td$-hensel configuration property. Then $K$ is $\d$-henselian.
\end{lemma}
\begin{proof}
First, we argue that $K$ is existentially closed in any immediate $\TdO$-extension $M$ of $K$; fix such an $M$.
Take an elementary $\TdO$-extension $L$ of $K$ that is $|\Gamma|^+$-saturated.
Then we can arrange by Theorem~\ref{thm:Tdhcclassembed} that $K \subseteq M \subseteq L$.
Since $L$ is an elementary $\TdO$-extension of $K$, it follows that $K$ is existentially closed in $M$ as $\TdO$-models.
By Fact~\ref{fact:sphcomp}, we can take $M$ to be a spherically complete immediate $\TdO$-extension of $K$, which is $\d$-henselian by \cite[Corollary~7.0.2]{ADH17}.
Since $K$ is existentially closed in $M$, $K$ itself is $\d$-henselian.
\end{proof}

For the converse in the case $T=\RCF$ we need a variant of Theorem~\ref{thm:Tdhcclassembed} for valued differential fields with small derivation. The statement involves the differential-henselian configuration property (abbreviated dh-configuration property) from \cite{DPC19} and indeed its proof is almost the same as that of \cite[Theorem~4.2]{DPC19}, but using saturation instead of spherical completeness to obtain pseudolimits (cf.\ \cite[Corollary~7.4.5]{ADH17}).
\begin{lemma}\label{lem:dhcpclassembed}
Let $E$ be a valued differential field with small derivation and suppose that $\k_E$ is linearly surjective and $\Gamma_E\neq\{0\}$.
Suppose that every immediate valued differential field extension of $E$ with small derivation has the dh-configuration property.
Let $L$ be a $\d$-henselian valued differential field extension of $E$.
If $L$ is $|\Gamma_E|^+$-saturated, then any immediate valued differential field extension of $E$ with small derivation can be embedded in $L$ over~$E$.
\end{lemma}

To apply Lemma~\ref{lem:dhcpclassembed} to $\TdO$-models in the case $T=\RCF$, note that $E$ is a $\TdO$-model if and only if $E$ is a real closed field equipped with a nontrivial convex valuation ring and a derivation that is small, by \cite[Proposition~4.2]{DL95} and \cite[Proposition~2.8]{FK21}.
In case $E\models \TdO$, an immediate valued differential field extension of $E$ can be uniquely ordered making it an \emph{ordered} valued differential field extension of $K$ with convex valuation ring (see for example \cite[Corollary~3.5.12]{ADH17}); below we always construe immediate extensions in this way.
Thus if in Lemma~\ref{lem:dhcpclassembed} $L$ is also an \emph{ordered} valued differential field extension of $E$ with convex valuation ring, then the embeddings into $L$ in the lemma are in fact \emph{ordered} valued differential field embeddings.

\begin{lemma}\label{lem:dhtoTdh}
Let $T=\RCF$ and suppose that $K$ is $\d$-henselian and every immediate valued differential field extension of $K$ with small derivation has the dh-configuration property.
Then $K$ is $\Td$-henselian.
\end{lemma}
\begin{proof}
First, we argue that $K$ is existentially closed as an ordered valued differential field in any immediate ordered valued differential field extension $M$ of $K$ with small derivation; fix such an $M$.
Take an elementary ordered valued differential field extension $L$ of $K$ that is $|\Gamma|^+$-saturated.
Then $L$ is in fact a $\TdO$-extension of $K$.
We can arrange by Lemma~\ref{lem:dhcpclassembed} that $K \subseteq M \subseteq L$ as ordered valued differential fields.
Since $L$ is an elementary $\TdO$-extension of $K$, it follows that $K$ is existentially closed in $M$ as ordered valued differential fields.
By Fact~\ref{fact:sphcomp}, we can take $M$ to be a spherically complete immediate $\TdO$-extension of $K$, which is $\Td$-henselian by Corollary~\ref{cor:sphcompTdh}.
We claim that $K$ itself is $\Td$-henselian. Let $(F,A,a,\gamma)$ be in $\Td$-hensel configuration. Since $\TO$ is model complete, the statements that $A$ linearly approximates $F$ on $B(a,\gamma)$ and that $vF(\jet^ra)>v_A(\gamma)$ are equivalent to universal sentences. Thus, both statements hold in $M$ as well, since $K$ is existentially closed in $M$. Using $\Td$-henselianity, we find $b \in B(a,\gamma)^M$ with $F(\jet^rb) = 0$ and $vA_{\times(b-a)}\geq vF(\jet^ra)$. Existential closedness gives an element in $K$ with the same properties.
%Since $K$ is existentially closed in $M$ as ordered valued differential fields, $K$ itself is $\Td$-henselian {\color{red}a comment about QE in RCF is in order here}.
\end{proof}

%%variant for monotone case
% \begin{lemma}%\label{lem:dhtoTdh}
% Suppose that $T=\RCF$ and $K$ is monotone and $\d$-henselian.
% Then $K$ is $\Td$-henselian.
% \end{lemma}
% \begin{proof}
% By Fact~\ref{fact:sphcomp}, let $L$ be an immediate, spherically complete $\TdO$-extension of $K$.
% Then $L$ is $\d$-henselian by \cite[Corollary~7.0.2]{ADH17} and $\Td$-henselian by Corollary~\ref{cor:sphcompTdh}.
% Now, by Section~\ref{sec:cmap}, we can equip $K$ with an additive map $c \colon \Gamma \to \k$.
% Construing $K$ and $L$ as three-sorted structures $(K,\k,\Gamma;\pi,v,c)$ and $(L,\k,\Gamma;\pi,v,c)$, where $\pi$ is the residue map, and applying \cite[Theorem~6.2]{Ha18}, we get that $K$ is elementarily equivalent to $L$.
% It follows that $K$ is $\Td$-henselian.
% \end{proof}

\begin{corollary}
Suppose that $T=\RCF$ and $K$ is monotone.
Then $K$ is $\Td$-henselian if and only if $K$ is $\d$-henselian.
\end{corollary}
\begin{proof}
We assume that the derivation on $\k$ is linearly surjective (otherwise, $K$ is neither $\Td$-henselian nor $\d$-henselian). 
For the left-to-right direction, note that if $\cC$ is the class of monotone $\TdO$-models, then $\cC$ has the $\Td$-hensel configuration property by Corollary~\ref{cor:monotoneTdhc2}, and $\cC$ contains every immediate $\TdO$-extension of $K$ \cite[Corollary~6.3.6]{ADH17} and every elementary $\TdO$-extension of $K$.
It remains to apply Lemma~\ref{lem:Tdhtodh}.

For the right-to-left direction, note that every immediate valued differential field extension of $K$ with small derivation is monotone as before.
Additionally, every monotone valued differential field $E$ with small derivation such that $\Gamma_E\neq\{0\}$ and the induced derivation on $\k_E$ is nontrivial has the dh-configuration property by \cite[Proposition~7.4.1]{ADH17}. 
It remains to apply Lemma~\ref{lem:dhtoTdh}.
\end{proof}

Combining this with Lemma~\ref{lem:Tdalgmaxiffdalgmax} shows that when $T=\RCF$, Theorem~\ref{thm:TdhTdalgmaxmono} follows immediately from \cite[Theorems~7.0.1 and 7.0.3]{ADH17} and Theorem~\ref{thm:Tdhenselizationmono} from \cite[Theorem~3.7]{Pyn20} (the latter also uses \cite[Corollary~6.3.6]{ADH17}).

%------------------------------------------------------------------------------%
\section{An AKE theorem for monotone \texorpdfstring{$T^\der$}{T∂}-henselian fields}\label{sec:AKE}
%------------------------------------------------------------------------------%

To establish our Ax--Kochen/Ershov theorem for monotone $\Td$-henselian fields, we construe a monotone $\Td$-henselian field $K$ as a three-sorted structure $\cK = (K, \k, \Gamma; \pi, v, c, \ac)$ with a sort $\f$ for $K$ as a structure in the language $\cL_{\f} = \Ld$, a sort $\r$ for $\k$ as a structure in the language $\cL_{\r} = \Ld$, and a sort $\val$ for $\Gamma$ as a structure in the  (one-sorted) language $\cL_{\val}$ of ordered $\Lambda$-vector spaces, together with symbols for maps $\pi, v, c, \ac$ connecting the sorts as follows.
Suppose that:
\begin{enumerate}
	\item $K \models \Td$;
	\item $\k \models \Td$;
	\item $\Gamma$ is an ordered $\Lambda$-vector space;
	\item $v \colon K^{\times} \to \Gamma$ is a (surjective) valuation making $K$ a monotone $\Td$-henselian $\TdO$-model such that $v(a^{\lambda})=\lambda va$ for all $a \in K^>$ and $\lambda \in \Lambda$;
	\item $\pi \colon \cO \to \k$ is a map with kernel $\smallo$ such that the map $\res(K) \to \k$ induced by $\pi$ is an $\Ld$-isomorphism;
	\item $c \colon \Gamma \to \k$ is $\Lambda$-linear and for every $\gamma \in \Gamma$, there is $a \in K^>$ with $va=\gamma$ and $\pi(a^{\dagger})=c(\gamma)$;
	\item $\ac \colon K \to \k$ is an angular component map such that $\ac(a)^{\dagger}=\pi(a^{\dagger})-c(va)$ for all $a \in K^{\times}$.
\end{enumerate}
Let $\cL_3$ be this three-sorted language of $\cK$, where we extend $v$ and $\pi$ to $K$ by $v(0)=0$ and $\pi(K\setminus\cO)=\{0\}$, respectively.
Note that in $\cL_3$ we have two distinct copies of the language $\Ld$, one for the sort $\f$ and one for the sort $\r$.
Let $\TdOmonac$ be the theory whose models are such $\cK$.
Note that by Section~\ref{sec:cmap}, any monotone $\Td$-henselian $\TdO$-model can be expanded to a $\TdOmonac$-model. For the remainder of this section, let $\cK = (K, \k, \Gamma; \pi, v, c, \ac)$ and $\cK^* = (K^*, \k^*, \Gamma^*; \pi^*, v^*, c^*, \ac^*)$ be $\TdOmonac$-models.

\subsection{Back-and-forth}
Our goal is to construct a back-and-forth system between $\cK$ and $\cK^*$, when they are %$\Td$-henselian and 
appropriately saturated.
A \textbf{good substructure} of $\cK$ is a triple $\cE = (E, \k_{\cE}, \Gamma_{\cE})$ such that:
\begin{enumerate}[label=(GS\arabic*)]
	\item $E$ is a $\Td$-submodel of~$K$;
	\item $\k_{\cE}$ is a $\Td$-submodel of $\k$ with $\ac(E) \subseteq \k_{\cE}$ (so $\pi(E) \subseteq \k_{\cE}$);
	\item $\Gamma_{\cE}$ is an ordered $\Lambda$-subspace of $\Gamma$ with $v(E^{\times}) \subseteq \Gamma_{\cE}$ and $c(\Gamma_{\cE}) \subseteq \k_{\cE}$.
\end{enumerate}
Note that in this definition we neither require $\ac(E) = \k_E$, let alone $\pi(E) = \k_E$, nor $v(E^{\times})=\Gamma_E$.
When needed we construe $E$ as a $\TdO$-model $(E, \cO_E)$ with the induced valuation ring $\cO_E \coloneqq \cO \cap E$.
If $\cE_1 = (E, \k_{\cE_1}, \Gamma_{\cE_1})$ and $\cE_2 = (E_2, \k_{\cE_2}, \Gamma_{\cE_2})$ are good substructures of $\cK$, then $\cE_1 \subseteq \cE_2$ means $E_1 \subseteq E_2$, $\k_{\cE_1} \subseteq \k_{\cE_2}$, and $\Gamma_{\cE_1} \subseteq \Gamma_{\cE_2}$.
Now let $\cE = (E, \k_{\cE}, \Gamma_{\cE})$ and $\cE^* = (E^*, \k_{\cE^*}, \Gamma_{\cE^*})$ be good substructures of $\cK$ and $\cK^*$, respectively.
A \textbf{good map} $\bf \colon \cE \to \cE^*$ is a triple $\bf = (f, f_{\r}, f_{\val})$ consisting of $\Ld$-isomorphisms $f \colon E \to E^*$ and $f_{\r} \colon \k_{\cE} \to \k_{\cE^*}$ and an isomorphism $f_{\val} \colon \Gamma_{\cE} \to \Gamma_{\cE^*}$ of ordered $\Lambda$-vector spaces such that:
\begin{enumerate}[label=(GM\arabic*)]
	\item $f_{\r}\big( \ac(a) \big) = \ac^*\big( f(a) \big)$ for all $a \in E$;
	\item $f_{\val}\big( v(a) \big) = v^*\big( f(a) \big)$ for all $a \in E^{\times}$;
	\item $(f_{\r}, f_{\val})$ is a partial elementary map $(\k, \Gamma; c) \to (\k^*, \Gamma^*; c^*)$ (so $f_{\r}\big(c(\gamma)\big) = c^*\big(f_{\val}(\gamma)\big)$ for all $\gamma \in \Gamma_{\cE}$).
\end{enumerate}

This lemma handles residue field extensions.
\begin{lemma}\label{lem:resext}
%Suppose that $\cK$ and $\cK^*$ are $\Td$-henselian and 
Let $\bf \colon \cE \to \cE^*$ be a good map and let $d \in \k_{\cE} \setminus \pi(E)$. Then there are $b \in \cO$ with $\pi(b) = d$ and a good map $\bg \colon (E\llangle b \rrangle, \k_{\cE}, \Gamma_{\cE}) \to \cK^*$ extending $\bf$.
\end{lemma}
\begin{proof}
Proposition~\ref{prop:resext} gives $b \in \cO$ with $\pi(b) = d$ and $v\big( E\llangle b \rrangle^\times \big) = v(E^\times)$, as well as an $\LdO$-embedding $g\colon E\llangle b \rrangle\to K^*$ extending $f$ such that $\pi^*\big(g(a)\big) =f_{\r}\big(\pi(a))$ for all $a \in E\llangle b \rrangle$.
Let $\bg = (g,f_{\r},f_{\val})$. To see that $(E\llangle b \rrangle, \k_{\cE}, \Gamma_{\cE})$ is a good substructure and that $\bg$ is a good map, we only need to show that $\ac(a) \in \k_{\cE}$ and that $\ac^*\!\big(g(a)\big) =f_{\r}\big(\ac(a))$ for all $a \in E\llangle b \rrangle$. This holds by~\cite[Lemma 4.2]{Ha18B}, but we repeat the short argument here.
Take $y \in E$ with $va = vy$ and take $u \in E\llangle b \rrangle$ with $a = uy$. Then $u \asymp 1$, so $\ac(a) = \ac(u)\ac(y) = \pi(u)\ac(y) \in \k_{\cE}$ and 
\[
f_{\r}\big(\ac(a))\ =\ f_{\r}\big(\pi(u)\big)f_{\r}\big(\ac(y)\big)\ =\ \pi^*\big(g(u)\big)\ac^*\!\big(f(y)\big) \ =\ \ac^*\!\big(g(a)\big).\qedhere
\]
\end{proof}

The next lemma handles value group extensions.
\begin{lemma}\label{lem:valext}
%Suppose that $\cK$ and $\cK^*$ are $\Td$-henselian and 
Let $\bf \colon \cE \to \cE^*$ be a good map and 
suppose that $\pi(E)=\k_{\cE}$ and $E$ is equipped with a $\Td$-lift of $\k_{\cE}$. %$\iota \colon \k_{\cE} \to E$ is a $\Td$-lift of $\k_{\cE}$ (by which we mean that $\iota(\k_{\cE})$ is a $\Td$-lift of $\k_{\cE}$).
Let $\gamma \in \Gamma_{\cE} \setminus v(E^{\times})$.
Then there are $b \in K^{\times}$ with %$b^{\dagger}=(\iota \circ c)(\gamma)$ and 
$vb=\gamma$ and a good map $\bg \colon (E\llangle b \rrangle, \k_{\cE}, \Gamma_{\cE}) \to \cK^*$ extending $\bf$. %and satisfying $v^*(gb)=f_{\val}(vb)$.
\end{lemma}
\begin{proof}
Let $E_0 \subseteq E$ be a $\Td$-lift of $\k_{\cE}$.
We will  find $b \in K^>$ with $b^{\dagger} \in E_0$, $\pi(b^{\dagger} ) = c(\gamma)$, $vb=\gamma$, and $\ac(b)=1$.
By assumption, we have $a \in K^>$ with $\pi(a^{\dagger})=c(\gamma)$ and $va=\gamma$. Take $u \in E_0$ with $\pi(u) = \pi(a^\dagger) = c(\gamma)$, and let $\varepsilon \coloneqq a^{\dagger}-u \in \smallo$. By Corollary~\ref{cor:7.1.9}, we have $\delta \in \smallo$ with $\varepsilon=(1+\delta)^{\dagger}$, so by replacing $a$ with $a/(1+\delta)$, we arrange that $a^\dagger \in E_0$.
By Corollary~\ref{cor:reslift}, extend $E_0$ to a $\Td$-lift $E_1 \subseteq K$ of $\k$. Now, take $e \in E_1$ with $\pi(e) = \ac(a)$. We have
\[
\pi(e^\dagger)\ =\ \pi(e)^\dagger\ =\ \ac(a)^\dagger\ =\ \pi(a^{\dagger})-c(\gamma)\ =\ 0.
\]
Since $e^\dagger \in E_1$, it follows that $e^\dagger = 0$. Let $b\coloneqq a/e$, so $b^\dagger = a^\dagger \in E_0$, $vb = va = \gamma$, and $\ac(b) = \ac(a)/\pi(e) = 1$.
Now $f(E_0)\subseteq E^*$ is a $\Td$-lift of $\k_{\cE^*}$, so as above take $b^* \in (K^*)^>$ with $(b^*)^{\dagger} \in f(E_0)$, $\pi^*((b^*)^{\dagger} ) = c^*(f_{\val}\gamma)$, $vb^*=f_{\val}\gamma$, and $\ac^*(b^*)=1$.
% Now, take $e \in E_0$ with $\pi(e) = \ac(a)$.

Since $b^\dagger \in E_0 \subseteq E$, we have $E\llangle b \rrangle = E\langle b \rangle$.
Then $\res(E\langle b \rangle)=\res(E)$ and $v(E\langle b \rangle^{\times}) = v(E^{\times}) \oplus \Lambda \gamma \subseteq \Gamma_{\cE}$ by the Wilkie inequality.
Since $b$ and $b^*$ have the same sign and realize the same cut over $v(E^\times)$, we may use~\cite[Lemma~2.3]{Ka23} to get an $\LO$-embedding $g \colon E\langle b \rangle \to K^*$ extending $f$ and satisfying $gb=b^*$. Note that $(b^*)^{\dagger} = f(b^\dagger)$, so $g$ is even an $\LdO$-embedding by Fact~\ref{fact:transext}.
To verify that $\bg \coloneqq (g, f_{\r}, f_{\val})$ is a good map, it remains to check that $\bg$ preserves $\ac$.
For this, use that for $\cL(E)$-definable $F \colon K \to K$ with $F(b)\neq 0$, we have $\lambda \in \Lambda$ and $d \in E^{\times}$ with $F(b) \sim b^{\lambda}d$, so $\ac(F(b)) = \ac(b^{\lambda}d) = \ac(b)^{\lambda}\ac(d) = \ac(d)$.
\end{proof}

\begin{theorem}\label{thm:equiv}
%	Suppose that $\cK$ and $\cK^*$ are $\Td$-henselian. Then 
Any good map $\cE \to \cE^*$ is a partial elementary map $\cK \to \cK^*$.
\end{theorem}
\begin{proof}
Let $\kappa$ be an uncountable cardinal with $\max\{ |\k_{\cE}|, |\Gamma_{\cE}| \}<\kappa$.
By passing to elementary extensions we arrange that $\cK$ and $\cK^*$ are $\kappa^+$-saturated.
We call a good substructure $(E_1, \k_{1}, \Gamma_{1})$ of $\cK$ \textbf{small} if $\max\{ |\k_1|, |\Gamma_1| \}<\kappa$.
It suffices to show that the set of good maps with small domain is a back-and-forth system from $\cK$ to $\cK^*$.

First, we describe several extension procedures.
\begin{enumerate}
	\item\label{akeext:ressort} Given $d \in \k$, arranging that $d \in \k_{\cE}$: By the saturation assumption, we can extend $f_{\r}$ to a map $g_{\r} \colon \k_{\cE}\llangle d \rrangle \to \k^*$ so that $(g_{\r}, f_{\val})$ is a partial elementary map. Then $(f, g_{\r}, f_{\val})$ is the desired extension of~$\bf$.
	\item\label{akeext:valsort} Given $\gamma \in \Gamma$, arranging that $\gamma \in \Gamma_{\cE}$: First use \ref{akeext:ressort} to arrange $c(\gamma) \in \k_{\cE}$, then use saturation as before to extend $f_{\val}$ to $g_{\val} \colon \Gamma_{\cE} \oplus \Lambda\gamma \to \Gamma^*$ with the desired properties.
	\item\label{akeext:res} Arranging $\pi(E)=\k_{\cE}$: If $d \in \k_{\cE} \setminus \pi(E)$, then Lemma~\ref{lem:resext} yields $b \in K$ and an extension of $\bf$ to a good map $(g, f_{\r}, f_{\val})$ with small domain $(E\llangle b\rrangle, \k_{\cE}, \Gamma_{\cE})$. Iterate this procedure to arrange $\pi(E)=\k_{\cE}$.
	\item\label{akeext:Tdh} Arranging that $(E, \cO_E)$ is $\Td$-henselian: By \ref{akeext:ressort} and \ref{akeext:res} we can assume that $\k_{\cE}$ is linearly surjective and $\pi(E)=\k_{\cE}$. Now use Fact~\ref{fact:sphcomp} to take a spherically complete immediate $\TdO$-extension $L$ of $E$. Then $L$ is $\Td$-henselian by Corollary~\ref{cor:sphcompTdh}, and $L$ embeds over $E$ into both $K$ and $K^*$ by Corollary~\ref{cor:sphcompembed}. Let $g$ be the extension of $f$ to an $\LdO$-isomorphism between these images of $L$ in $K$ and $K^*$, respectively. Then $(g,f_{\r},f_{\val})$ is a good map by \cite[Corollary~4.4]{Ha18B}.
	\item\label{akeext:val} Arranging $v(E^{\times})=\Gamma_{\cE}$: We can assume that $\pi(E)=\k_{\cE}$ and that $\cE$ is $\Td$-henselian and is equipped with a $\Td$-lift of $\k_{\cE}$ by Theorem~\ref{thm:reslift}. If $\gamma \in \Gamma_{\cE} \setminus v(E^{\times})$, then Lemma~\ref{lem:valext} yields $b \in K$ and an extension of $\bf$ to a good map $(g, f_{\r}, f_{\val})$ with small domain $(E\llangle b \rrangle, \k_{\cE}, \Gamma_{\cE})$. Iterate this procedure to arrange $v(E^{\times})=\Gamma_{\cE}$.
\end{enumerate}

Given $a \in K$, we need to extend $\bf$ to a good map with small domain containing $a$.
By the above, we can assume that $\pi(E)=\k_{\cE}$ and $v(E^{\times})=\Gamma_{\cE}$.
From
\[
\rkL\big(\pi(E\llangle a \rrangle) | \pi(E)\big)\ \leq\ \rkL(E\llangle a \rrangle | E)\ \leq\ \aleph_0,
\]
we get $|\pi(E\llangle a \rrangle)|<\kappa$, and from 
\[
\dim_{\Lambda}\big( v(E\llangle a \rrangle^{\times}) | v(E^{\times}) \big)\ \leq\ \rkL(E\llangle a \rrangle | E)\ \leq\ \aleph_0,
\]
we get $|v(E\llangle a \rrangle^{\times})| < \kappa$.
Hence by \ref{akeext:ressort}--\ref{akeext:val} we extend $\bf$ to a good map $\bf_1 = (f_1, f_{1, \r}, f_{1, \val})$ with small domain $\cE_1 = (E_1, \k_1, \Gamma_1) \supseteq \cE$ such that $\k_1$ is linearly surjective and
\[
\pi\big(E\llangle a \rrangle\big) \subseteq \k_1 = \pi(E_1) \qquad \text{and} \qquad v\big(E\llangle a \rrangle^{\times}\big) \subseteq \Gamma_1 = v(E_1^{\times}).
\]
In the same way, we extend $\bf_1$ to a good map $\bf_2$ with small domain $\cE_2 = (E_2, \k_2, \Gamma_2) \supseteq \cE_1$ such that $\k_2$ is linearly surjective and
\[
\pi\big(E_1\llangle a \rrangle\big) \subseteq \k_2 = \pi(E_2) \qquad \text{and} \qquad v\big(E_1\llangle a \rrangle^{\times}\big) \subseteq \Gamma_2 = v(E_2^{\times}).
\]
Iterating this procedure and taking unions yields an extension of $\bf$ to a good map $\bf_{\omega} = (f_{\omega}, f_{\omega, \r}, f_{\omega, \val})$ with small domain $\cE_{\omega} = (E_{\omega}, \k_{\omega}, \Gamma_{\omega}) \supseteq \cE$ such that $\k_{\omega}$ is linearly surjective and
\[
\pi\big(E_{\omega}\llangle a \rrangle\big) = \k_{\omega} = \pi(E_{\omega}) \qquad \text{and} \qquad v\big(E_{\omega}\llangle a \rrangle^{\times}\big) = \Gamma_{\omega} = v(E_{\omega}^{\times}).
\]
This makes $\big(E_{\omega}\llangle a \rrangle, \cO_{E_{\omega}\llangle a \rrangle}\big)$ an immediate $\TdO$-extension of $(E_{\omega}, \cO_{E_{\omega}})$, so by Fact~\ref{fact:sphcomp} and Corollary~\ref{cor:sphcompembed} we can take a spherically complete immediate $\TdO$-extension $(E_{\omega+1}, \cO_{E_{\omega+1}})$ of $\big(E_{\omega}\llangle a \rrangle, \cO_{E_{\omega}\llangle a \rrangle}\big)$ inside $\cK$, which is also an immediate $\TdO$-extension of $(E_{\omega}, \cO_{E_{\omega}})$.
Then $\cE_{\omega+1} = (E_{\omega+1}, \k_{\omega}, \Gamma_{\omega})$ is a good substructure of $\cK$.
Likewise taking a spherically complete immediate $\TdO$-extension of $\big(f_{\omega}(E_{\omega}), f_{\omega}(\cO_{E_{\omega}})\big)$ inside $\cK^*$, Theorem~\ref{thm:sphcompuniquemono} and \cite[Corollary~4.4]{Ha18B} yield an extension of $\bf_{\omega}$ to a good map with small domain $\cE_{\omega+1}$ containing~$a$.
\end{proof}

For the next result, we construe $(\k, \Gamma; c)$ as a structure in the two-sorted language $\cL_{\r\!\val, c} = \cL_{\r} \cup \cL_{\val} \cup \{c\}$.
\begin{corollary}\label{cor:elementaryequivalence}
%Suppose that $\cK$ and $\cK^*$ are $\Td$-henselian.
We have $\cK \equiv \cK^*$ if and only if $(\k, \Gamma; c) \equiv (\k^*, \Gamma^*; c^*)$.
\end{corollary}
\begin{proof}
The left-to-right direction is obvious.
For the converse, suppose that $(\k, \Gamma; c) \equiv (\k^*, \Gamma^*; c^*)$.
Let $\PP$ be the prime model of $T$.
Then $\cE = (\PP, \PP, \{0\})$ is a good substructure of $\cK$ and $\cE^* = (\PP, \PP, \{0\})$ is a good substructure of $\cK^*$, and we have a good map $\cE \to \cE^*$, which is partial elementary by Theorem~\ref{thm:equiv}.
\end{proof}

For $T=T_{\an}$, this yields a theorem claimed in the introduction.
\begin{corollary}\label{cor:elemequivHahn}
Any $\Td_{\an}$-henselian monotone $\TdO_{\an}$-model is elementarily equivalent to $\k\llp t^{\Gamma}\rrp_{\an,c}$ for some $\k \models \Td_{\an}$, divisible ordered abelian group $\Gamma$, and additive map $c \colon \Gamma \to \k$.
\end{corollary}

\begin{corollary}\label{cor:elementarysubstructure}
Let $\cE \subseteq \cK$ be a $\TdOmonac$-model such that $(\k_{\cE}, \Gamma_{\cE}; c) \preceq (\k, \Gamma; c)$. Then $\cE \preceq \cK$.
%Suppose that $\cK$ is $\Td$-henselian and let $\cE = (E, \k_{\cE}, \Gamma_{\cE}; \pi, v, c, ac) \subseteq \cK$ be a $\Td$-henselian $\TdOmonac$-model such that $(\k_{\cE}, \Gamma_{\cE}; c) \preceq (\k, \Gamma; c)$. Then $\cE \preceq \cK$.
\end{corollary}
\begin{proof}
The identity map on $(E, \k_{\cE}, \Gamma_{\cE})$ is a good map from $\cE$ to $\cK$, so $\cE \preceq \cK$ by Theorem~\ref{thm:equiv}.
\end{proof}

We can eliminate angular components from the previous corollary.
\begin{corollary}\label{cor:elementarysubstructure2}
%Suppose that $(K, \k, \Gamma; \pi, v, c)$ is $\Td$-henselian and 
Let $(E, \k_{\cE}, \Gamma_{\cE}; \pi, v, c) \subseteq (K, \k, \Gamma; \pi, v, c)$ be $\Td$-henselian such that $(\k_{\cE}, \Gamma_{\cE}; c) \preceq (\k, \Gamma; c)$. Then $(E, \k_{\cE}, \Gamma_{\cE}; \pi, v, c) \preceq (K, \k, \Gamma; \pi, v, c)$.
\end{corollary}
\begin{proof}
Let $\Delta$ be a $\Lambda$-subspace of $\Gamma$ such that $\Gamma = \Gamma_{\cE} \oplus \Delta$ and let $B$ be a $\Lambda$-subspace of $K^>$ such that $K^>=E^>\cdot B$ is the direct sum of $E^>$ and $B$.
By Lemma~\ref{lem:sectionexists}, take a section $s_E \colon \Gamma_{\cE} \to E^>$.
By the proof of the same lemma, take a $\Lambda$-vector space embedding $s_B \colon \Delta \to B$ such that $v(s_B(\delta))=\delta$ for all $\delta \in \Delta$.
Then $s\colon \Gamma \to K^>$ defined by $s(\gamma+\delta)=s_E(\gamma)s_B(\delta)$ for $\gamma \in \Gamma_{\cE}$ and $\delta \in \Delta$ is a section for $K$.
Letting $\ac$ be the angular component induced by $s$, the result now follows from Corollary~\ref{cor:elementarysubstructure}. 
\end{proof}

Combining the previous corollary with Theorem~\ref{thm:sphcompHahniso} yields the following improvement of Corollary~\ref{cor:elemequivHahn}.
\begin{corollary}
Any $\Td_{\an}$-henselian monotone $\TdO_{\an}$-model is isomorphic to an elementary substructure of $\k\llp t^{\Gamma}\rrp_{\an,c}$ for some $\k \models \Td_{\an}$, divisible ordered abelian group $\Gamma$, and additive map $c \colon \Gamma \to \k$.
\end{corollary}
\begin{proof}
Let $E \models \TdO_{\an}$ be $\Td_{\an}$-henselian and monotone.
By Fact~\ref{fact:sphcomp}, $E$ has a spherically complete immediate $\TdO_{\an}$-extension $L$, which is again monotone by \cite[Corollary~6.3.6]{ADH17}.
By Theorem~\ref{thm:sphcompHahniso}, $L$ is isomorphic to $\k\llp t^{\Gamma}\rrp_{\an, c}$ for some additive map $c \colon \Gamma \to \k$.
Finally, $E$ is isomorphic to an elementary substructure of this $\k\llp t^{\Gamma}\rrp_{\an, c}$ by Corollary~\ref{cor:elementarysubstructure2}.
\end{proof}

%------------------------------------------------------------------------------%
\subsection{Relative quantifier elimination and preservation of NIP}
%------------------------------------------------------------------------------%

In this subsection we eliminate quantifiers relative to the two-sorted structure $(\k, \Gamma; c)$ in the language $\cL_{\r\!\val, c} \coloneqq \cL_{\r} \cup \cL_{\val} \cup \{c\}$.
Let $l, s, t \in \N$.
Let $x$ be an $l$-tuple of variables of sort $\f$, $y$ be an $m$-tuple of variables of sort $\r$, and $z$ be an $n$-tuple of variables of sort $\val$.
A formula in $(x,y,z)$ is \textbf{special} if it is of the form
\[
\psi\big(\ac(F_1(\jet^{r}x)), \dots, \ac(F_{s}(\jet^{r}x)), y,  v(G_1(\jet^{r}x)), \dots, v(G_{t}(\jet^{r}x)), z\big)
\]
where $F_1, \dots, F_s, G_1, \dots, G_t \colon K^{(1+r)l} \to K$ are $\cL(\0)$-definable functions and $\psi(u_1, \dots, u_s, y, v_1, \dots, v_t, z)$ is an $\cL_{\r\!\val, c}$-formula with $u_1, \dots, u_s$ variables of sort $\r$ and $v_1, \dots, v_t$ variables of sort $\val$.

\begin{theorem}\label{thm:special}
Every $\cL_3$-formula is equivalent to a special formula, modulo $\TdOmonac$.
\end{theorem}
\begin{proof}
For %a $\TdOmonac$-model $\cK$ and
$a \in K^{l}$, $d \in \k^m$, and $\gamma \in \Gamma^n$, define the \textbf{special type} of $(a, d, \gamma)$ to be
\[
\textrm{sptp}(a, d, \gamma)\ \coloneqq\ \{ \theta(x,y,z) : \cK \models \theta(a, d, \gamma)\ \text{and}\ \theta\ \text{is special}\}.
\]
Now let %$\cK$ and $\cK^*$ be $\TdOmonac$-models and 
$a \in K^{l}$, $a^* \in (K^*)^{l}$, $d \in \k^m$, $d^* \in (\k^*)^m$, $\gamma \in \Gamma^n$, $\gamma^* \in (\Gamma^*)^n$, and suppose that $(a, d, \gamma)$ and $(a^*, d^*, \gamma^*)$ have the same special type.
It suffices to show that $(a, d, \gamma)$ and $(a^*, d^*, \gamma^*)$ have the same type.

Let $E$ be the $\Ld$-substructure of $K$ generated by $a$, $\Gamma_{\cE}$ be the ordered $\Lambda$-subspace of $\Gamma$ generated by $\gamma$ and $v(E^{\times})$, and $\k_{\cE}$ be the $\Ld$-substructure of $\k$ generated by $\ac(E)$, $c(\Gamma_{\cE})$, and $d$.
Then $\cE = (E, \k_{\cE}, \Gamma_{\cE})$ is a good substructure of $\cK$.
Likewise define a good substructure $\cE^* = (E^*, \k_{\cE^*}, \Gamma_{\cE^*})$ of $\cK^*$.
Note that for an $\cL(\0)$-definable $F \colon K^{(1+r)l} \to K$, we have $F(\jet^{r}a)=0$ if and only if $\ac(F(\jet^{r}a))=0$ , and likewise in $\cK^*$.
From this and the assumption that $(a, d, \gamma)$ and $(a^*, d^*, \gamma^*)$ have the same special type we obtain an $\Ld$-isomorphism $f \colon E \to E^*$ with $f(a)=a^*$.
We next get an isomorphism $f_{\val} \colon \Gamma_{\cE} \to \Gamma_{\cE^*}$ of ordered $\Lambda$-vector spaces such that $f_{\val}(\gamma)=\gamma^*$ and $f_{\val}(vb)=v^*(f(b))$ for all $b \in E^{\times}$.
Finally, we get an $\Ld$-isomorphism $f_{\r} \colon \k_{\cE} \to \k_{\cE^*}$ such that $f_{\r}(d)=d^*$, $f_{\r}(\ac(b))=\ac^*(f(b))$ for all $b \in E$, and $f_{\r}(c(\delta))=c^*(f_{\val}(\delta))$ for all $\delta \in \Gamma_{\cE}$.
By the assumption on special types, $(f_{\r}, f_{\val})$ is a partial elementary map $(\k, \Gamma; c) \to (\k^*, \Gamma^*; c^*)$, so $(f, f_{\r}, f_{\val})$ is a good map $\cE \to \cE^*$.
The result now follows from Theorem~\ref{thm:equiv}.
\end{proof}

It follows that $(\k, \Gamma; c)$ is stably embedded in $\cK$:
\begin{corollary}
Any subset of $\k^m \times \Gamma^n$ definable in $\cK$ is definable in $(\k, \Gamma; c)$.
\end{corollary}

Now we turn to preservation of NIP. Consider the following languages, each of which has the same three sorts as $\cL_3$:
\begin{enumerate}
\item The language $\cL'$, where we drop the derivation on the field sort $\f$. Explicitly, the field sort is an $\cL$-structure, the residue field sort is still an $\Ld$-structure, the value group sort is still an ordered $\Lambda$-vector space, and we keep the maps $\pi$, $v$, $c$, and $\ac$. We let $T'$ be the restriction of $\TdOmonac$ to the language $\cL'$; that is, $T'$ consists of all $\cL'$-sentences that hold in all models of $\TdOmonac$. 
\item The language $\cL^{\ac}$, where we drop the derivation on both the field sort $\f$ and the residue field sort $\r$, as well as $c$. Explicitly, the field sort and the residue field sort are both $\cL$-structures, the value group sort is still an ordered $\Lambda$-vector space, and we keep the maps $\pi$, $v$, and $\ac$. We let $T^{\ac}$ be the restriction of $\TdOmonac$ to the language $\cL^{\ac}$.
\end{enumerate}

\begin{lemma}\label{lem:acNIP}
The theory $T^{\ac}$ is complete and has NIP. The residue field and value group are stably embedded in any model of $T^{\ac}$.
\end{lemma}
\begin{proof}
For this proof, let $\cL^s$ extend $\cL^{\ac}$ by an additional map $s$ from the value group sort to the field sort, and let $T^s$ be the $\cL^s$-theory that extends $T^{\ac}$ by axioms stating that $s$ is a section and that $\ac$ is induced by $s$. By Corollary~\ref{cor:acinduced}, any model of $T^{\ac}$ admits an expansion to a model of $T^s$. The theory $T^s$ is complete and has NIP~\cite[Corollary 2.2 and Proposition 4.2]{KK24}, and the residue field and value group are stably embedded in any model of $T^s$~\cite[Corollary 2.4]{KK24}, so the lemma follows. Note that in~\cite{KK24}, the language $\cL^s$ does not contain a function symbol for the angular component, but this is nonetheless definable by $\ac(y) = \pi(y/s(vy))$ for nonzero~$y$.
\end{proof}

\begin{theorem}
%Let $\cK = (K,\k,\Gamma;\pi,v,c,\ac) \models \TdOmonac$ and 
Suppose that $\Th(\k,\Gamma;c)$ has NIP. Then $\Th(\cK)$ has NIP.
\end{theorem}
\begin{proof}
By Lemma~\ref{lem:acNIP}, the theory of the reduct $\cK|_{\cL^{\ac}}$ has NIP, with stably embedded residue field and value group. By~\cite[Proposition 2.5]{JS20}, we may expand the residue field and value group by any additional NIP structure, and the theory of the resulting expansion of $\cK|_{\cL^{\ac}}$ still has NIP. In particular, the theory of the intermediate reduct $\cK|_{\cL'}$ has NIP. Suppose now that $\Th(\cK)$ has IP. After replacing $\cK$ with an elementary extension, we find an $\cL_3$-indiscernible sequence $(a_i)_{i<\omega}$, a tuple of parameters $b$, and an $\cL_3$-formula $\phi(x,y)$ such that $\cK \models \phi(a_i,b)$ if and only if $i$ is even. Note that $x$ and $y$ may span multiple sorts, and we let $x_{\f}$, $x_{\r}$, and $x_{\val}$ denote the parts of $x$ coming from the field sort, residue field sort, and value group sort; likewise for $y_{\f}$, $y_{\r}$, and $y_{\val}$. By Theorem~\ref{thm:special}, we may assume that $\phi(x,y)$ is of the form
\[
\psi\big(\ac(F_1(\jet^{r}x_{\f},\jet^{r}y_{\f})), \dots, \ac(F_{s}(\jet^{r}x_{\f},\jet^{r}y_{\f})), x_{\r},y_{\r},  v(G_1(\jet^{r}x_{\f},\jet^{r}y_{\f})), \dots, v(G_t(\jet^{r}x_{\f},\jet^{r}y_{\f})), x_{\val},y_{\val}\big),
\]
where $\psi$ is an $\cL_{\r\!\val,c}$-formula. Now, we augment each $a_i$ to a tuple $a_i^*$ as follows: write $a_i = a_{i,\f}a_{i,\r}a_{i,\val}$ where  $a_{i,\f}$, $a_{i,\r}$, and $a_{i,\val}$ denote the parts of $a_i$ coming from the relevant sorts,  and put $a_i^*\coloneqq \jet^r(a_{i,\f})a_{i,\r}a_{i,\val}$. We augment $b$ to $b^*$ similarly. Let $x^*_{\f}$ be a field sort variable whose length is $1+r$ times that of $x_{\f}$, likewise for $y^*_{\f}$, and let $\phi^*$ be the $\cL'$-formula
\[
\phi^*(x^*,y^*)\ \coloneqq\ \psi\big(\ac(F_1(x^*_{\f},y^*_{\f})), \dots, \ac(F_{s}(x^*_{\f},y^*_{\f})), x_{\r},y_{\r},  v(G_1(x^*_{\f},y^*_{\f})), \dots, v(G_t(x^*_{\f},y^*_{\f})), x_{\val},y_{\val}\big).
\]
Then $\cK \models \phi^*(a_i^*,b^*)$ if and only if $\cK\models \phi(a_i,b)$ (so  if and only if $i$ is even). Since the sequence $(a_i)_{i<\omega}$ is $\cL_3$-indiscernible, the augmented sequence $(a_i^*)_{i<\omega}$ is $\cL'$-indiscernible, so this contradicts the fact that the theory of $\cK|_{\cL'}$ has NIP.
\end{proof}
\begin{remark}
Instead of appealing to the theory $T^s$ to show that $T^{\ac}$ has NIP, and then using~\cite[Proposition 2.5]{JS20} to deduce that  the $\cL'$-reduct of $\cK$  has NIP, one could likely show that $\Th(\cK|_{\cL'})$ has NIP directly, by building an appropriate back-and-forth system between appropriate substructures of models of $T'$ and applying~\cite[Theorem 2.3]{JS20}. From there, the rest of the proof would proceed as above. It seems likely that distality and NTP$_2$ also transfer from $\Th(\k,\Gamma;c)$ to $\Th(\cK)$, but in both cases, certain aspects of the above proof do not work.
\end{remark}

%------------------------------------------------------------------------------%
\section{The model completion of monotone \texorpdfstring{$T$}{T}-convex \texorpdfstring{$T$}{T}-differential fields}\label{sec:modelcompletion}
%------------------------------------------------------------------------------%
In this section, we return to the one-sorted setting. Let $\TdOmon$ be the theory of monotone $T$-convex $T$-differential fields, in the language $\LdO$. Note that unlike in $\TdO$, we do not require the $T$-convex valuation ring to be proper, and that unlike in $\TdOmonac$, we do not require $\Td$-henselianity. Let $\TdOmonG$ be the theory of $\Td$-henselian $T$-convex $T$-differential fields with many constants, nontrivial valuation, and generic $T$-differential residue field (that is, with differential residue field a model of $\TdG$; see Section~\ref{sec:Td}). Note that any model of $\TdOmonG$ is monotone, as it has many constants and small derivation. The purpose of this section is to show that $\TdOmonG$ is the model completion of $\TdOmon$. For this, we need the following general proposition on residue field extensions.

In this section, $K=(K,\cO,\der)$ is a $T$-convex $T$-differential field with small derivation. The only difference with the earlier standing assumption is that we may have $\cO=K$ (equivalently, $\Gamma=\{0\}$).
\begin{proposition}\label{prop:residueext}
Let $E$ be a $\Td$-extension of $\k$. Then there is a $T$-convex $T$-differential field extension $L$ of $K$ with small derivation and the following properties:
\begin{enumerate}
\item\label{prop:residueexti} $\k_L$ is $\Ld(\k)$-isomorphic to $E$;
\item\label{prop:residueextii} $\Gamma_L= \Gamma$;
\item\label{prop:residueextiii} If $K^*$ is any $\Td$-henselian $\TdO$-extension of $K$, then any $\Ld(\k)$-embedding $\imath\colon \k_L \to \res(K^*)$ lifts to an $\LdO(K)$-embedding $\jmath\colon L \to K^*$.
\end{enumerate}
\end{proposition}
\begin{proof}
If $\Gamma = \{0\}$, then we may take $L = E$, so we will assume that $\Gamma \neq \{0\}$. We may further assume that $E = \k\llangle a\rrangle$, where $a \not\in \k$. First, we build the extension $L$, and then we verify that it is indeed a $\TdO$-extension $L$ of $K$ satisfying \ref{prop:residueexti}--\ref{prop:residueextiii}. Consider the case that  $a$ is $\Td$-transcendental over $\k$, so $E = \k\langle a,a',\ldots\rangle$. We define $\TO$-extensions $K = L_0 \subseteq L_1\subseteq \cdots$ as follows:\ for each $i$, let $L_{i+1}\coloneqq L_i\langle b_i\rangle$, where $b_i$ realizes the cut
\[
\{y\in L_i: y <\cO_{L_i} \text{ or } y \in \cO_{L_i}\text{ and } \bar{y}<a^{(i)}\},
\]
and equip $L_{i+1}$ with the $T$-convex valuation ring
\[
\cO_{L_{i+1}}\ \coloneqq \ \big\{y\in L_{i+1}:|y|<d\text{ for all }d \in L_i\text{ with } d >\cO_{L_i}\big\}.
\]
By~\cite[Main Lemma 3.6]{DL95}, each $L_{i+1}$ is indeed a $\TO$-extension of $L_i$. Let $L\coloneqq \bigcup_{i}L_i$, and using Fact~\ref{fact:transext}, extend $\der$ to a $T$-derivation on $L$ such that $\der b_n = b_{n+1}$ for each $n$. Now, consider the case that $a$ is $\Td$-algebraic over $\k$. Let $r>0$ be minimal with $E = \k\langle \jet^{r-1}a \rangle$ and build $\TO$-extensions $K = L_0 \subseteq L_1\subseteq \cdots\subseteq L_r$ as in the $\Td$-transcendental case, so $L_{i+1}\coloneqq L_i\langle b_i\rangle$ for each $i<r$, where $b_i$ is a lift of $a^{(i)}$. This time, we put $L\coloneqq  L_r = K\langle b_0,\ldots,b_{r-1}\rangle$. Let $K_0 \subseteq K$ be a $T$-lift of $\k$, so there is an $\cL$-isomorphism $K_0\langle b_0,\ldots,b_{r-1}\rangle\to E$ that agrees with the residue map on $K_0$ and sends $b_i$ to $a^{(i)}$ for each $i<r$. Let $b_{r} \in K_0\langle b_0,\ldots,b_{r-1}\rangle$  be the unique element that maps to $a^{(r)} \in E$, and again using Fact~\ref{fact:transext}, extend $\der$ to a $T$-derivation on $L$ such that $\der b_i = \der b_{i+1}$ for $i<r$.

We claim that $L$ has small derivation. Put $b\coloneqq b_0$, so $L = K\llangle b \rrangle$. If $a$ is $\Td$-algebraic over $\k$, then let $r$ be as above, and if $a$ is $\Td$-transcendental over $\k$, then let $r$ be arbitrary. Let $F\colon K^r\to K$ be an $\cL(K)$-definable function with $F(\jet^{r-1}b)\prec 1$. We need to verify that $F(\jet^{r-1}b)' \prec 1$. By~\cite[Lemma 2.12]{FK21}, there is some $\cL(K)$-definable function $F^{[\der]}\colon K^r\to K$ such that 
\[
F(\jet^{r-1}b)'\ =\ F^{[\der]}(\jet^{r-1}b) + \frac{\partial F}{\partial Y_0}(\jet^{r-1}b)b'+\cdots+\frac{\partial F}{\partial Y_{r-1}}(\jet^{r-1}b)b^{(r)}.
\] 
For each $i<r$, we note that $b^{(i)} \asymp 1$ and that $b^{(i)} \not\sim g$ for any $g \in K\langle (b^{(j)})_{j<r,j\neq i}\rangle$, so Fact~\ref{fact:smallderivsim} gives us that $\frac{\partial F}{\partial Y_i}(\jet^{r-1}b)\preceq F(\jet^{r-1}b)\prec 1$. Since $b^{(i+1)} \preceq 1$, we conclude that $\frac{\partial F}{\partial Y_i}(\jet^{r-1}b)b^{(i+1)}\prec 1$, so it remains to show that $F^{[\der]}(\jet^{r-1}b)\prec 1$. Suppose that this is not the case, and let $U \subseteq K^r$ be the $\LO(K)$-definable set
\[
U\ \coloneqq\ \Big\{u \in \cO^r: F^{[\der]}(u) \succeq 1\text{ and }F(u),\frac{\partial F}{\partial Y_0}(u),\ldots,\frac{\partial F}{\partial Y_{r-1}}(u)\prec 1\Big\}.
\]
Note that $U$ is nonempty, since $L$ is an elementary $\TO$-extension of $K$ and $\jet^{r-1}(b) \in U^L$. Take a tuple $u = (u_0,\ldots,u_{r-1}) \in U$. Then $u_0',\ldots,u_{r-1}' \preceq 1$ since $K$ has small derivation. But then 
\[
F(u)' \ =\ F^{[\der]}(u) + \frac{\partial F}{\partial Y_0}(u)u_0'+\cdots+\frac{\partial F}{\partial Y_{r-1}}(u)u_{r-1}'\ \asymp\ F^{[\der]}(u)\ \succeq\ 1,
\]
contradicting that $K$ has small derivation. This concludes the proof of the claim.

With the claim taken care of, we see that $L$ is a $\TdO$-extension of $K$ and that the differential residue field $\k_L$ is $\Ld(\k)$-isomorphic to $E$. By the Wilkie inequality, we have that $\Gamma_L = \Gamma$. Let $K^*$ be a $\Td$-henselian $\TdO$-extension of $K$ and let $\imath\colon \k_L\to  \res(K^*)$ be an $\Ld(\k)$-embedding. The argument that $\imath$ lifts to an $\LdO(K)$-embedding $L \to K^*$ is very similar to the proof of part \ref{prop:resextiii} of Proposition~\ref{prop:resext}. If $a$ is $\Td$-transcendental over $\k$, then for any lift $b^*$ of $\imath(a)$, the unique $\LdO(K)$-embedding $L \to K^*$ that sends $b$ to $b^*$ is a lift of $\imath$. Suppose that $a$ is $\Td$-algebraic over $\k$. Then by construction $b^{(r)} = F(\jet^{r-1}b)$ for some minimal $r$ and some $\cL(K_0)$-definable function $F$, where $K_0\subseteq K$ is a lift of $\k$. As $K^*$ is $\Td$-henselian, we find $b^* \in K^*$ that lifts $\imath(a)$ and satisfies the identity $(b^*)^{(r)} = F(\jet^{r-1}b^*)$. Then the unique $\LdO(K)$-embedding $L \to K^*$ that sends $b$ to $b^*$ is a lift of $\imath$. 
\end{proof}

\begin{corollary}\label{cor:genericextension}
Let $K\models\TdOmon$. Then $K$ has a $\TdO$-extension $L\models \TdOmonG$ that embeds over $K$ into any $|K|^+$-saturated $K^*\models\TdOmonG$ extending~$K$.
\end{corollary}
\begin{proof}
We construct $L$ in three steps. First, we build a nontrivially valued  $\TdOmon$-model $K_1$ extending $K$. If $K$ itself is nontrivially valued, then we take $K_1\coloneqq K$. If $K$ is trivially valued, then let $K_1\coloneqq K\langle a\rangle$ where $a>K$. In this case, we equip $K_1$ with the convex hull of $K$ as its $T$-convex valuation ring and, using Fact~\ref{fact:transext}, we uniquely extend the $T$-derivation on $K$ to a $T$-derivation on $K_1$ such that $\der a = 0$. Then $\res(K_1) = \k$, and it is fairly easy to check that $K_1$ is monotone, arguing as in Proposition~\ref{prop:fieldbuilding}. In fact, one can build $K_1$ explicitly using Proposition~\ref{prop:fieldbuilding} by taking $K$ in place of $\k$, taking $\Lambda$ in place of $\Gamma$, and taking $c\colon \Lambda \to K$ to be the zero map. Next we use Fact~\ref{fact:tdextend} to take a $\TdG$-model $E$ extending $\k$ with $|E| = |\k|$. Take a $\TdO$-extension $K_2$ of $K_1$ as in  Proposition~\ref{prop:residueext}, so $\Gamma_{K_2}= \Gamma_{K_1}$ and $\res(K_2)$ is $\Ld(\k)$-isomorphic to $E$. Finally, we use Fact~\ref{fact:sphcomp} to take a spherically complete immediate $\TdO$-extension $L$ of $K_2$. Then $\Gamma_L= \Gamma_{K_1}$, so $L$ is monotone by~\cite[Corollary 6.3.6]{ADH17}. Moreover, $L$ is $\Td$-henselian by Corollary~\ref{cor:sphcompTdh}, so it has many constants by Corollary~\ref{cor:7.1.11} (it follows easily from the axioms of $\TdG$ that $(\k_L^\times)^\dagger = \k_L$).

Now let $K^*\models \TdOmonG$ be a $|K|^+$-saturated $\TdO$-extension of $K$. In the case that $K_1 = K\langle a \rangle$ where $a>K$ and $a' = 0$, we take $a^* \in C_{K^*}$ with $a^*> \cO_{K^*}$ and we let $\imath_1\colon K_1\to K^*$ be the $\cL(K)$-embedding that sends $a$ to $a^*$.  By Fact~\ref{fact:transext} and~\cite[Corollary 3.7]{DL95}, $\imath_1$ is even an $\LdO$-embedding. If $K_1 = K$, then we just take $\imath_1$ to be the identity on $K$. Next, as $\res(K^*)$ is a $|\k|^+$-saturated model of $\TdG$, the inclusion $\k \subseteq \res(K^*)$ extends to an $\Ld$-embedding $\res(K_2) \to \res(K^*)$ by Fact~\ref{fact:tdextend}. By Proposition~\ref{prop:residueext}, we may lift this to an $\LdO(K)$-embedding $\imath_2\colon K_2\to K^*$ that extends $\imath_1$. Finally, note that $K^*$ is $|\Gamma_{K_2}|^+$-saturated, since $|\Gamma_{K_2}| = \max\{|\Gamma_K|,|\Lambda|\}\leq |K|$, so $\imath_2$ extends to an $\LdO(K)$-embedding $\jmath\colon L\to K^*$ by Corollary~\ref{cor:sphcompembed}.
\end{proof}

\begin{theorem}\label{thm:modelcompletion}
$\TdOmonG$ is the model completion of $\TdOmon$. Consequently, $\TdOmonG$ is complete, and if $T$ has quantifier elimination and a universal axiomatization, then $\TdOmonG$ has quantifier elimination. 
\end{theorem}
\begin{proof}
Let $E\models \TdOmon$, and let $K$ and $K^*$ be models of $\TdOmonG$ extending $E$. In light of  Corollary~\ref{cor:genericextension}, it is enough to show that $K$ and $K^*$ are $\LdO(E)$-elementarily equivalent. We may assume that both $K$ and $K^*$ are $|E|^+$-saturated. Let $L\models \TdOmonG$ be the $\TdO$-extension of $E$ constructed in Corollary~\ref{cor:genericextension}. Then $L$ admits an $\LdO(E)$-embedding into both $K$ and $K^*$, so by replacing $E$ with $L$, we may assume that $E \models \TdOmonG$. It remains to note that $\TdOmonG$ is model complete by Corollary~\ref{cor:elementarysubstructure2}, since $\TdG$ and the theory of nontrivial ordered $\Lambda$-vector spaces are both model complete (in invoking Corollary~\ref{cor:elementarysubstructure2}, we take $c$ to be the zero map).

Let $\PP$ be the prime model of $T$. Then $\PP$, equipped with the $T$-convex valuation ring $\cO_\PP = \PP$ and the trivial $T$-derivation, is a model of $\TdOmon$ that embeds into every model of $\TdOmonG$, so $\TdOmonG$ is complete. %reworded to avoid ``prime model''
% Let $\cP$ be the prime model of $T$. Then $\cP$, equipped with the $T$-convex valuation ring $\cO_\cP = \cP$ and the trivial $T$-derivation, is the prime model of $\TdO$. In particular, $\cP$ is the prime model of $\TdOmon$, so we conclude that $\TdOmonG$ is complete.
Suppose that $T$ has quantifier elimination and a universal axiomatization. Then $\TdOmon$ has a universal axiomatization, so $\TdOmonG$ has quantifier elimination.
\end{proof}

\begin{corollary}\label{cor:formulaform}
For every $\LdO$-formula $\varphi$ there is some $r$ and some $\LO$-formula $\tilde{\varphi}$ such that
\[
\TdOmonG\ \vdash\ \forall x\big( \varphi(x) \leftrightarrow \tilde{\varphi}\big(\jet^r(x)\big)\big).
\]
\end{corollary}
\begin{proof}
Argue as in~\cite[Lemma 4.11]{FK21}, using that every $\LO$-term is an $\cL$-term.
\end{proof}

\begin{corollary}
$\TdOmonG$ is distal.
\end{corollary}
\begin{proof}
Argue as in~\cite[Theorem 4.15]{FK21}, using Corollary~\ref{cor:formulaform} and the fact that $\TO$ is distal. 
%By extending $\cL$ by function symbols for $\cL(\emptyset)$-definable functions, we may assume that $T$ has quantifier elimination and a universal axiomatization. Then $\TdOmonG$ has quantifier elimination in the language $\LdO$ by Theorem~\ref{thm:modelcompletion}. We claim that for every unary $\Ld$-term $t(x)$, there is some $n$ and an $n$-ary $\cL$-term $s$ such that $\TdOmonG\models \forall x(t(x) = s(x,\der x,\ldots,\der^n x))$. We prove this by induction on $e(t)$, the number of times in $t$ that $\der$ is applied to a term that is not of the form $\der^kx$. If $e(t)>0$. Then $t$ is of the form $t_0(x,\der s(x,\der x,\ldots,\der^n x))$ for some $\Ld$-term $t_0$ and some $\cL$-term $s$. Let $\cD$ be an $\cL(\emptyset)$-definable $\cC^1$-cell decomposition for $s$, so for each $D \in \cD$, there is an $\cL(\emptyset)$-definable $\cC^1$-function $s_D$ defined in an open neighborhood of $D$ with $s_D(y) = F(y)$ for all $y \in D$. Then the map $f(y)$ which agrees with $\nabla s_D(y)$ whenever $y \in D$ is $\cL(\emptyset)$-definable, and $\Td\models \der s(y) = f(y)\cdot \der y$ for all $y =(y_0,\ldots,y_n)$. Set 
%\[t^*(x)\ \coloneqq \ t_0\big(x,f(x,\ldots,\der^n x)\cdot (\der x,\ldots,\der^{n+1}x)\big).\]
%Then $e(t^*) < e(t)$, concluding the proof of the claim. It follows from the claim,  from~\cite[Proposition 7.1]{ACGZ22} (with $\TO$ in place of $T$ and with $\mathfrak{F} = \{\der\}$), and from the distality of $\TO$ that $\TdOmonG$ is distal. 
\end{proof}

\begin{corollary}
Let $K\models\TdOmonG$ and let $C$ be the constant field of $K$. For every $\LdO(K)$-definable set $A \subseteq C^n$, there is an $\LO(K)$-definable set $B\subseteq K^n$ such that $A = B\cap C^n$. In particular, the induced structure on $C$ is weakly o-minimal.
\end{corollary}
\begin{proof}
Argue as in~\cite[Lemma 4.23]{FK21}, using Corollary~\ref{cor:formulaform}.
\end{proof}

%------------------------------------------------------------------------------%
\section*{Acknowledgements}
%------------------------------------------------------------------------------%
Research for this paper was conducted in part at the Fields Institute for Research in Mathematical Sciences.
This material is based upon work supported by the National Science Foundation under Grants DMS-2103240 (Kaplan) and DMS-2154086 (Pynn-Coates).
This research was funded in whole or in part by the Austrian Science Fund (FWF) 10.55776/ESP450 (Pynn-Coates). For open access purposes, the authors have applied a CC BY public copyright licence to any author accepted manuscript version arising from this submission. We thank the referee for their helpful feedback, which in particular stimulated improvements to Section~\ref{sec:RCF}.

%------------------------------------------------------------------------------%

%------------------------------------------------------------------------------%
\end{document}